\documentclass[review,onefignum,onetabnum]{siamart190516}

\usepackage{amssymb}
\usepackage{algpseudocode}
\usepackage{bm}
\usepackage{booktabs}
\usepackage{capt-of}
\usepackage{color, colortbl}
\usepackage{mathtools}
\usepackage{multirow}
\usepackage{siunitx}
\usepackage{subcaption}
\usepackage[sort,compress]{cite}

\usepackage{lipsum}
\usepackage{amsfonts}
\usepackage{graphicx}
\usepackage{epstopdf}
\ifpdf
  \DeclareGraphicsExtensions{.eps,.pdf,.png,.jpg}
\else
  \DeclareGraphicsExtensions{.eps}
\fi

\usepackage{enumitem}
\setlist[enumerate]{leftmargin=.5in}
\setlist[itemize]{leftmargin=.5in}

\newsiamremark{remark}{Remark}
\newsiamremark{hypothesis}{Hypothesis}
\crefname{hypothesis}{Hypothesis}{Hypotheses}
\newsiamthm{claim}{Claim}
\newsiamremark{conjecture}{Conjecture}
\newsiamremark{partition}{Partitioning}

\Crefname{subsection}{Section}{Sections}

\headers{DD Preconditioner for IE}{C. Chen, and G. Biros}

\title{Overlapping Domain Decomposition Preconditioner for Integral Equations 
\thanks{Submitted to the editors on \today.
\funding{
This material is based upon work supported by the U.S. Department of Energy, Office of Science, Office of Advanced Scientific Computing Research, Applied Mathematics program under Award Number DE-SC0019393, by the U.S. Department of Energy, National Nuclear Security Administration Award Number DE-NA0003969; by NIH award 5R01NS042645-11A1; and  by the Portugal Foundation for Science and Technology (FCT) and the UT Austin-Portugal program.  Any opinions, findings, and conclusions or recommendations expressed herein are those of the authors and do not necessarily reflect the views of the   DOE, NSF, and NSF. Computing time on the Texas Advanced Computing Centers Frontera system was provided by an allocation from TACC and the NSF.
}}}

\author{Chao Chen\thanks{University of Texas at Austin, United States (\email{chenchao.nk@gmail.com},\email{biros@oden.utexas.edu}).}
\and George Biros\footnotemark[2]}

\usepackage{amsopn}

\ifpdf
\hypersetup{
  pdftitle={Domain-Decomposition Preconditioner},
}
\fi

\newcommand{\bigO}{ {\mathcal{O}} }

\definecolor{Gray}{gray}{0.9}
\newcolumntype{g}{>{\columncolor{Gray}}c}

\newcolumntype{H}{>{\setbox0=\hbox\bgroup}c<{\egroup}@{}}

\graphicspath{{figs/}}

\nolinenumbers

\begin{document}
\maketitle

\begin{abstract}
The discretization of certain integral equations, e.g., the first-kind Fredholm equation of Laplace's equation, leads to symmetric positive-definite linear systems, where the coefficient matrix is dense and often ill-conditioned. We introduce a new preconditioner based on a novel overlapping domain decomposition that can be combined efficiently with existing fast direct solvers. Empirically, we observe that the condition number of the preconditioned system is $\bigO(1)$, independent of the problem size. Our domain decomposition is designed so that we can construct approximate factorizations of subproblems efficiently. In particular, we apply the recursive skeletonization algorithm to subproblems associated with every subdomain. We present numerical results on problem sizes up to $16\,384^2$ in 2D and $256^3$ in 3D, which were solved in less than 16 hours and three hours, respectively, on an Intel Xeon Platinum 8280M.
\end{abstract}

% REQUIRED
\begin{keywords}
  Integral Equations, Overlapping Domain Decomposition, Symmetric Positive-definite Preconditioners, Fast Direct Solvers, Hierarchical Matrices
\end{keywords}

% REQUIRED
\begin{AMS}
  35R09, % 	Integral partial differential equations
  65F08, %	Numerical analysis, Preconditioners for iterative methods
  65N55, % 	Multigrid methods; domain decomposition for boundary value problems involving PDEs
\end{AMS}

\section{Introduction}

Consider the following integral equation (IE)
\begin{equation} \label{e:int}
a(x) u(x) + b(x) \int_{\Omega} K(x-y) c(y) u(y) dy = f(x), \quad x\in\Omega \subset \mathbb{R}^d,
\end{equation}
where $d=2$ or 3, $a(x), b(x), c(x)$ and $f(x)$ are given functions, $K(r) \equiv K(x-y)$ is the Green's function of an  elliptic operator,  $\Omega$ is a finite simply-connected domain, and $u(x)$ is the unknown. Upon discretization of \cref{e:int} through either Galerkin or Nystr\"{o}m methods, we obtain 
\begin{equation} \label{e:axb}
A \, u = f,
\end{equation}
where $u$ and $f$ are the discrete analogues of $u(x)$ and $f(x)$, respectively, and $A \in \mathbb{R}^{N\times N}$ is a \emph{dense} matrix. See an example of the discretization in \Cref{s:setup}.

We further assume the discretized integral operator, $A$ in \cref{e:axb}, is symmetric positive-definite (SPD). This typically occurs when \cref{e:int} is a first-kind Fredholm integral equation for Laplace's equation or the Stokes equation, which has applications in magnetostatics, electrostatics and fluid dynamics. More theory and discussion on symmetric formulations of integral equations (including hypersingular integrals) can be found in~\cite{mund1998two}. One challenge for solving \cref{e:axb} is that $A$ usually has a large condition number~\cite{xing2018preserving,xing2021efficient}, and this paper is concerned with solving \cref{e:axb} iteratively using domain decomposition preconditioners.

\subsection{Previous work}

\begin{figure}
         \centering
         \includegraphics[width=0.45\textwidth]{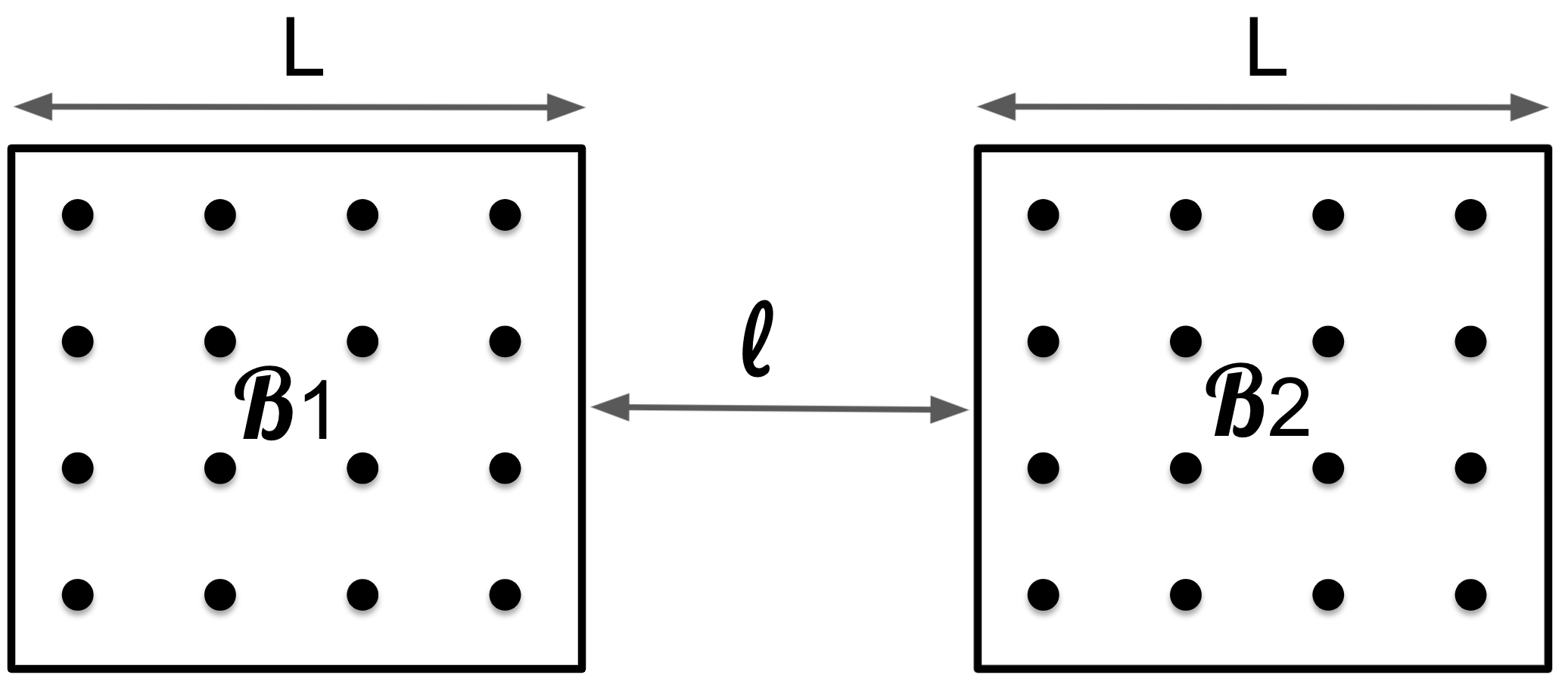}
     \caption{\em Given two sets of points $\mathcal{B}_1$ and $\mathcal{B}_2$ inside two boxes, respectively. If the two boxes are adjacent (weak admissibility), i.e., $\ell = 0$, then the numerical rank of the off-diagonal block $A(\mathcal{B}_1, \mathcal{B}_2)$ is $\bigO(L)$. On the other hand, if the two boxes are \emph{well-separated} (strong admissibility), i.e., $\ell = \bigO(L)$, then  the numerical rank of the off-diagonal block $A(\mathcal{B}_1, \mathcal{B}_2)$ is $\bigO(1)$.}
     \label{f:interaction}
\end{figure}

Since $A$ is a dense matrix, classical direct methods such as Gaussian elimination require $\bigO(N^3)$ operations  and $\bigO(N^2)$ storage. Such costs limit the application of these methods to solving only small problems. 

Although $A$ is a dense matrix, in many applications it is \emph{data sparse} in the sense that the singular values of certain off-diagonal blocks in $A$ decay exponentially fast. This is typically the case for matrices related to integral equation formulations of problems with non-oscillatory coefficients and kernels. The numerical low-rank property is exploited in the fast multipole method  (FMM)~\cite{greengard1987fast,ying2004kernel,martinsson2007accelerated,fong2009black,malhotra2015pvfmm,chen2018fast}, which requires only $\bigO(N)$ operations and storage for applying $A$ to a vector. 
%The FMM was initially developed for Green's functions associated with elliptic operators~\cite{greengard1987fast,ying2004kernel,martinsson2007accelerated,malhotra2015pvfmm} and was extended to general kernel functions~\cite{fong2009black,chen2018fast,wang2021pbbfmm3d}. 
Such a fast matrix-vector multiplication can be coupled with iterative methods such as the conjugate gradient method~\cite{shewchuk1994introduction} to solve \cref{e:axb}. However, the number of iterations can be large when the condition number of $A$ is large. {Such a situation  occurs when \cref{e:int} is a first-kind Fredholm integral equation, i.e., $a(x) = 0$; when $b(x)$ or $c(x)$ exhibits high contrast; or when the problem domain $\Omega$ has a complex geometry.} %~\cite{quaife2015preconditioners}. 

% explain FMM: near-field and far-field decomposition; far-field leads to constant-rank approximation
%Iterative methods  such as the conjugate gradient method~\cite{shewchuk1994introduction} provide an efficient alternative. These methods are based on the matrix-vector multiplication with $A$, which requires only $\bigO(N)$ computation and storage using the fast multipole method (FMM)~\cite{greengard1987fast,ying2004kernel,martinsson2007accelerated,malhotra2015pvfmm,chen2018fast,wang2021pbbfmm3d}.
%, so the cost of every iteration is the optimal $\bigO(N)$ complexity. 

%To accelerate the convergence of an iterative method for solving \cref{e:axb}, additive Schwarz preconditioners have been investigated. 
%A multigrid preconditioner was proposed for the second kind Fredholm equation of Laplace's equation in 2D~\cite{quaife2015preconditioners}, but no timing results were reported. 

To accelerate the convergence of an iterative method for solving \cref{e:axb}, the block Jacobi preconditioner is arguably the simplest preconditioner but may not be particularly effective. It was observed empirically that the single-level additive Schwarz preconditioner with overlapping blocks  led to better convergence for some realistic engineering applications~\cite{nabors1992fast,nabors1994preconditioned}. Previous work on two-level additive Schwarz preconditioners mainly focused on analyzing the condition number of the preconditioned linear system.
In~\cite{tran2000overlapping} and~\cite{mund1998two}, the authors introduced  two-level  additive Schwarz preconditioners to solve first-kind boundary integral equations for the Laplace's equation on a curve in $\mathbb{R}^2$ and on an open surface in $\mathbb{R}^3$, respectively. 
In~\cite{tran2000overlapping}, overlapping domain decomposition was used, and it was proved that the condition number of the preconditioned linear system is bounded only if the coarse mesh size is proportional to the size of overlap.
In~\cite{mund1998two}, non-overlapping domain decomposition was used, and it was proved that the condition number of the preconditioned linear system still depends on the ratio between the sizes of the coarse mesh and the fine mesh.
%the authors considered weakly singular integral equations of the first kind for the Laplace's equation  on open surface pieces in $\mathbb{R}^3$. They introduced a two-level  additive Schwarz preconditioner based on non-overlapping domain decomposition. But the condition number of the preconditioned linear system still depends on the ratio between the sizes of the coarse mesh and the fine mesh.
%While these work mainly focused on analyzing the condition number of the preconditioned linear system, we focus on efficient numerical solution of \cref{e:axb} in this paper.
 Recently, the authors of~\cite{borges2019domain} introduced  a two-level additive Schwarz preconditioner based on overlapping domain decomposition, which showed great potential in solving {indefinite} dense linear systems from the discretization of the Lippmann-Schwinger equation. For a more detailed discussion on existing preconditioners for solving discretized integral equations, we refer interested readers to~\cite{quaife2015preconditioners} and the references therein.

 %and showed that the condition number of the preconditioned linear system depends on $H/h$, where $H$ denotes the size of the subdomains and $h$ is the mesh size. 
%In~\cite{tran2000overlapping}, the authors introduced a two-level  additive Schwarz preconditioner based on a non-overlapping domain decomposition strategy. 
%and showed that the condition number of the preconditioned linear system is $\bigO(1+\log^2(H/\delta)$, where $\delta$ is the size of overlap. In~\cite{hsiao2000domain}, the authors analyzed employing various  boundary element methods in domain decomposition methods.

Fast direct solvers (FDS's) are a class of methods that construct approximate factorizations with a tunable accuracy. When the accuracy is low, they can be used as preconditioners; otherwise with a high accuracy, they behave similarly as a direct method that is suitable for solving multiple right-hand sides. The basic idea of an FDS is to compress certain off-diagonal blocks in $A$ to a prescribed accuracy because they are numerically low rank. This class of methods can be further divided into two groups based on which off-diagonal blocks are compressed or the admissibility criteria; see \cref{f:interaction}.

The first group employs the so-called weak admissibility and compresses off-diagonal blocks corresponding to two spatially adjacent regions. The resulting numerical ranks typically increase proportionally to the perimeter or the surface area of the regions in 2D or 3D, respectively. Consequently, the construction time of these methods typically scale as $\bigO(N^{3/2})$ and   $\bigO(N^{2})$ in 2D and 3D, respectively~\cite{martinsson2005fast,chandrasekaran2006fast,chandrasekaran2007fast,greengard2009fast,xia2010fast,ho2012fast,gillman2012direct,aminfar2016fast,xia2020robust}. Assuming the same rank behavior on the Schur complement, this type of methods can be further accelerated to attain quasilinear complexity~\cite{corona2015n,ho2016hierarchical,xia2021multi}. But existing numerical results seem to indicate that  reaching the asymptotic regime requires a really large problem size in 3D.

The second group employs the so-called strong admissibility and compresses off-diagonal blocks corresponding to two sufficiently distant regions~\cite{ambikasaran2014inverse,coulier2017inverse,minden2017recursive,takahashi2020parallelization,sushnikova2022fmm}. The resulting numerical ranks are constant regardless of the sizes of the regions according to standard fast multipole estimates~\cite{greengard1987fast,greengard1996new}. Consequently, these methods can achieve quasilinear complexity assuming the constant-rank estimates also hold for the Schur complement, which appears be true in practice but lacks firm theoretical support. In practice, the constants in the asymptotic scalings tend to be quite large, especially in 3D, because off-diagonal blocks corresponding to adjacent regions are treated exactly without compression.

\subsection{Contributions}

%While many physical problems can be formulated as either a PDE or an IE, the connection between the differential operator and the corresponding integral operator is quite interesting.  For Laplace's operator, high-frequency modes are associated with large eigenvalues, and low-frequency modes are associated with small eigenvalues, a well-known fact exploited by multi-level solvers such as the multigrid method. Here, the observation for $A$ in \cref{e:axb} is the opposite: high-frequency modes are associated with small eigenvalues while low-frequency modes are associated with large eigenvalues. 

\begin{figure}
     \centering
     \hfill
     \begin{subfigure}[b]{0.22\textwidth}
         \centering
         \includegraphics[width=\textwidth]{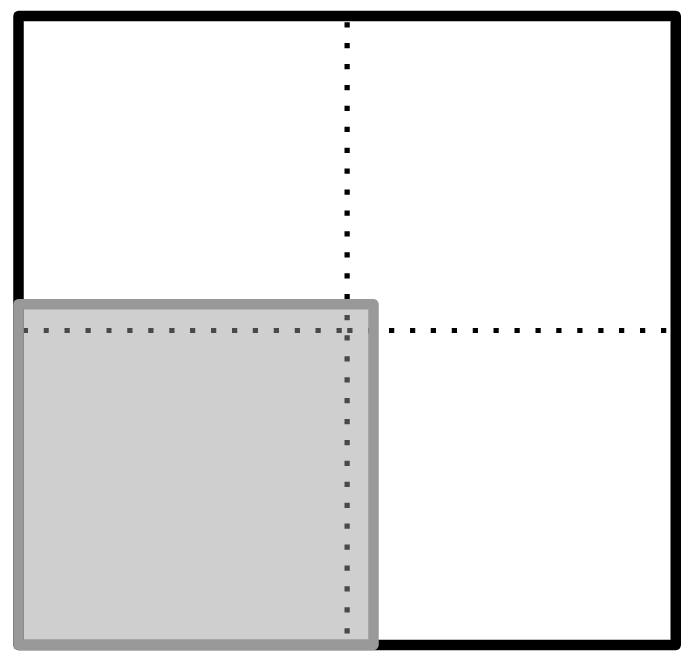}
         \caption{}
         \label{sf:jac_4}
         %\caption{Non-overlapping partitioning}
     \end{subfigure}
     \hfill
     \begin{subfigure}[b]{0.22\textwidth}
         \centering
         \includegraphics[width=\textwidth]{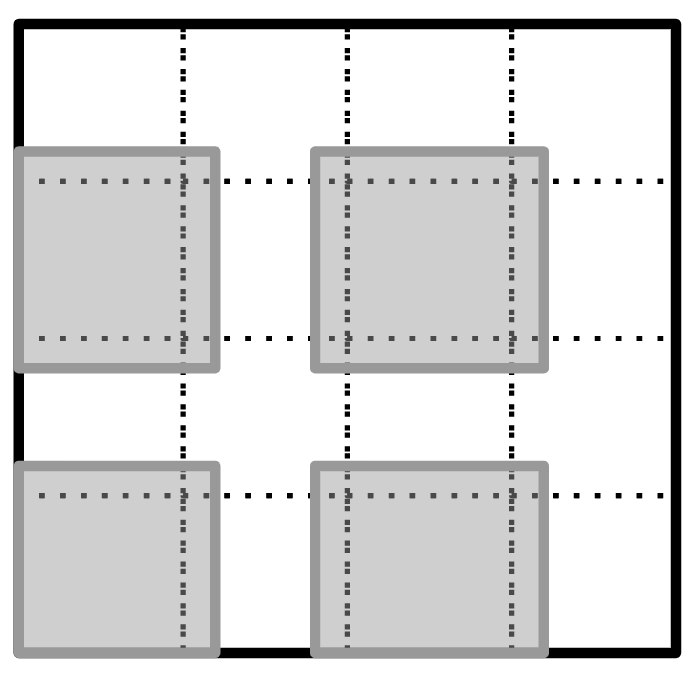}
         \caption{}
         \label{sf:cbd_4}
         %\caption{An ex}
     \end{subfigure}
     \hfill % \\ \hfill
     \begin{subfigure}[b]{0.22\textwidth}
         \centering
         \includegraphics[width=\textwidth]{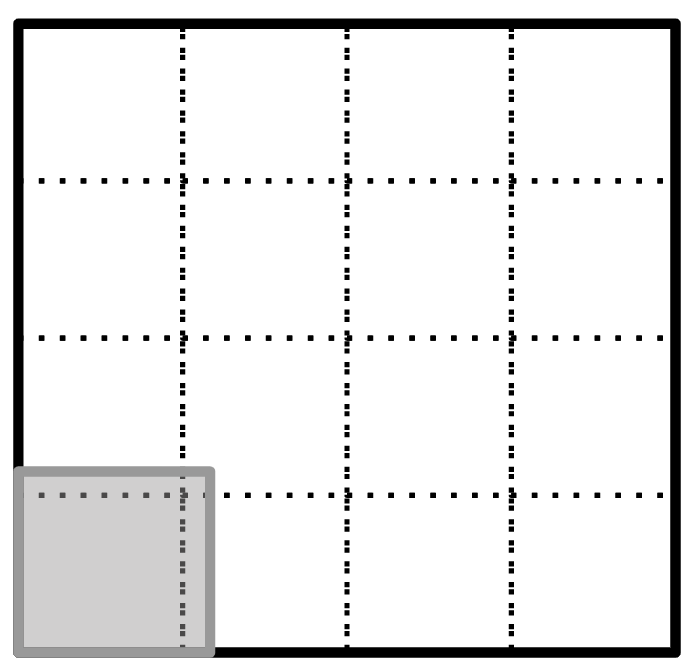}
         \caption{}
         \label{sf:jac_16}
         %\caption{Coloring of a partitioning of the unit square.}
     \end{subfigure}
     \hfill
     \begin{subfigure}[b]{0.22\textwidth}
         \centering
         \includegraphics[width=\textwidth]{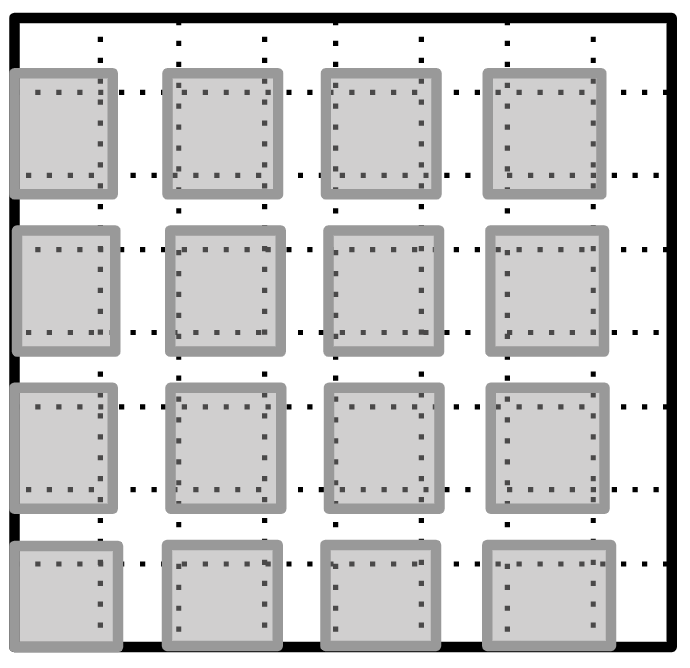}
         \caption{}
         \label{sf:cbd_16}
         %\caption{Subdomain $\Omega_1$ (gray)  in the CBD scheme.}
     \end{subfigure}
     \hfill
     \caption{\em A subdomain (grey) in the Schwarz or the CBD preconditioner. Dashed lines indicate the partitioning $\Omega = \bigcup_{i=1}^M  \mathcal{P}_i$, and every box in grey stands for an extended partition $\tilde{\mathcal{P}}_i$. (a) Schwarz preconditioner with four subdomains. (b) CBD preconditioner with four subdomains, where every subdomain has four separated regions. (c) Schwarz preconditioner with 16 subdomains. (d) CBD preconditioner with \emph{four} subdomains, where every subdomain has 16 separated regions.}
     \label{f:}
\end{figure}

%Two main contributions of this work include an analysis of the Schwarz preconditioner and the introduction of an enhanced preconditioner for solving \cref{e:axb} iteratively. Both preconditioners are based on overlapping domain decomposition. 

%The first contribution is that we apply and analyze the single-level additive-Schwarz preconditioner (Schwarz preconditioner hereafter) for solving \cref{e:axb} with the preconditioned conjugate gradient (PCG) method. The Schwarz preconditioner was originally developed for solving sparse linear systems from the discretization of partial differential equations (PDEs)~\cite{chan1994domain,dolean2015introduction}. It was recently applied to solving \emph{indefinite} dense linear systems from the discretization of the Lippmann-Schwinger equation~\cite{borges2019domain}. Here, our focus is solving SPD dense linear systems. 

We introduce a new preconditioner based on overlapping domain decomposition to solve \cref{e:axb} iteratively with the preconditioned conjugate gradient (PCG) method. The key feature is that the number of  iterations required is $\bigO(1)$, independent of the problem size. Our method is based on an empirically observation of the single-level additive-Schwarz preconditioner (Schwarz preconditioner $T^{-1}_\text{Schwarz}$ hereafter): the minimum eigenvalue of the preconditioned matrix is lower bounded away from zero. In other words, there exists a positive constant $C>0$ such that $\lambda_{\min}(T^{-1}_\text{Schwarz}A) \ge C$. Our preconditioner is associated with a new strategy for constructing an overlapping domain decomposition that has a fixed number of subdomains, and we apply an FDS as a subdomain solver to construct the preconditioner efficiently. In the following, we illustrate these ideas in more details.

{Let a partitioning (non-overlapping decomposition) of the problem domain be 
%$\Omega = \bigcup_{i=1}^M  \mathcal{P}_i$, where $\mathcal{P}_i \cap  \mathcal{P}_j = \emptyset $ if $i\not=j$, and $M=\bigO(N)$. 
\begin{equation}
\Omega = \bigcup_{i=1}^M  \mathcal{P}_i,  \quad M=\bigO(N),
\end{equation}
where $\mathcal{P}_i \cap  \mathcal{P}_j = \emptyset $ if $i\not=j$.}
An overlapping decomposition $\Omega = \bigcup_{i=1}^M  \tilde{\mathcal{P}}_i$ is obtained by extending $\mathcal{P}_i$ to overlap with its neighbors, so that $\mathcal{P}_i \subset \tilde{\mathcal{P}}_i$ and $\mathcal{P}_{j} \cap \tilde{\mathcal{P}}_i \not= \emptyset$ if $\mathcal{P}_i$ and $\mathcal{P}_{j}$ are spatially adjacent. See \cref{sf:jac_4,sf:jac_16} for two examples of the domain decomposition corresponding to $M=4$ and 16, respectively. In the Schwarz preconditioner, every $\tilde{\mathcal{P}}_i$ is a subdomain, and we solve $M$ corresponding subproblems, which are diagonal sub-blocks in $A$  (up to a permutation). The Schwarz preconditioner $T^{-1}_\text{Schwarz}$ can be viewed as a generalization of the block Jacobi preconditioner $T^{-1}_\text{Jacobi}$, where a subproblem is associated with $\mathcal{P}_i$ rather than $\tilde{\mathcal{P}}_i$. 
%In particular, we take every $\tilde{\mathcal{P}}_i$ as a subdomain and solve \cref{e:int} locally on every subdomain. The number of subdomains $D$ equals the number of partitions $M$. Notice that if we take every $\mathcal{P}_i$ as a subdomain, then we obtain the block Jacobi preconditioner (the Jacobi preconditioner hereafter). 
Numerical results show that $\lambda_{\min}(T^{-1}_\text{Schwarz}A) \ge C > 0$, whereas $\lambda_{\min}(T^{-1}_\text{Jacobi}A)$ decreases toward zero as $N\to\infty$.

%the spectrum of $T^{-1}_\text{Schwarz}A$ differs from that of $T^{-1}_\text{Jacobi}A$ in two ways: (1) the minimum eigenvalue is lower bounded away from zero, and (2) the eigenvalues are better clustered. %Numerical results also show that the number of PCG iterations with the Schwarz preconditioner is much smaller than that with the Jacobi preconditioner; see \cref{t:pcg}.

%While the cost of constructing and applying the Schwarz preconditioner is almost the sa as the Jacobi preconditioner, it leads to a much smaller condition number as well as a more clustered spectrum. Experimentally, we found that the minimum eigenvalue of the preconditioned matrix is bounded away from zero whereas it decreases towards zero with the Jacobi preconditioner. Therefore, the number of PCG iterations with the Schwarz preconditioner is much smaller than that with the Jacobi preconditioner; see numerical results in \Cref{result:schwarz}.

%Two kinds of preconditioners are introduced. The first one is the classical overlapping domain decomposition preconditioner that is normally applied for PDEs. It is obvious that the factorization time is $\bigO(N)$ as the number of subdomains $M$ increases as $\bigO(N)$. In other words, the local problem size per subdomain is constant, independent of the global problem size $N$. Unlike the Jacobi preconditioner, the eigenvalues of the Schwarz preconditioned matrices tend to cluster. As a result, Krylov solvers coupled with the Schwarz preconditioner are efficient, although the maximum eigenvalue increases towards infinity.

The new overlapping domain decomposition that we introduce is a coloring-based decomposition (CBD). 
The key feature of the CBD is that the number of subdomains is fixed, and every subdomain $\Omega_i$ consists of spatially distant $\tilde{\mathcal{P}}_i$'s. See \cref{sf:cbd_4,sf:cbd_16} for two examples of a subdomain in the CBD corresponding to $M=16$ and 64, respectively. 
To compute such a decomposition, we apply graph coloring to the adjacency graph of $\{\mathcal{P}_i\}_{i=1}^M$, so every $\mathcal{P}_i$ (and $\tilde{\mathcal{P}}_i$) is assigned a color $c_i$. Then, a subdomain $\Omega_i$ consists of $\tilde{\mathcal{P}}_k$'s of the same color, i.e., 
\[
{
\Omega_i 
\triangleq \{ \tilde{\mathcal{P}}_k: \text{ the color $c_k$ assigned to $\tilde{\mathcal{P}}_k$ is $i$} \}
= \bigcup_{c_k=i} \tilde{\mathcal{P}}_k. 
}
\]
%$\Omega_i = \bigcup_{c_k=i} \tilde{\mathcal{P}}_k$. 
Given a CBD, we define the corresponding preconditioner (CBD preconditioner hereafter) $T^{-1}_\text{CBD}$
following the same algebraic formulation of the  Schwarz preconditioner except that a subproblem $A_i$ is associated with a subdomain $\Omega_i$ instead of a single $\tilde{\mathcal{P}}_i$. As a result, the diagonal blocks in $A_i$ (with an appropriate permutation) are the subproblems in a related Schwarz preconditioner. In other words, the CBD preconditioner incorporates more information (off-diagonal blocks) from the global problem $A$. So we expect $T^{-1}_\text{CBD}$ to be a better preconditioner than $T^{-1}_\text{Schwarz}$ and thus $\lambda_{\min}(T^{-1}_\text{CBD}A) \ge \lambda_{\min}(T^{-1}_\text{Schwarz}A) \ge C > 0$. Under mild assumptions, the number of subdomains is upper bounded. For example, we have at most four subdomains in 2D and eight subdomains for a uniform partitioning in 3D. This implies that the maximum eigenvalue of the preconditioned matrix is upper bounded. Therefore, the condition number of $T^{-1}_\text{CBD}A$ is bounded; consequently, \cref{e:axb} can be solved in $\bigO(1)$ PCG iterations.

%employs a constant number of subdomains $D=\bigO(1)$, independent of the problem size $N$. In particular, we apply coloring to the graph of $\{ \mathcal{P}_i \}_{i=1}^M$, where an edge $( \mathcal{P}_i,  \mathcal{P}_j)$ exists if the two parts are physically adjacent; and then, we take all parts with the same color as one subdomain. So the number of subdomains $D$ equals the number of colors $N_s$. In 2D, the famous four-color theorem states  that we need only four colors. In 3D, the number of colors can be unbounded for an arbitrary graph, but is usually small for practical problems that have some regularity. For example, a uniform partition of a cubical domain leads to 8 colors. Once the subdomains are constructed, we solve \cref{e:int} locally on every subdomain in the preconditioner. We call this approach coloring-based decomposition (CBD), and the corresponding preconditioner is named the CBD preconditioner.

To accelerate the construction and reduce the memory footprint of the CBD preconditioner, we apply an FDS to $A_i$ associated with every subdomain $\Omega_i$. For demonstration, we employ the recursive skeletonization (RS) factorization~\cite{ho2016hierarchical} based on weak admissibility. RS has a relatively simple formulation, and its theoretical complexity is well understood (without any rank assumptions on the Schur complement). If the RS is applied to the global problem $A$ directly, the problem domain $\Omega$ is partitioned at multiple scales, and then the interaction between every pair of adjacent regions is compressed recursively down to the leaf (finest) level. At the leaf level, we have $\Omega = \bigcup_{i=1}^M  \mathcal{P}_i$, where the interaction rank is proportional to the perimeter or the surface area of every $\mathcal{P}_i$ in 2D or 3D, respectively. This type of rank estimation also holds for coarse partitioning. %At the root (coarsest) level, the number of \emph{skeletons}, or the interaction rank of the finer level, typically dominates the overall cost.

Consider a subdomain $\Omega_i$
%$\Omega_i = \bigcup_{c_k=i} \tilde{\mathcal{P}}_k$ 
in the CBD preconditioner, which consists of distant regions at the leaf level. When we apply RS to $A_i$ associated with $\Omega_i$, the separation distance among regions leads to significantly smaller number of skeletons at all scales, compared to the case when these regions are adjacent. In particular, the interaction rank between every pair of regions at the leaf level is provably constant (as in the FMM) because they satisfy the strong admissibility condition. As a result, the construction cost of the CBD preconditioner is much smaller  than that of applying the RS to $A$ directly.

Finally, we clarify that it is not as efficient as the CBD preconditioner if we fix the number of subdomains in the Schwarz preconditioner (problem size per subdomain increases as $\bigO(N)$) and employ the RS as the subdomain solver. The reason is basically the same as above that the interaction rank between a pair of adjacent regions is larger than that between a pair of distant regions. So the cost associated with every subdomain is larger than that in the CBD preconditioner. See more details in \cref{rmk:schwarz}.

\subsection{Outline and notations}

In \Cref{s:setup}, we introduce a model problem and overlapping domain decompositions that we use throughout this paper. In \Cref{s:schwarz}, we define the Schwarz preconditioner $T^{-1}_\text{Schwarz}$ and show the crucial observation that $\lambda_{\min}(T^{-1}_\text{Schwarz}A) \ge C > 0$. In \Cref{s:cbd}, we introduce the CBD preconditioner $T^{-1}_\text{CBD}$, show results on the spectrum of $T^{-1}_\text{CBD}A$, and discuss using the RS to construct $T^{-1}_\text{CBD}$ efficiently. 
%In \Cref{s:rs}, we show how to construct the second preconditioner rapidly using an FDS. 
In \Cref{s:result}, we present numerical results to demonstrate the performance of the CBD preconditioner, and \Cref{s:end} draws conclusions. 

%We use $\|\cdot\|$ to denote the 2-norm (matrix norm). In addition, 
We adopt the MATLAB notation of submatrices, e.g., $A(\mathcal{I}, :)$ and $A(:, \mathcal{I})$ denote the rows and columns in matrix $A$ corresponding to an index set $\mathcal{I}$, respectively.

\section{Model problem and domain decomposition} \label{s:setup}

In this section, we introduce a model problem and the discretized linear system to be solved.  We also present two strategies for constructing an overlapping domain decomposition
$
\Omega = \bigcup_{i=1}^D \Omega_i,
$
where $\Omega \subset \mathbb{R}^d$ is the problem domain ($d=2$ or 3), $\Omega_i$'s are overlapping subdomains, and $D$ is the number of subdomains.  

For illustration purpose, we assume $\Omega=[0,1]^d$, a uniform discretization grid of size $N=n^d$, and a uniform partitioning of the grid with $M=m^d$ partitions, where $n$ and $m$ are the number of grid points and partitions per dimension, respectively. So every partition has $N/M=(n/m)^d$ grid points. {In this paper, we always assume $M=\bigO(N)$ unless stated otherwise, so the number of points per partition is fixed to be $N/M = \bigO(1)$.}

In \Cref{ss:mp}, we introduce a model problem and the corresponding discretization.
In \Cref{ss:odd}, we introduce the overlapping domain decomposition for the Schwarz preconditioner, where the number of subdomains equals the number of partitions, i.e., $D=M$. In \Cref{ss:cbd}, we introduce the CBD strategy, where the number of subdomains $D=2^d$. These setups are used throughout this paper for numerical results.

\subsection{IE for Laplace's equation} \label{ss:mp}

We focus on a prototypical example of \cref{e:int}, the first-kind volume IE of Laplace's equation in 2D and 3D:
\begin{equation} \label{e:ie}
\int_{\Omega} K(x-y) \, u(y) \, dy = f(x), \quad x \in \Omega=[0,1]^d,
\end{equation}
where $K(r) \equiv K(x-y)$ is the fundamental solution of the free-space Laplace's equation, i.e.,
\begin{equation} \label{e:k}
K(r) = 
\left\{
\begin{array}{cc}
-\frac{1}{2\pi} \log(\|r\|), & d=2, \\
\frac{1}{4\pi \|r\|}, & d=3.
\end{array}
\right.
\end{equation}
Here $u(x)$ is a scalar, and the extension of our approach to cases where  $u(x)$ is a vector (e.g., in the Stokes equation) is straightforward.

For simplicity, we discretize \cref{e:ie} with a piecewise-constant collocation method over a uniform grid as follows.
%Suppose a regular grid has $N=n^2$ nodes. 
In 2D, let $h=1/n$ and $x_j = h (j_1-1/2, j_2-1/2)$, where $j$ is the index for an integer pair $(j_1,j_2)$ for $1\le j_1, j_2 \le n$. Using the trapezoidal rule and all $x_j$'s as the collocation points, \cref{e:ie} is discretized into the following linear system
\begin{equation} \label{e:dis}
\sum_{j} A_{ij} \, u_j = f_i
\end{equation}
where $u_j \approx u(x_j)$ is to be solved, $f_i = f(x_i)$ is given, and 
\begin{equation} \label{e:a_2d}
A_{ij} = 
\left\{
\begin{array}{cc}
h^2 \, K(x_i - x_j) ,  &i\not= j, \\
\int_{-h/2}^{h/2} \int_{-h/2}^{h/2} \, K(r) \, dx dy, & i=j,
\end{array}
\right.
\end{equation}
can be evaluated using \cref{e:k} and a proper singular quadrature.

In 3D, let $h=1/n$ and $x_j = h (j_1-1/2, j_2-1/2, j_3-1/2)$, where $j$ is the index for an integer pair $(j_1,j_2,j_3)$ for $1\le j_1, j_2, j_3 \le n$. The same procedure as above leads to 
\begin{equation} \label{e:a_3d}
A_{ij} = 
\left\{
\begin{array}{cc}
h^3 \, K(x_i - x_j) ,  &i\not= j, \\
\int_{-h/2}^{h/2} \int_{-h/2}^{h/2} \int_{-h/2}^{h/2} \, K(r) \, dx dy dz, & i=j.
\end{array}
\right.
\end{equation}

{
Empirically, the condition number of matrix $A$ increases as $\bigO({1}/{h^2})$, where $h$ is the mesh spacing. That is,
\begin{equation} 
\text{cond}(A) = 
\left\{
\begin{array}{cc}
\bigO(N), & d=2, \\
\bigO(N^{2/3}), & d=3.
\end{array}
\right.
\end{equation}
(The same behavior as the discretized Laplace operator.)
}

\subsection{Overlapping domain decomposition} \label{ss:odd}

\begin{figure}
     \centering
     \hfill
     \begin{subfigure}[b]{0.22\textwidth}
         \centering
         \includegraphics[width=\textwidth]{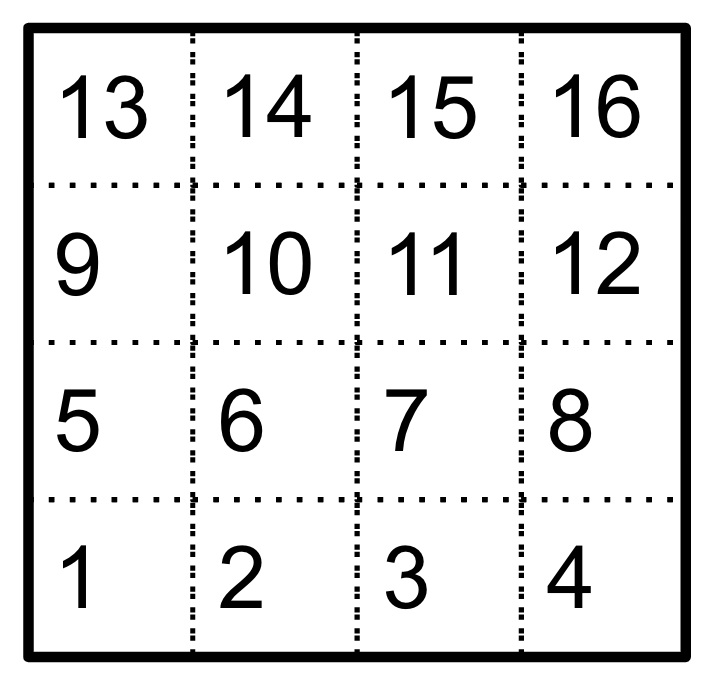}
         \caption{}
         \label{sf:partition}
         %\caption{Non-overlapping partitioning}
     \end{subfigure}
     \hfill
     \begin{subfigure}[b]{0.22\textwidth}
         \centering
         \includegraphics[width=\textwidth]{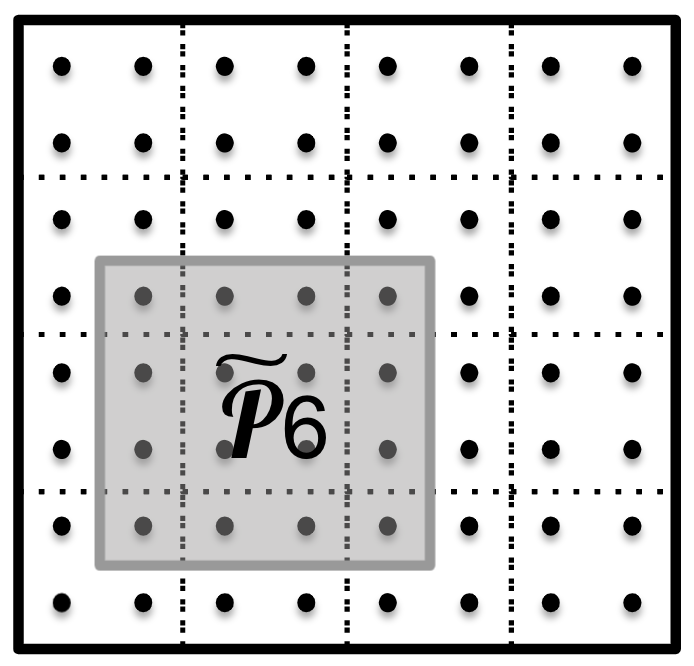}
         \caption{}
         \label{sf:extension}
         %\caption{An ex}
     \end{subfigure}
     \hfill
     \begin{subfigure}[b]{0.22\textwidth}
         \centering
         \includegraphics[width=\textwidth]{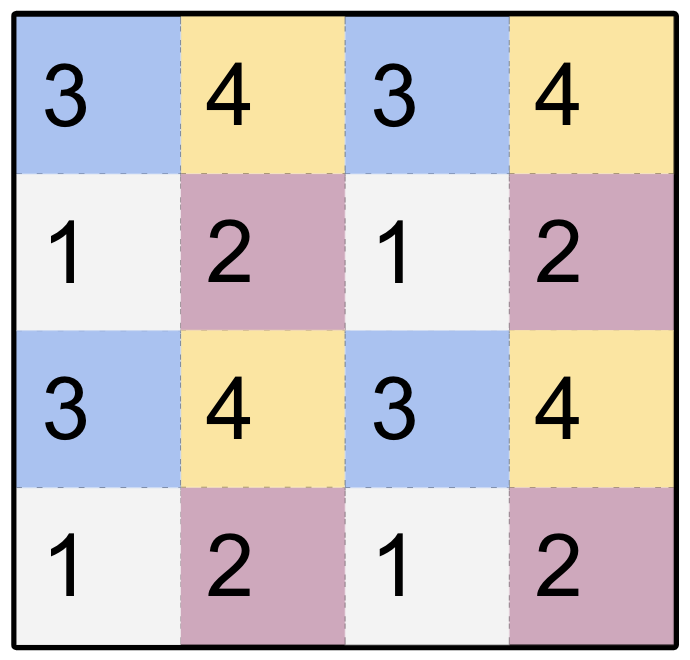}
         \caption{}
         \label{sf:coloring}
         %\caption{Coloring of a partitioning of the unit square.}
     \end{subfigure}
     \hfill
     \begin{subfigure}[b]{0.22\textwidth}
         \centering
         \includegraphics[width=\textwidth]{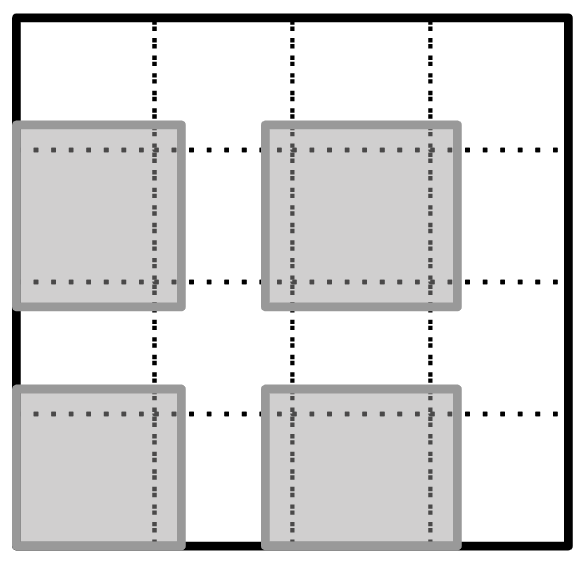}
         \caption{}
         \label{sf:subdomain}
         %\caption{Subdomain $\Omega_1$ (gray)  in the CBD scheme.}
     \end{subfigure}
     \hfill
     \caption{\em (a) A uniform partitioning of the unit square $\Omega = \bigcup_{i=1}^{16}  \mathcal{P}_i$. (b) A uniform grid with $8^2$ points and an extended partition $\tilde{\mathcal{P}}_6$ (shaded) containing $4^2$ points. (c) A four-coloring of the partitioning. (d) A subdomain $\Omega_1$ (shaded) with four spatially separated regions in CBD.}
     \label{f:part}
\end{figure}

Let a uniform partitioning or non-overlapping domain decomposition of the problem domain $\Omega=[0,1]^d$ be 
\begin{equation} \label{e:dd}
\Omega = \bigcup_{i=1}^M  \mathcal{P}_i,
\end{equation}
where $M=\bigO(N)$, and $\mathcal{P}_i \cap  \mathcal{P}_j = \emptyset \text{ if } i\not=j$; see \cref{sf:partition} for an example.
%where $M$ is the number of partitions. 
An overlapping domain decomposition
$
\Omega = \bigcup_{i=1}^M  \tilde{\mathcal{P}}_i,
$ 
follows by extending $\mathcal{P}_i$ to overlap with all of its spatially adjacent partitions $\mathcal{P}_j$, i.e., $\mathcal{P}_i \subset \tilde{\mathcal{P}}_i$ and $\mathcal{P}_j \cap \tilde{\mathcal{P}}_i \not= \emptyset$.  In particular, we extend every partition $\mathcal{P}_i$ to include one extra layer of grids in every direction (when possible) to form $\tilde{\mathcal{P}}_i$; see \cref{sf:extension} for an example. Except for those extended partitions near the boundary, every $\tilde{\mathcal{P}}_i$ has $(n/m+2)^d$ nodes given a uniform partitioning of a uniform grid.

Our first overlapping domain decomposition is simply taking $\Omega_i = \tilde{\mathcal{P}}_i$ as a subdomain. Hence, we get
\begin{equation} \label{e:first}
\Omega = \bigcup_{i=1}^M  \tilde{\mathcal{P}}_i = \bigcup_{i=1}^D \Omega_i,
\end{equation}
where $D=M=\bigO(N)$ is the number of subdomains. {(The right-hand side is the union of overlapping subsets of $\Omega$, which quals the entire set $\Omega$.)} In \Cref{s:schwarz}, \cref{e:first} and \cref{e:dd} are used to construct the Schwarz preconditioner and the block Jacobi preconditioner, respectively. In the block Jacobi preconditioner, $\mathcal{P}_i$ is treated as a subdomain.

% In general, $\tilde{\mathcal{P}}_i$ can have $(n/m+2\delta)^d$ nodes, where $\delta$ is a non-negative integer indicating the overlapping size.

%In this paper, we consider uniform partitioning of regular grids in 2D and 3D. Suppose a $d$-dimensional regular grid ($d=2$ or 3) has $N=n^d$ nodes, and it is partitioned into $M=m^d$ (non-overlapping) parts. Every part $\mathcal{P}_i$ has $N/M=(n/m)^d$ nodes. Every partition $\mathcal{P}_i$ is extended to include one extra layer of grids in each direction to form $\tilde{\mathcal{P}}_i$ with $(n/m+2)^d$ nodes ($\tilde{\mathcal{P}}_i$ near the boundary has fewer nodes).  

\subsection{Coloring-based decomposition} \label{ss:cbd}

Given the partitioning in \cref{e:dd}, we apply graph coloring to the adjacency graph of $\{ \mathcal{P}_i \}_{i=1}^M$, where an edge between $\mathcal{P}_i$ and $\mathcal{P}_j$ exists if the two partitions are spatially adjacent. As a result, every partition $\mathcal{P}_i$ (and its extension $\tilde{\mathcal{P}}_i$) is assigned a color $c_i$ in $\{ 1, 2, \ldots, N_c \}$, where $N_c$ is the number of colors; see \cref{sf:coloring} for an example. With the overlapping domain decomposition in \cref{e:first}, we form a subdomain $\Omega_i$ as
\begin{equation} \label{e:cbd}
{
\Omega_i 
= \{ \tilde{\mathcal{P}}_k: \text{ the color $c_k$ assigned to $\tilde{\mathcal{P}}_k$ is $i$} \}
= \bigcup_{c_k=i} \tilde{\mathcal{P}}_k, \quad i=1,2,\ldots,N_c;
}
\end{equation}
see \cref{sf:subdomain} for an example. It is easy to see
\begin{equation} \label{e:second}
\Omega = \bigcup_{i=1}^M  \tilde{\mathcal{P}}_i = \bigcup_{i=1}^{N_c} \left( \bigcup_{c_k=i} \tilde{\mathcal{P}}_k \right) = \bigcup_{i=1}^{D} \Omega_i,
\end{equation}
where $D=N_c$ is the number of subdomains. We call \cref{e:second} a coloring-based decomposition (CBD), which  is used in \Cref{s:cbd} to construct the CBD preconditioner.
%\begin{remark}
If the number of partitions $M=2^d$ in \cref{e:dd}, then \cref{e:second} reduces to \cref{e:first} because every subdomain (color) has only one partition.
%\end{remark}

According to the four color theorem~\cite{appel1977every}, we know $N_c = 4$ in 2D. For a general graph in 3D, $N_c$ can be arbitrarily large  but  is  usually  small  for  practical  problems  that  have  some regularity.  For example,  $N_c=8$ for a uniform partitioning of a uniform grid. To summarize, we state the following:

\begin{theorem}
For a uniform partitioning of a uniform grid over $\Omega = [0,1]^d$, the number of subdomains (colors)  is $2^d$ in the CBD.
\end{theorem}

As a summary of \Cref{ss:odd,ss:cbd}, \cref{t:dd} shows the three preconditioners to be introduced associated with \cref{e:dd,e:first,e:second}.

\begin{table}
    \centering \small 
    \caption{\em Given a domain decomposition $\Omega = \bigcup_{i=1}^{D} \Omega_i$, the definition of a subdomain $\Omega_i$ leads to different preconditioners including the block Jacobi preconditioner (\Cref{s:schwarz}), the Schwarz preconditioner (\Cref{s:schwarz}), and the CBD preconditioner (\Cref{s:cbd}).}
    \label{t:dd}
    \begin{tabular}{cccc} 
    \toprule
    & Jacobi & Schwarz & CBD \\ \midrule 
    $\Omega_i$ & $\mathcal{P}_i$ & $\tilde{\mathcal{P}}_i$ & $\bigcup_{c_k=i} \tilde{\mathcal{P}}_k$ \\
    \bottomrule
    \end{tabular}
\end{table}

%\begin{figure}
%     \centering
%     \hfill
%     \begin{subfigure}[b]{0.3\textwidth}
%         \centering
%         \includegraphics[width=0.6\textwidth]{partition}
%         \caption{Coloring of a partitioning of the unit square.}
%     \end{subfigure}
%     \hfill
%     \begin{subfigure}[b]{0.3\textwidth}
%         \centering
%         \includegraphics[width=0.6\textwidth]{part}
%         \caption{Coloring of a partitioning of the unit square.}
%     \end{subfigure}
%     \hfill
%     \begin{subfigure}[b]{0.3\textwidth}
%         \centering
%         \includegraphics[width=0.6\textwidth]{cbd_color}
%         \caption{Coloring of a partitioning of the unit square.}
%     \end{subfigure}
%     \hfill
%     \begin{subfigure}[b]{0.3\textwidth}
%         \centering
%         \includegraphics[width=0.6\textwidth]{c}
%         \caption{Subdomain $\Omega_1$ (gray)  in the CBD scheme.}
%     \end{subfigure}
%     \hfill
%     \caption{\em Illustration of the CBD scheme: the unit square is partitioned into 16 parts; every part and its extension is assigned a color (adjacent parts have different colors); and extensions of parts with the same color form a subdomain.}
%     \label{f:cbd}
%\end{figure}

\section{Schwarz preconditioner} \label{s:schwarz}

Given a domain decomposition $\Omega = \bigcup_{i=1}^D \Omega_i$ (not necessarily an overlapping domain decomposition), we construct a preconditioner as follows. Let $\mathcal{I}_i$ denote the indices of the discretization points in $\Omega_i$. First, we define the restriction operator for every subdomain $\Omega_i$:
\[
R_i = I_N(\mathcal{I}_i,:),
\]
a subset of rows in the identity matrix $I_N \in \mathbb{R}^{N \times N}$ corresponding to indices $\mathcal{I}_i$. 
%In other words, $R_i \in \mathbb{R}^{N_i \times N}$ extracts the entries corresponding to $\mathcal{I}_i$ from a vector in $\mathbb{R}^{N}$.
Then, we define the subproblem associated with $\Omega_i$:
\begin{equation} \label{e:ai}
A_i = R_i A R_i^\top.
\end{equation}
Finally, we define the preconditioner:
\begin{equation} \label{e:ad}
%\boxed{
T^{-1} = \sum_{i=1}^D R_i^\top A_i^{-1} R_i.
%}
\end{equation}
With the overlapping domain decompositions \cref{e:first,e:second}, we obtain the Schwarz preconditioner (technically, single-level additive Schwarz preconditioner) and the CBD preconditioner (\Cref{s:cbd}). With the non-overlapping domain decomposition \cref{e:dd}, we obtain the block Jacobi preconditioner.

In the following, we focus on the Schwarz preconditioner $T^{-1}_\text{Schwarz}$, derive a theorem on $\lambda_{\max}(T^{-1}_\text{Schwarz}A)$, and show empirical results on $\lambda_{\min}(T^{-1}_\text{Schwarz}A)$. We also compare the Schwarz preconditioner to the block Jacobi preconditioner.

\subsection{Maximum and minimum eigenvalues}

Consider the preconditioned matrix
\begin{equation} \label{e:ta}
T^{-1}A = \sum_{i=1}^D R_i^\top A_i^{-1} R_i \, A.
\end{equation}
We define 
\begin{equation} \label{e:pi}
P_i = R_i^\top A_i^{-1} R_i A,
\end{equation}
and we can verify that $P_i$ is an orthogonal projection with respect to the inner product defined by the SPD matrix $A$:
\[
P_i^2 = P_i, \quad AP_i = P_i^\top A.
\] 
As a result, we have
\[
%\|P_i\| = 1,
\lambda_{\max}(P_i) = 1,
\]
which immediately leads to the following:

\begin{theorem} %[$\lambda_{\max}(T^{-1} A)$]
\label{th:max}
Let $A$ be an SPD matrix and  the preconditioner  $T^{-1}$ defined in \cref{e:ad}. Then, 
\[
\lambda_{\max}(T^{-1}A) \le  D,
\]
where $D$ is the number of subdomains.
\end{theorem}
\begin{proof}
%It is obvious that
$\lambda_{\max}(T^{-1}A) \le \sum_{i=1}^D \lambda_{\max}(P_i) = D$.
\end{proof}
Note the theorem applies to both the block Jacobi and the Schwarz preconditioners because it only requires a decomposition of the problem domain (not necessarily an overlapping decomposition); see numerical results in \cref{t:schwarz}. The table also shows that as the problem size $N$ increases, $\lambda_{\min}(T^{-1}_{\text{Schwarz}} A)$ appears to converge to a positive constant close to one, whereas $\lambda_{\min}(T^{-1}_{\text{Jacobi}} A)$ is much smaller and keeps decreasing. We state the following (and provide more concrete evidence in \Cref{ss:jacobi}):

\begin{conjecture} %[$\lambda_{\min}(T^{-1}_{\text{Schwarz}} A)$]
\label{conjecture}
Let $A$ be the discretized integral operator in \cref{e:a_2d} or \cref{e:a_3d}. Given the overlapping domain decomposition \cref{e:first}, the Schwarz preconditioner $T^{-1}_{\text{Schwarz}}$ is defined in \cref{e:ad}. Then,
\[
\lambda_{\min}(T^{-1}_{\text{Schwarz}} A) \ge C > 0,
\]
where $C$ is a positive constant.
% independent of the problem size $N$ and the number of subdomains/partitions $D$.
\end{conjecture}

\begin{table}
    %\centering \small 
    \caption{\em Maximum and minimum eigenvalues of the preconditioned matrix \cref{e:ta}. The block Jacobi and the Schwarz preconditioners are defined in \cref{t:dd}. The number of subdomains $D=\bigO(N)$.
    %\cref{e:ad}
    %A uniform grid in the unit square of size $N$ is partitioned uniformly into $M$ subdomains.
    }
    \label{t:schwarz}
    \begin{subtable}[h]{\textwidth}
    \centering
    \begin{tabular}{cc|cc|cc} 
    \toprule
      \multirow{2}{*}{$N$}  &  \multirow{2}{*}{$D$}  
      & \multicolumn{2}{c|}{Jacobi}  & \multicolumn{2}{c}{Schwarz}  \\
      && $\lambda_{\max}$ & $\lambda_{\min}$ & $\lambda_{\max}$ & $\lambda_{\min}$  \\ \midrule 
      $8^2$     & $2^2$       & 2.8479     & 0.1695 & 4.0000    & 0.8209       \\
      $16^2$   & $4^2$       & 6.6883     & 0.0804 & 8.8046    & 0.9112       \\
      $32^2$   & $8^2$       & 19.3756     & 0.0533 & 23.5948    & 0.9350        \\
      $64^2$   & $16^2$       & 61.6629     & 0.0409 & 71.5192    & 0.9360       \\
      $128^2$ & $32^2$       & 205.7705     & 0.0336 & 231.7813    & 0.9331       \\
      % 781.1555		0.9294
    \bottomrule
    \end{tabular}
    \caption{\em Discretized integral operator in 2D; see \cref{e:a_2d}.}
    \end{subtable}
    % 3D
    %\centering \small 
    \begin{subtable}[h]{\textwidth}
    \centering
    \begin{tabular}{cc|cc|cc} 
    \toprule
      \multirow{2}{*}{$N$}  &  \multirow{2}{*}{$D$}  
      & \multicolumn{2}{c|}{Jacobi} & \multicolumn{2}{c}{Schwarz} \\
      & &  $\lambda_{\max}$  & $\lambda_{\min}$  & $\lambda_{\max}$  & $\lambda_{\min}$  \\ \midrule 
      $4^3$     & $2^3$       & 4.0618     & 0.2602 & 8.0000    & 0.9750      \\
      $8^3$     & $4^3$       & 15.4234     & 0.2056 & 33.1198    & 0.9942      \\
      $16^3$   & $8^3$       & 60.9327     & 0.1916 & 134.7002    & 0.9967        \\
      $32^3$   & $16^3$       & 242.9799     & 0.1878 & 547.3752    & 0.9972       \\
    \bottomrule
    \end{tabular}
    \caption{\em Discretized integral operator in 3D; see \cref{e:a_3d}.}
    \end{subtable}
\end{table}

\begin{remark}
If $A$ is a \emph{sparse} SPD matrix arising from a local discretization (e.g., by finite elements or finite differences) of a second-order self-adjoint, coercive elliptic problem, then the upper bound in \cref{th:max} can be reduced to a constant, independent of the number of subdomains $D$~\cite{chan1994domain,dolean2015introduction}. The reason is that $P_i + P_j$ is still an orthogonal projection when $\Omega_i \cap \Omega_j = \emptyset$ (because $R_i A R_j^\top = 0$ when $A$ is sparse).
\end{remark}

\begin{remark}
In the context of solving PDEs, the (single-level additive) Schwarz preconditioner results in a bounded maximum eigenvalue of the preconditioned matrix, whereas the minimum eigenvalue goes to zero as $M=\bigO(N)\to\infty$.% Interestingly, the observation for solving IEs considered in this paper is reversed.
\end{remark}

\subsection{Analysis for $D=2^d$} \label{ss:jacobi}
%\subsection{Analysis in 1D}

%\subsection{Comparison to block Jacobi preconditioner} \label{ss:jacobi}

In this section, we assume the number of subdomains $D=2^d$ and study the maximum and minimum  eigen-pairs of $T^{-1}_{\text{Schwarz}} A$.  The analysis of the maximum eigen-pair  prepares for the corresponding analysis for $T^{-1}_{\text{CBD}} A$ in \Cref{ss:spectral}. The observations for the minimum eigen-pair provide insights and evidence for \cref{conjecture}. We also compare the Schwarz preconditioner to the block Jacobi preconditioner, for which \cref{conjecture} does not hold. 

Recall $P_i$ defined in \cref{e:pi} is a projection associated with subdomain $\Omega_i$, so we have
\begin{lemma} \label{l:subdomain}
Let $x \in \mathbb{R}^N$ be nonzero only in $\Omega_i$, i.e., $x = R_i^\top R_i \, x$. Then, 
\[
P_i  \,  x = x.
\]
\end{lemma}
\begin{proof}
With the definitions of $A_i$ and $P_i$ in \cref{e:ai,e:pi}, respectively, we have
\[
P_i  \, x = (R_i^\top A_i^{-1} R_i A) (R_i^\top R_i  \, x) = R_i^\top A_i^{-1} (R_i A R_i^\top) R_i  \, x = R_i^\top R_i  \, x = x.
\]
\end{proof}
This lemma is useful for analyzing eigenvectors of the preconditioned matrix. The following theorem addresses the maximum eigenvalue and the associated eigenvector of $T^{-1}_{\text{Schwarz}} A$.
\begin{theorem} %[$D=2^d$] %[$\lambda_{\max}(T^{-1}_{\text{Schwarz}} A)$]
\label{th:max48}
Suppose the problem domain $\Omega = [0,1]^d$ is partitioned uniformly into $D=2^d$ overlapping subdomains. 
%Assume the number of subdomains $D=2^d$. 
Given the overlapping domain decomposition \cref{e:first}, we have
\[
\lambda_{\max}(T^{-1}_{\text{Schwarz}} A)=2^d,
\]
and the corresponding eigen-space is of dimension $2^d$ and consists of vectors that are nonzero only in the overlapped region shared by all $2^d$ subdomains.
\end{theorem}
\begin{proof}
\begin{figure}
         \begin{subfigure}[b]{0.48\textwidth}
         \centering
         \includegraphics[width=0.46\textwidth]{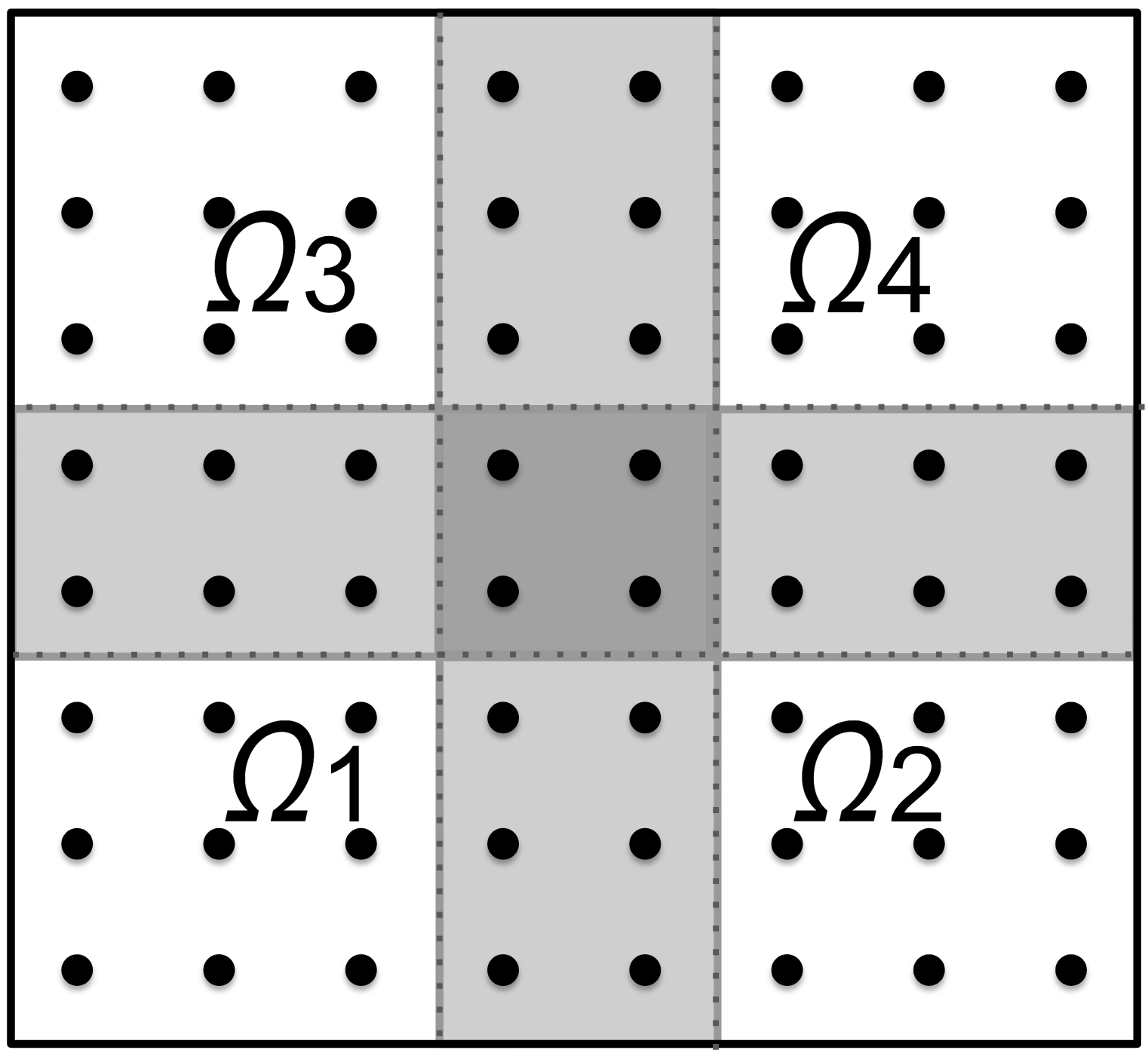}
         \caption{}
        \label{f:shared}
        \end{subfigure}
        \begin{subfigure}[b]{0.48\textwidth}
         \centering
         \includegraphics[width=0.46\textwidth]{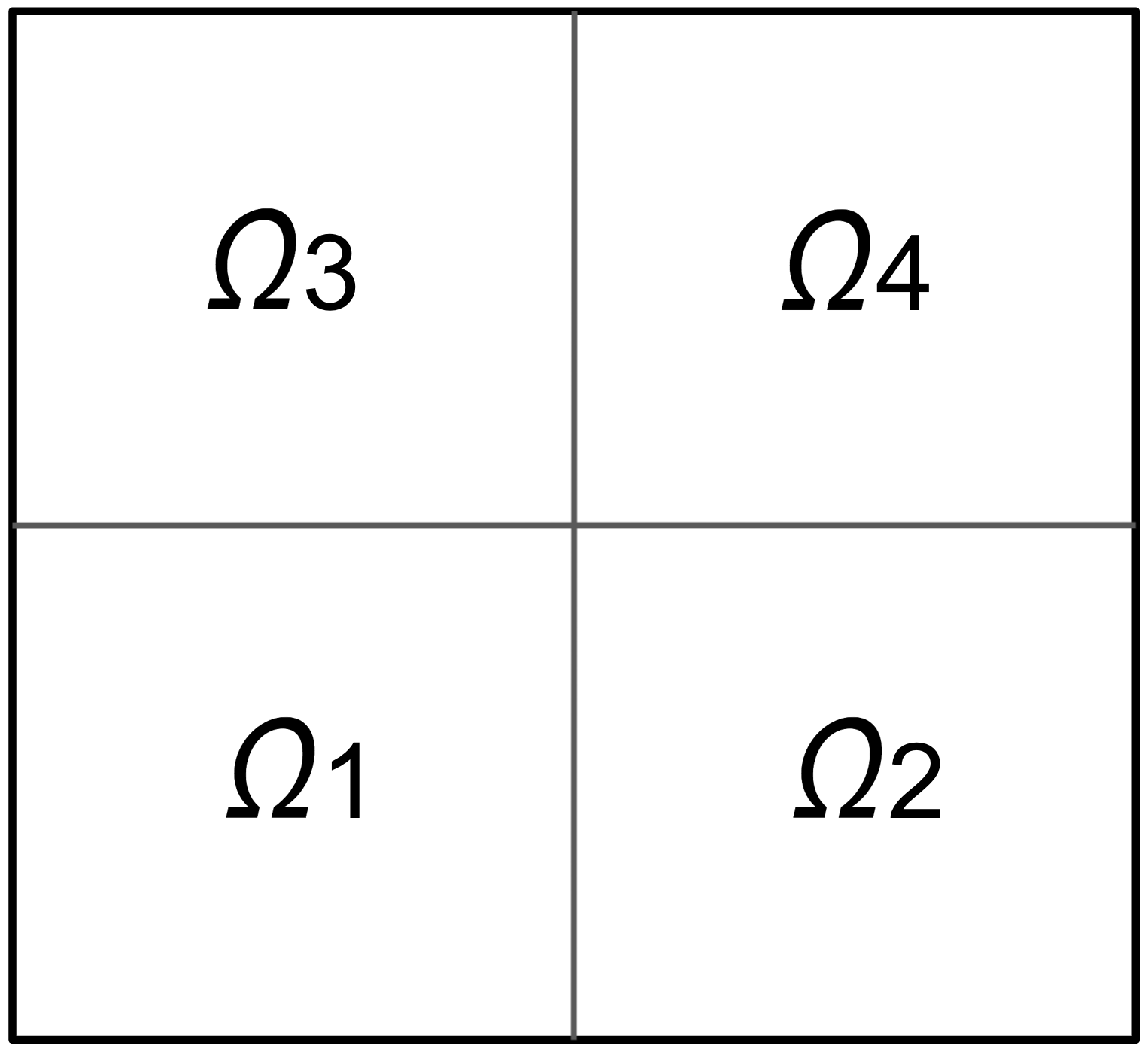}
         \caption{}
        \label{f:shared_jacobi}
        \end{subfigure}
        \caption{\em (a) $D=4$ in \cref{e:first}. Shaded regions are shared by adjacent subdomains, and the four grid points at the center are shared by all four subdomains. (b) $D=4$ in \cref{e:dd}. As $N\to\infty$, all four subdomains contain the center of the square.}        
\end{figure}

According to \cref{th:max}, we know that $\lambda_{\max}(T^{-1}_{\text{Schwarz}} A) \le 2^d$. Recall that when $D=2^d$, we have $\cap_{i=1}^{2^d} \Omega_i \not= \emptyset$, and the shared region contains $2^d$ grid points; see \cref{f:shared} for an example. Suppose $x$ is nonzero only in the shared region. %; see \cref{sf:vmax_1d,sf:vmax_2d} for examples. 
According to \cref{l:subdomain}, we know that $P_i \, x = x$ for all $i$. Therefore,
\[
T^{-1}_{\text{Schwarz}} A \, x = \sum_{i=1}^{2^d} P_i \, x = 2^d \, x.
\]
\end{proof}
The theorem does not apply to $T^{-1}_{\text{Jacobi}} A$ directly but provides insight on the maximum eigen-pair when the problem size $N \to \infty$. In the limit, all subdomains associated with the block Jacobi preconditioner share the center of the problem domain $\Omega=[0,1]^d$; see \cref{f:shared_jacobi} for an example. So we know that %\cref{th:max48} serves as the  
\begin{equation} \label{e:emax_jacobi}
\lim_{N\to\infty} \lambda_{\max}(T^{-1}_{\text{Jacobi}} A)=2^d;
\end{equation}
see numerical results in \cref{t:cbd}. The corresponding eigenvector converges to the indicator function of the center point $x_c$, i.e.,
\begin{equation} \label{e:delta}
\bm{1}_{x_c}(x) = \\
\left\{
\begin{array}{cc}
1 & x = x_c, \\
0 & \text{else};
\end{array}
\right.
\end{equation}
see numerical results plotted in \cref{sf:max_1d,sf:max_2d}.

Next, we consider the minimum eigen-pair of $T^{-1}_{\text{Jacobi}} A$. We start with an example in one dimension (1D). Suppose the problem domain $\Omega = [0,1]$ is cut into two halves of equal length ($D=2$), i.e., $\Omega_1 = [0,1/2], \Omega_2 = [1/2, 1]$. Suppose a uniform discretization grid has $N$ nodes, where $N$ is an even integer and the discretization points are indexed from left to right. Let the indices of discretization points in $\Omega_1$ and $\Omega_2$ be $\mathcal{I}_1 = \{1,2,\ldots, N/2\}$ and $\mathcal{I}_1 = \{N/2+1,N/2+2,\ldots, N\}$, respectively. Recall the block Jacobi preconditioner $T^{-1}_{\text{Jacobi}}$ defined in \cref{e:ad}. Consider the eigenvalue problem
\begin{equation} \label{e:eig}
T^{-1}_{\text{Jacobi}} \, A \, x = \lambda \, x,
\end{equation}
where $\lambda$ is an eigenvalue. Write  $x=x_1+x_2$, where $x_i = R_i x$ for $i=1,2$. We state the following theorem and give the proof in \cref{s:app}.
\begin{theorem} \label{th:jacobi}
Let $x=x_1+x_2$ be an eigenvector associated with an eigenvalue $\lambda$ of the preconditioned matrix $T^{-1}_{\text{Jacobi}} \, A$ in 1D. Then, $x'=x_1-x_2$ is also an eigenvector, and the corresponding eigenvalue is $2-\lambda$. In other words,
\[
T^{-1}_{\text{Jacobi}} \, A \, x = \lambda \, x
\quad
\Leftrightarrow
\quad
T^{-1}_{\text{Jacobi}} \, A \, x' = (2-\lambda) \, x'.
\]
\end{theorem}

\begin{figure}
     \centering
     \begin{subfigure}[b]{0.48\textwidth}
         \centering
         \includegraphics[width=0.67\textwidth]{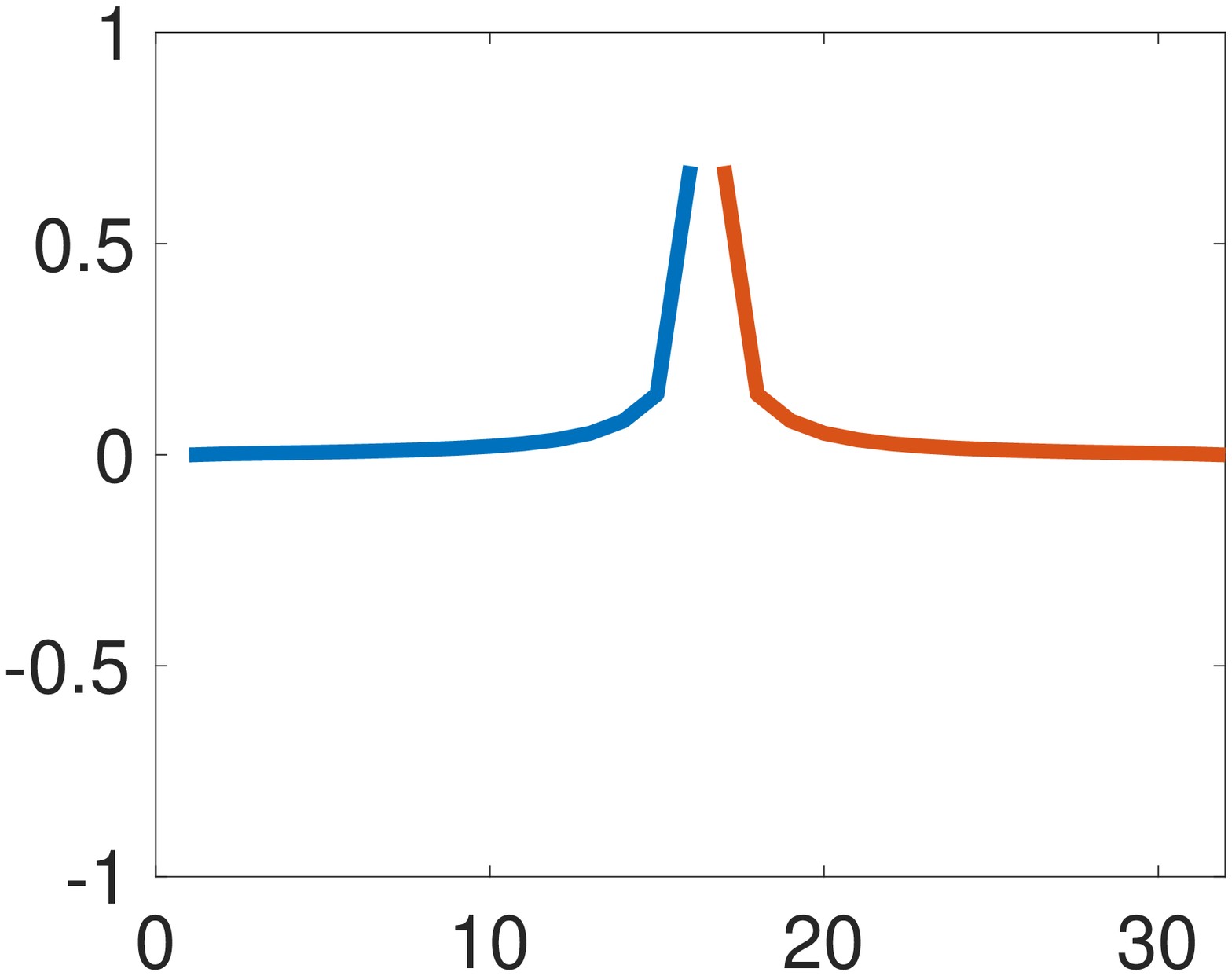}
         \caption{\em Eigenvector for $\lambda_{\max}$ in 1D}
         \label{sf:max_1d}
     \end{subfigure}
     \begin{subfigure}[b]{0.48\textwidth}
         \centering
         \includegraphics[width=0.67\textwidth]{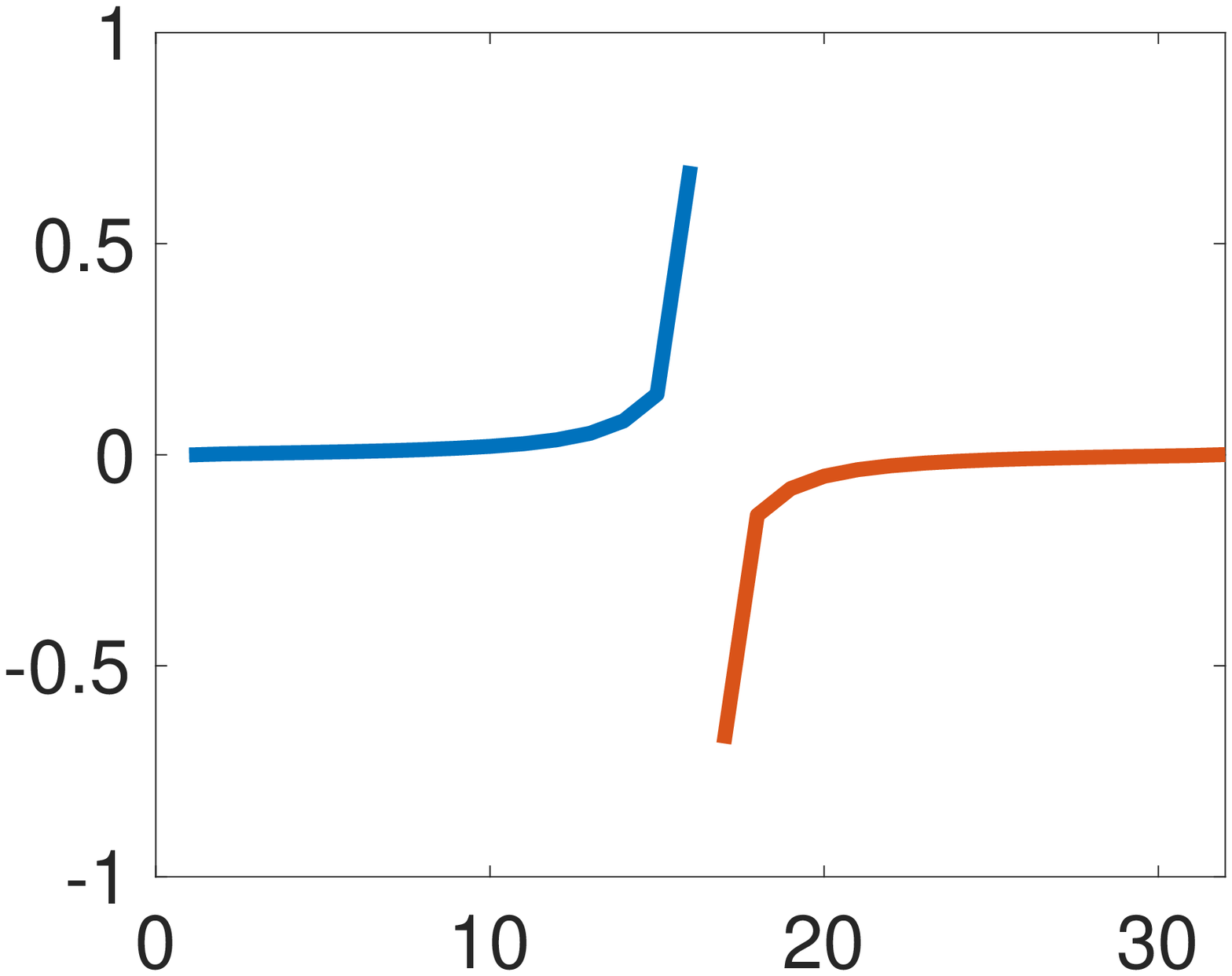}
         \caption{\em Eigenvector for $\lambda_{\min}$ in 1D}
         \label{sf:min_1d}
     \end{subfigure}
     % 2nd row
     \centering
          \begin{subfigure}[b]{0.48\textwidth}
         \centering
         \includegraphics[width=\textwidth]{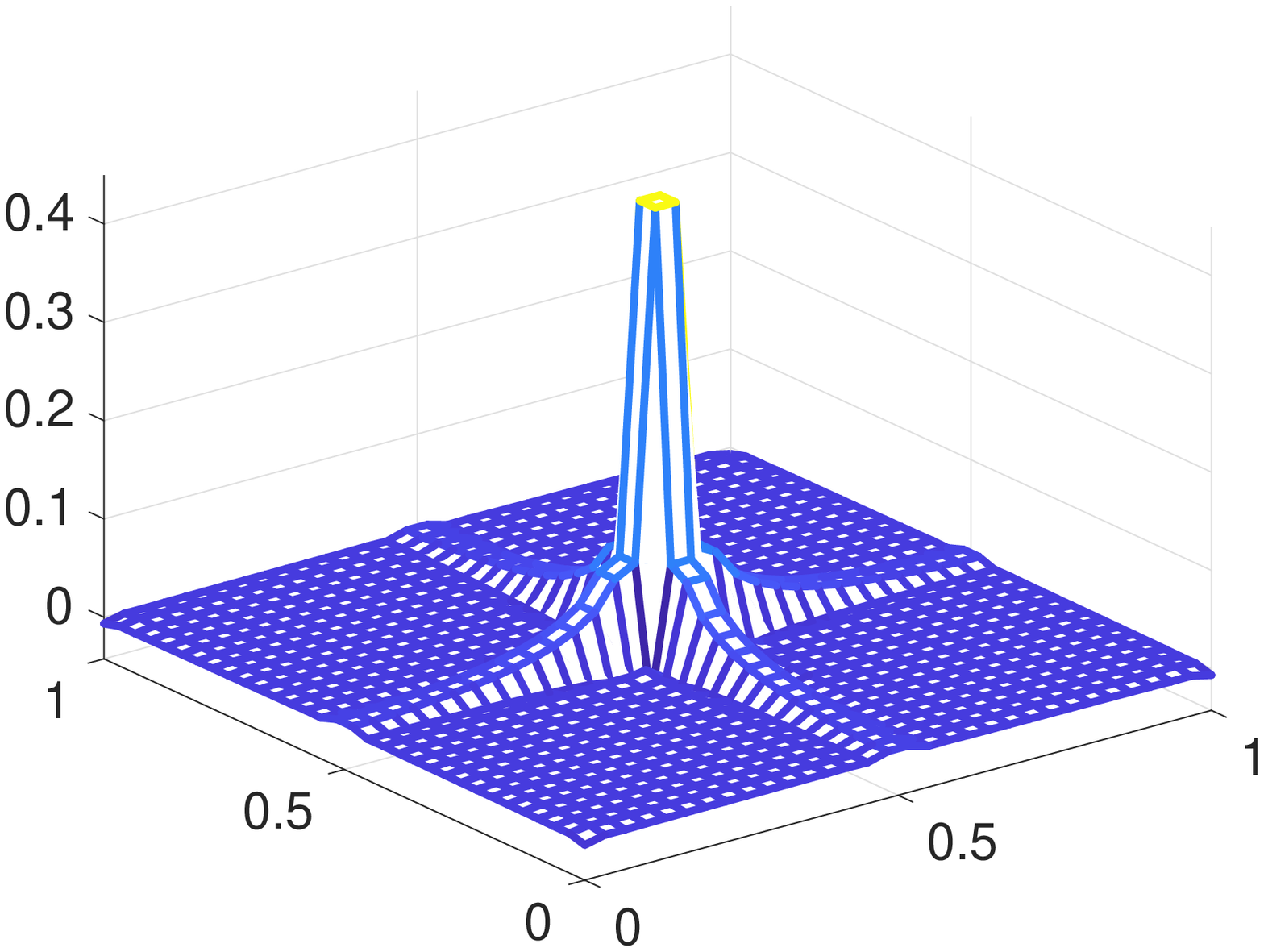}
         \caption{\em Eigenvector for $\lambda_{\max}$ in 2D}
         \label{sf:max_2d}
     \end{subfigure}
     \begin{subfigure}[b]{0.48\textwidth}
         \centering
         \includegraphics[width=\textwidth]{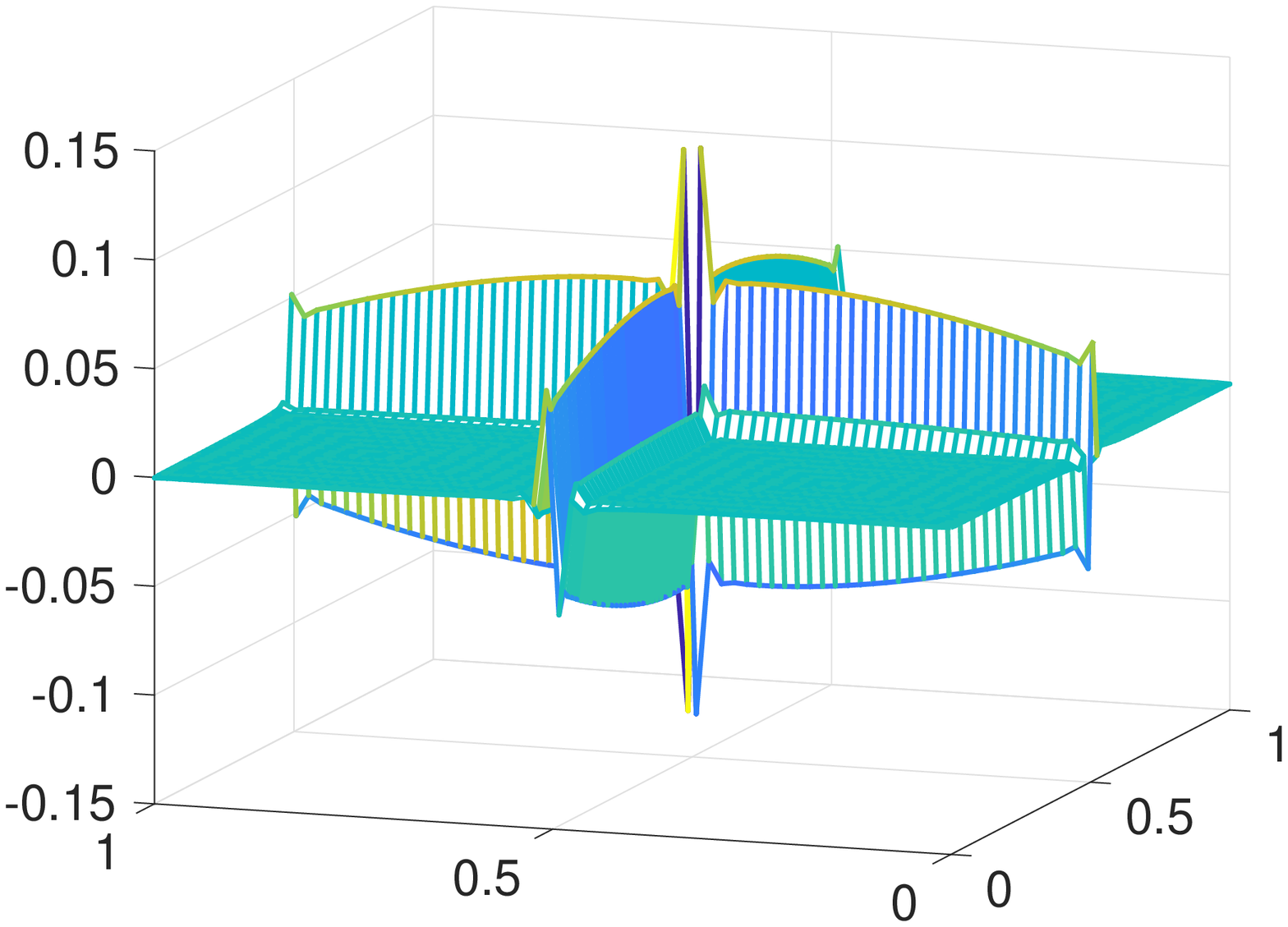}
         \caption{\em Eigenvector for $\lambda_{\min}$ in 2D}
         \label{sf:min_2d}
     \end{subfigure} 
     \caption{\em Eigenvectors of $T^{-1}_{\text{Jacobi}} A$ in 1D ($N=32, D=2$) and 2D ($N=32^2, D=2^2$). The discretized integral operator is given in \cref{e:a_2d}.}
     \label{f:jacobi_evec}
\end{figure}

According to the above theorem, \cref{e:emax_jacobi} implies
\begin{equation} \label{e:emin_jacobi}
\lim_{N\to\infty} \lambda_{\min}(T^{-1}_{\text{Jacobi}} A) = 0
\end{equation}
for our example in 1D; see numerical results of eigenvectors in \cref{sf:max_1d,sf:min_1d}. 
One interpretation of the eigenvectors of $T^{-1}_{\text{Jacobi}} A$ from electrostatics is the following. 
The limit ($N\to\infty$) of the eigenvector associated with $\lambda_{\max}$, namely, \cref{e:delta}, can be interpreted as putting a point charge at the domain center. The eigenvector associated with $\lambda_{\min}$  corresponds to putting dipoles near the boundaries of adjacent subdomains. For the example in 1D, the boundary is just one point $x_c=1/2$; see \cref{sf:min_1d}. For the unite square domain in 2D, the boundaries are two lines $1/2 \times [0,1]$ and $[0,1] \times 1/2$; see \cref{sf:min_2d}. As a result, the electric potential induced by all charges cancels out approximately.
Based on this intuition, we conjecture that \cref{e:emin_jacobi} also holds when $d=2$ and 3; see numerical results in \cref{t:cbd}.

Finally, we consider the minimum eigen-pair of $T^{-1}_{\text{Schwarz}} A$. Let us revisit the previous example in 1D, where the  problem domain $\Omega = [0,1]$ is discretized with a uniform grid of size $N$ (an even integer) and the grid points are indexed from left to right. To construct the Schwarz preconditioner, we use the overlapping domain decomposition where $\Omega_1$ and $\Omega_2$ contain the first and the last $N/2+1$ grid points, respectively. \cref{sf:vmin_1d} shows the eigenvector corresponding to $\lambda_{\min}$. In the figure, there is a ``dipole'' at the overlapped region, and the eigenvector decays more slowly away from the boundary than that in \cref{sf:min_1d}. 
 This observation extends to higher dimensions, and the differences between the two eigenvectors for $\lambda_{\min}$ are more pronounced  in 2D; compare \cref{sf:vmin_2d} to \cref{sf:min_2d}. 
 Both eigenvectors contain ``dipoles'' near the boundaries of the four subdomains in 2D. For the eigenvector associated with the Schwarz preconditioner, the ``dipoles'' mainly concentrate around $(1/2,0), (1/2,1), (0,1/2)$ and $(1,1/2)$ on the boundary of the square domain. 
For the eigenvector associated with the block Jacobi preconditioner, however, 
the ``dipoles'' mainly exist in the interior of the domain. This empirical observation extends to cases when $D>2^d$. 
%It is not clear how 
%{We speculate that the behavior of the eigenvector associated with the smallest eigenvalue }
The difference between the minimum eigenvalues of the preconditioned matrix is shown in \cref{t:cbd} for problems in 2D and in 3D. %, which provides evidence for  \cref{conjecture}.

\begin{figure}
     \centering
%     \begin{subfigure}[b]{0.48\textwidth}
%         \centering
%         \includegraphics[width=0.67\textwidth]{evec_schwarz_large_1d}
%         \caption{Eigenvector for $\lambda_{\max}$ in 1D}
%         \label{sf:vmax_1d}
%     \end{subfigure}
     \begin{subfigure}[b]{0.48\textwidth}
         \centering
         \includegraphics[width=0.75\textwidth]{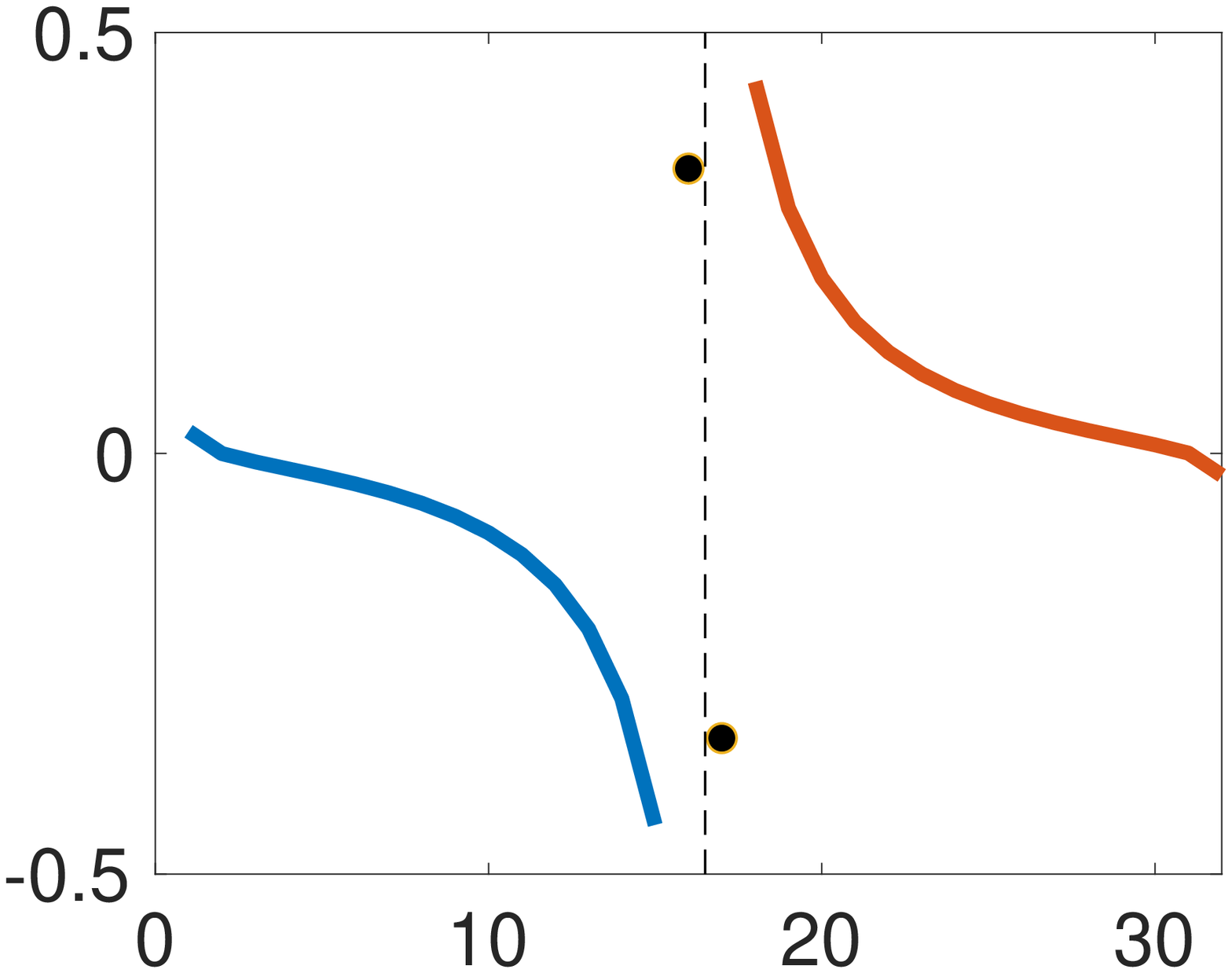}
         \caption{\em Eigenvector for $\lambda_{\min}$ in 1D}
         \label{sf:vmin_1d}
     \end{subfigure}
%     \centering
%     \begin{subfigure}[b]{0.48\textwidth}
%         \centering
%         \includegraphics[width=\textwidth]{evec_schwarz_large}
%         \caption{Eigenvector for $\lambda_{\max}$ in 2D}
%         \label{sf:vmax_2d}
%     \end{subfigure}
     \begin{subfigure}[b]{0.48\textwidth}
         \centering
         \includegraphics[width=\textwidth]{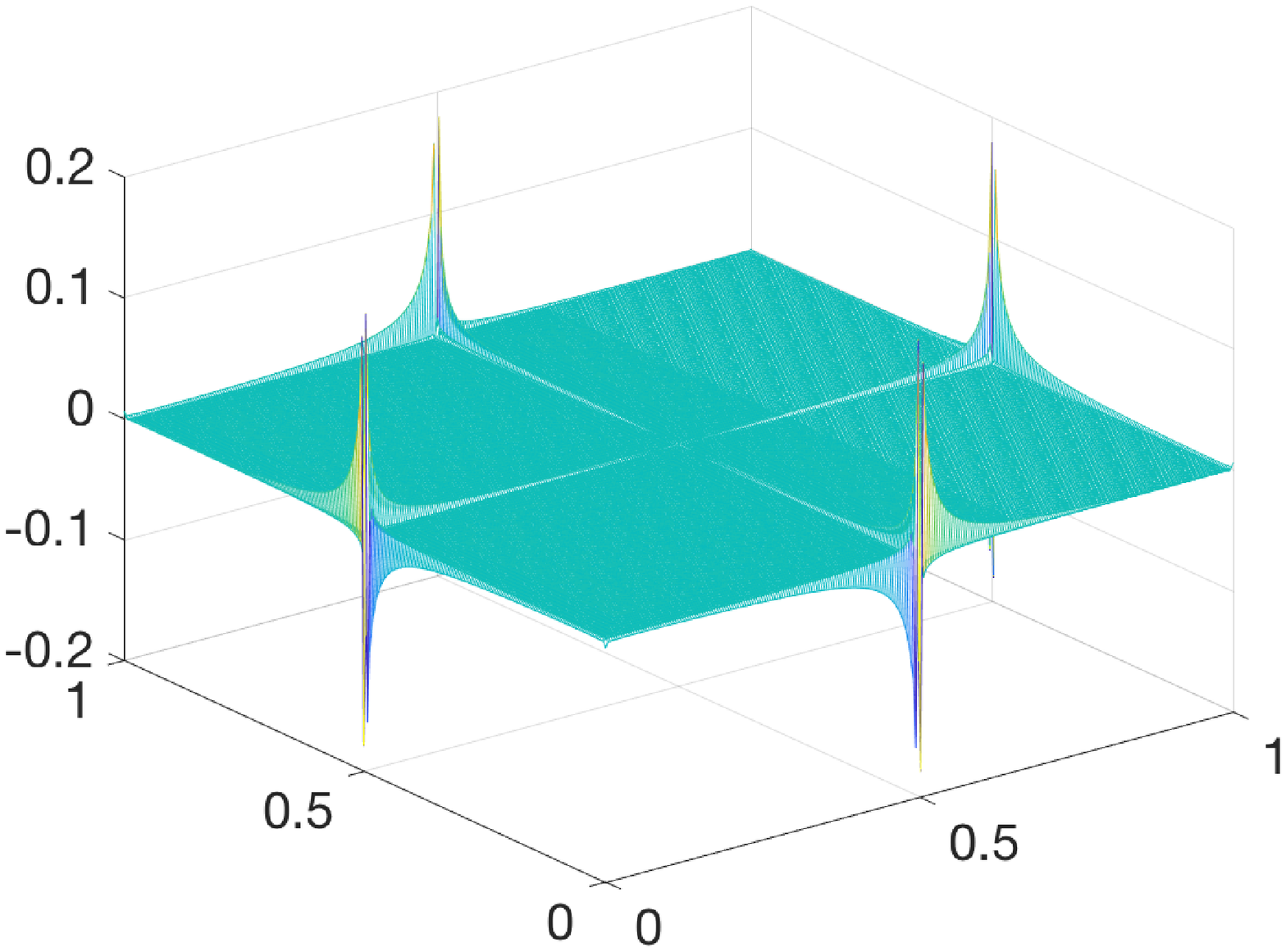}
         \caption{\em Eigenvector for $\lambda_{\min}$ in 2D}
         \label{sf:vmin_2d}
     \end{subfigure}     
     \caption{\em Eigenvectors of $T^{-1}_{\text{Schwarz}} A$ in 1D ($N=32, D=2$) and 2D ($N=256^2, D=2^2$). In (a), two dots (in black) lie in the overlapped region of two subdomains, and two curves (in blue and red, respectively) lie in the interior of two subdomains, respectively. The discretized integral operator is given in \cref{e:a_2d}.}
\end{figure}

\section{Coloring-based decomposition (CBD) and preconditioner} \label{s:cbd}

%We introduce a new strategy of domain decomposition, where every subdomain consists of non-adjacent parts as shown in \cref{f:dd_new}. In general, such a decomposition can be computed using graph coloring following three steps. First, we partition the entire domain into $\bigO(N)$ parts. Second, we apply graph coloring on the graph consists of all parts and their physical adjacency. Finally, parts assigned the same color form a subdomain. According to the famous four-color theorem, we have at most four colors or subdomains in 2D. Although the number of colors can be  infinite for an arbitrary 3D graph, real-world applications normally need only a few colors.

In this section, we introduce the CBD preconditioner based on the decomposition \cref{e:second}. Recall that a subdomain consists of spatially distant regions as shown in \cref{sf:subdomain} and the algebraic formulation of the CBD preconditioner is given in \cref{e:ad}. In \Cref{ss:spectral}, we show that the maximum and minimum eigenvalues of the preconditioned matrix $T^{-1}_{\text{CBD}} A$ are bounded from above and from below, respectively. In \Cref{ss:subdomain}, we apply the recursive skeletonization (RS) factorization~\cite{ho2016hierarchical} to construct approximate factorizations of subproblems in the CBD preconditioner, which are used to apply the preconditioner efficiently. In \Cref{ss:compare}, we compare the CBD preconditioner to the original RS method, and in \Cref{s:fast}, we provide some analysis of the CBD preconditioner.

%Of course, such a preconditioner would not be interesting if the construction cost is $\bigO(N^3)$. So we employ an FDS on every subdomain for fast construction and reduced memory footprint.

%As we see in the previous section, the minimum eigenvalue of the Schwarz preconditioned matrix is lower bounded while the maximum eigenvalue increases as the number of subdomains increases. So we introduce a new (overlapping) domain decomposition strategy, where the number of subdomains does not depend on the problem size. 

%We apply the RS factorization~\cite{ho2012fast,ho2016hierarchical} on every subdomain to construct an approximate factorization of $A_i$ in \cref{e:ai} corresponding to a subdomain in the CBD preconditioner. In the following, we first review the original RS algorithm and discuss its computation and storage complexities. Then, we explain using RS as the subdomain solver in our CBD preconditioner. Finally, we propose an asymptotically faster solver than the original RS method.

\begin{table}
    %\centering \small 
    \caption{\em Maximum and minimum eigenvalues of the preconditioned matrix \cref{e:ta}. The block Jacobi, the Schwarz and the CBD preconditioners are defined in \cref{t:dd}. The number of subdomains $D=2^d$ is fixed.
    %\cref{e:ad}
    %A uniform grid in the unit square of size $N$ is partitioned uniformly into $M$ subdomains.
    }
    \label{t:cbd}
    \begin{subtable}[h]{\textwidth}
    \centering
    \begin{tabular}{cc|cc|cc|ccc} 
    \toprule
      \multirow{2}{*}{$N$}  &  \multirow{2}{*}{$D$}  
      & \multicolumn{2}{c|}{Jacobi}  & \multicolumn{2}{c|}{Schwarz} & \multicolumn{2}{c}{CBD} \\
      && $\lambda_{\max}$ & $\lambda_{\min}$ & $\lambda_{\max}$ & $\lambda_{\min}$ 
      & $M$ & $\lambda_{\max}$ & $\lambda_{\min}$  \\ \midrule 
      $8^2$     & 4       & 2.8479     & 0.1695 & 4.0000    & 0.8209 & $2^2$ & 4.0000    & 0.8209      \\
      $16^2$   & 4       & 3.1876     & 0.0838 & 4.0000    & 0.8237 & $4^2$ & 4.0000    & 0.9201      \\
      $32^2$   & 4       & 3.3965     & 0.0419 & 4.0000    & 0.8280 & $8^2$ & 4.0000    & 0.9397       \\
      $64^2$   & 4       & 3.5349     & 0.0210 & 4.0000    & 0.8305 & $16^2$ & 4.0000    & 0.9403      \\
      $128^2$ & 4       & 3.6316     & 0.0105 & 4.0000    & 0.8317 & $32^2$ & 4.0000    & 0.9399      \\
    \bottomrule
    \end{tabular}
    \caption{\em Discretized integral operator in 2D; see \cref{e:a_2d}.}
    \end{subtable}
    % 3D
    %\centering \small 
    \begin{subtable}[h]{\textwidth}
    \centering
    \begin{tabular}{cc|cc|cc|ccc} 
    \toprule
      \multirow{2}{*}{$N$}  &  \multirow{2}{*}{$D$}  
      & \multicolumn{2}{c|}{Jacobi} & \multicolumn{2}{c|}{Schwarz} & \multicolumn{2}{c}{CBD} \\
      & &  $\lambda_{\max}$  & $\lambda_{\min}$  & $\lambda_{\max}$  & $\lambda_{\min}$
      & $M$ & $\lambda_{\max}$  & $\lambda_{\min}$  \\ \midrule 
      $4^3$     & 8       & 4.0618     & 0.2602 & 8.0000    & 0.9750  & $2^3$ & 8.0000    & 0.9750    \\
      $8^3$     & 8       & 4.6797     & 0.1532 & 8.0000    & 0.9408  & $4^3$ &  8.0000    & 0.9965    \\
      $16^3$   & 8       & 5.1116     & 0.0850 & 8.0000    & 0.9020  & $8^3$ &  8.0000    & 0.9992      \\
      $32^3$   & 8       & 5.4158     & 0.0451 & 8.0000    & 0.8661  & $16^3$ &  8.0000    & 0.9998     \\
    \bottomrule
    \end{tabular}
    \caption{\em Discretized integral operator in 3D; see \cref{e:a_3d}.}
    \end{subtable}
\end{table}

\subsection{Spectral equivalence} \label{ss:spectral}

\begin{figure}
%     \centering
%     \includegraphics[width=0.23\textwidth]{overlap}
     \hfill
     \begin{subfigure}[b]{0.4\textwidth}
         \centering
         \includegraphics[width=0.6\textwidth]{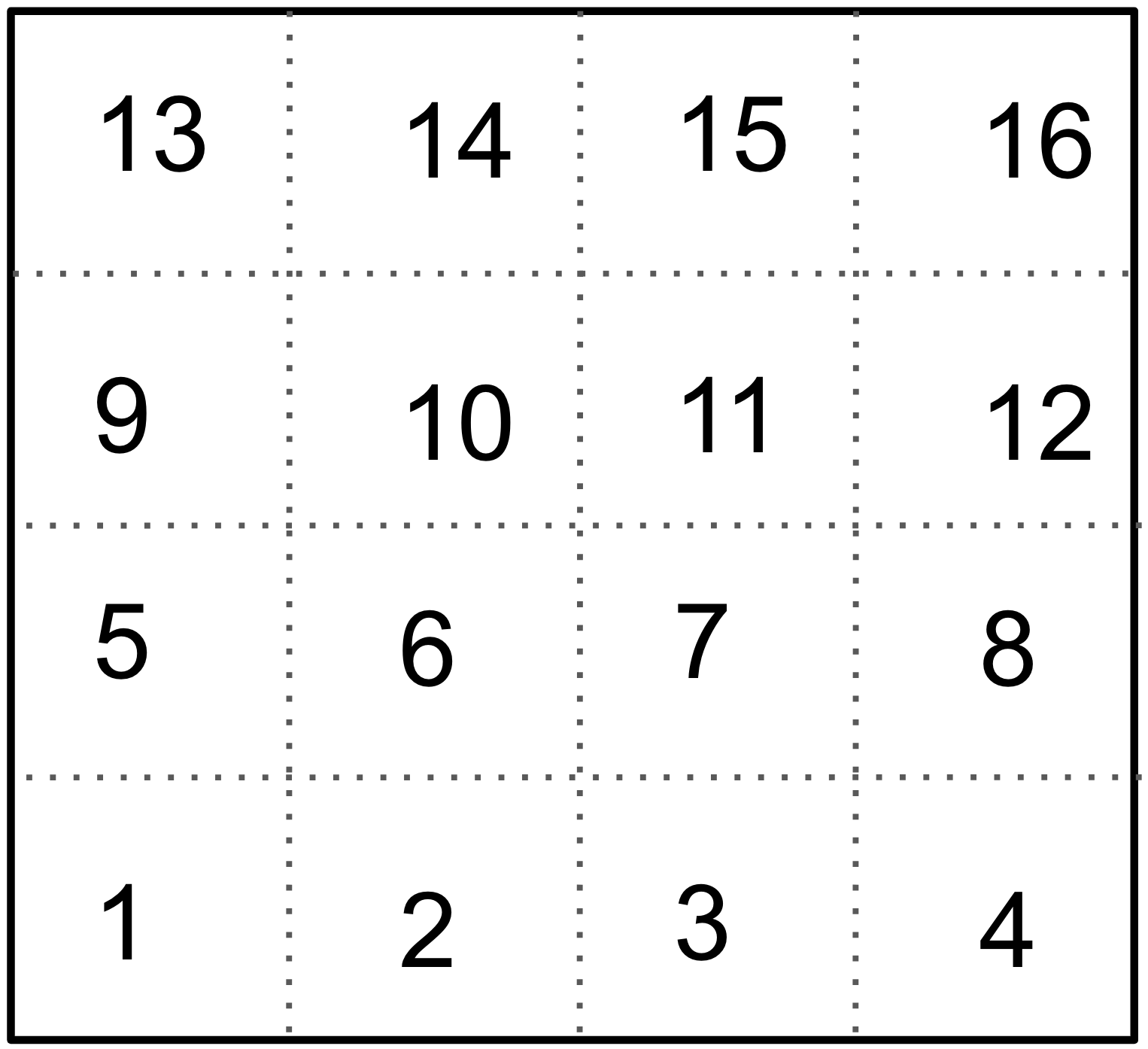}
         \caption{}
         \label{sf:partition2}
     \end{subfigure}
     \hfill
     \begin{subfigure}[b]{0.4\textwidth}
         \centering
         \includegraphics[width=0.6\textwidth]{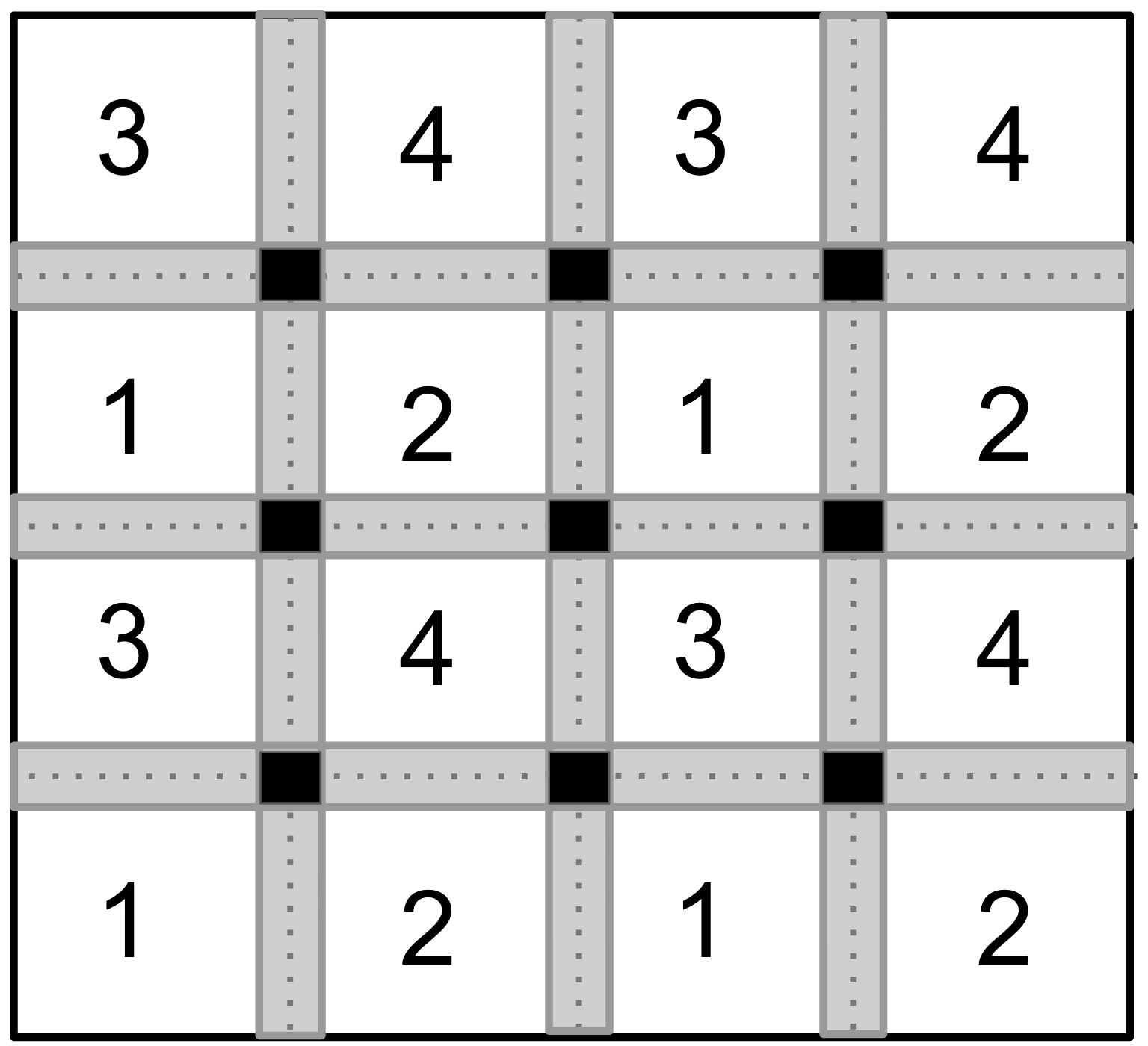}
         \caption{}
     	 \label{sf:overlap}     
     \end{subfigure}
     \hfill
     \caption{\em (a) a square domain is partitioned uniformly into 16 partitions; (b) shared region (in black) of the four subdomains in the CBD preconditioner. Every subdomain has four separated regions indicated by the index.}
     \label{f:cbd}
\end{figure}

The number of subdomains used in the CBD preconditioner is typically a constant independent of the problem size. This implies that the maximum eigenvalue of the preconditioned matrix is upper bounded according to  \cref{th:max}. In addition, numerical results reveal strong evidence that the minimum eigenvalue is lower bounded away from zero, which can also be justified by \cref{conjecture}. We illustrate these statements with more details in the following.

Consider the maximum eigenvalue of the preconditioned matrix $T^{-1}_{\text{CBD}} A$. As previously mentioned in \cref{s:setup}, we assume the problem domain $\Omega = [0,1]^d$ is discretized with a uniform grid and is partitioned uniformly into $M=m^d$ partitions ($m$ partitions along every dimension); see \cref{sf:partition2}. In this case, we have $D=2^d$ subdomains in the CBD preconditioner. There are $\left(m-1\right)^d$ overlapped regions shared by all subdomains; see \cref{sf:overlap} for an example. When $M=D=2^d$, the CBD preconditioner is reduced to the Schwarz preconditioner. We generalize the analysis in \Cref{ss:jacobi} and  \cref{th:max48} to the following:

\begin{theorem} %[$\lambda_{\max}(T^{-1}_{\text{CBD}} A) = 2^d$]
Suppose the problem domain $\Omega = [0,1]^d$ is partitioned uniformly. Given the overlapping decomposition in \cref{e:second}, where $M=m^d \ge D$, we have 
\[
\lambda_{\max}(T^{-1}_{\text{CBD}} A)=2^d,
\]
and the corresponding eigen-space is of dimension $\left(2 \left(m-1\right) \right)^d$ and consists of vectors that are nonzero only in the overlapped region shared by all $2^d$ subdomains.
% $\left(2 \left(m-1\right) \right)^d$ dimensions, spanned by indicator functions associated with the overlapped region shared by all $2^d$ subdomains.
\end{theorem}
\begin{proof}
The proof is almost identical to that of \cref{th:max48}, except there are $\left(m-1\right)^d$ overlapped regions shared by all subdomains. (Recall that every shared region has $2^d$ discretization points; see \cref{f:shared}.)
\end{proof}
%See numerical results of the maximum eigenvalue in \cref{t:cbd}.

Next, we consider the minimum eigenvalue of the preconditioned matrix in \cref{e:ta}. Assume we are given a partitioning of the problem domain as in \cref{e:dd}, where the number of partitions $M=\bigO(N)$. We can construct the Schwarz and the CBD preconditioners based on \cref{e:first,e:second}, respectively. In this case, a subdomain in the CBD preconditioner is a union of non-adjacent subdomains in the Schwarz preconditioner. This implies that every subproblem in the Schwarz preconditioner is a diagonal block in a subproblem in the CBD preconditioner. As a concrete example, we have $\Omega_1^\text{CBD} = \Omega_1^\text{Schwarz} \cup \Omega_3^\text{Schwarz} \cup \Omega_9^\text{Schwarz} \cup \Omega_{11}^\text{Schwarz}$ in \cref{f:cbd}. Correspondingly, we know the discretization points satisfy the same relationship, i.e., $\mathcal{I}_1^\text{CBD} = \mathcal{I}_1^\text{Schwarz} \cup  \mathcal{I}_3^\text{Schwarz} \cup \mathcal{I}_9^\text{Schwarz} \cup \mathcal{I}_{11}^\text{Schwarz}$. According to the definition \cref{e:ai}, we have
\[
A_{1}^\text{CBD} = 
\begin{bmatrix*}[l]
A_{1}^\text{Schwarz} & \times & \times & \times  \\
\times & A_{3}^\text{Schwarz} & \times & \times  \\
\times & \times & A_{9}^\text{Schwarz} & \times  \\
\times & \times & \times & A_{11}^\text{Schwarz}
\end{bmatrix*},
\]
{where $\times$ stands for the ``interaction'' between two Schwarz subdomains, entries in the original matrix $A$ but not in any $A_{i}^\text{Schwarz}$.} In general, the CBD preconditioner uses (much) more entries in the original problem $A$ than the related Schwarz preconditioner. A subproblem $A_{i}^\text{CBD}$ contains $\bigO(N/2^d)$ entries in $A$, whereas a subproblem $A_{i}^\text{Schwarz}$ contains only $\bigO(1)$ entries. Therefore, we expect the CBD preconditioner to be a better preconditioner. Compare $\lambda_{\min}(T^{-1}_{\text{CBD}} A)$ in \cref{t:cbd} to $\lambda_{\min}(T^{-1}_{\text{Schwarz}} A)$ in \cref{t:schwarz}. We state the following:

\begin{conjecture} %[$\lambda_{\min}(T^{-1}_{\text{CBD}} A) \ge C > 0$]
Let $A$ be the discretized integral operator in \cref{e:a_2d} or \cref{e:a_3d}. Given the overlapping domain decomposition \cref{e:second}, the CBD preconditioner $T^{-1}_{\text{CBD}}$ is defined in \cref{e:ad}. Then,
\[
\lambda_{\min}(T^{-1}_{\text{CBD}} A) \ge C > 0,
\]
where $C$ is a positive constant.
%Given the overlapping decomposition in \cref{s:cbd},
%the minimum eigenvalue of the CBD preconditioned matrix is low bounded by a constant $C$ independent of the problem size $N$ (the number of partitions $M = \bigO(N)$).
\end{conjecture}

\subsection{Recursive skeletonization (RS) as subdomain solver} \label{ss:subdomain}

To construct the CBD preconditioner, we apply the RS algorithm to compute an approximate factorization of every subproblem $A_i$ in the CBD preconditioner. The resulting factorizations consist of block triangular factors, which can be inverted easily for  applying  the  preconditioner.
 Suppose the subproblem $A_i$ is associated with the subdomain 
\begin{equation} \label{e:oi}
\Omega_i = \bigcup_{k=1}^{\tilde{M}} \tilde{\mathcal{P}}_k,
\end{equation} 
where $\tilde{M} = M/D = M/2^d$, and $\tilde{\mathcal{P}}_k$'s are spatially distant regions of the same color; see an example in \cref{sf:overlap}. To simplify notations, we drop the subscript $i$ in $A_i$ for the rest of this section, and we refer interested readers to~\cite{ho2016hierarchical} for more details. 

%Without loss of generality, we drop the subscript $i$ associated with a subproblem and describe the RS algorithm for matrix $A$ corresponding to $\Omega = \bigcup \tilde{\mathcal{P}}_k$

Without loss of generality, write
\begin{equation} \label{e:block}
A =
\left(
\begin{array}{cc}
A_{pp} & A_{pq}  \\ 
A_{qp} & A_{qq}
\end{array}
\right),
\end{equation}
where rows/columns indices $p$ and $q$ correspond to $\tilde{\mathcal{P}}_1$ and the union of remaining regions in \cref{e:oi}, respectively. Assume $A_{qp}$ is numerically low-rank, we compute the following interpolative decomposition (ID)~\cite{cheng2005compression} :
\begin{equation} \label{e:id}
A_{qp} = 
\left(
\begin{array}{cc}
A_{qr} & A_{qs}  
\end{array}
\right)
\approx
A_{qs} 
\left(
\begin{array}{cc}
T_p & I
\end{array}
\right),
\end{equation}
where the index set $p$ is split into a \emph{redundant} subset $r$ and a \emph{skeleton} subset $s$ (up to a permutation) and
\begin{equation} \label{e:id_err}
\|A_{qr} - A_{qs}T\| \le \varepsilon \, \|A_{qr}\|
\end{equation}
for a prescribed (relative) accuracy $\varepsilon$. We call the indices in $s$ skeletons, which tend to be discretization points close to the boundary of $\tilde{\mathcal{P}}_1$. See a pictorial illustration in \cref{f:scbd}. Define the numerical rank $k = |s|$. Since $\tilde{\mathcal{P}}_1$ is well-separated from the remaining regions in \cref{e:oi}, we know
\begin{equation} \label{e:k1}
k = \bigO(1),
\end{equation}
according to standard fast multipole estimates~\cite{greengard1987fast,greengard1996new}. Computing the ID requires $\bigO(k \, |p| \, |q|)$ operations in general but can be accelerated to $\bigO(k \, |p|^2)$ operations using the so-called proxy trick~\cite{ying2004kernel,martinsson2019fast}. %(The radius of the proxy surface is set to 1.5 as suggested in~\cite{ho2016hierarchical}.)

%In general, the concatenation of $r$ and $s$ equals to $p$ up to a permutation. \cref{e:id} is known as the interpolative decomposition (ID). It can be computed using the column-pivoting QR~\cite{cheng2005compression}, which was observed to be stable in practice although counter examples can be constructed in theory. If $A_{qp} \in \mathbb{R}^{(N-N/M) \times N/M}$ is formed explicitly, computing \cref{e:id} requires $\bigO(k N^2/M)$ operations, where $k=|s|$ is the numerical rank. Fortunately, an acceleration scheme known as the  proxy trick~\cite{ying2004kernel,martinsson2019fast}  avoids forming the entire  $A_{qp}$ explicitly and  reduces the computation cost to $\bigO(k N^2/M^2)$ operations. 

With \cref{e:id}, we have the approximation
\[
A \approx
\left(
\begin{array}{ccc}
A_{rr} & A_{rs} & T_p^\top A_{sq}  \\ 
A_{sr} & A_{ss} & A_{sq}  \\ 
A_{qs} T_p & A_{qs} & A_{qq}
\end{array}
\right),
\]
so we can eliminate the original $A_{rq}$ and $A_{qr}$ blocks without affecting the $A_{qq}$ block:
\begin{equation} \label{e:al}
L_p^\top A L_p \approx
\left(
\begin{array}{ccc}
B_{rr} & B_{rs} &   \\ 
B_{sr} & A_{ss} & A_{sq}  \\ 
 & A_{qs} & A_{qq}
\end{array}
\right),
\quad
L_p = 
\left(
\begin{array}{ccc}
I \\
- T_p & I \\
&& I
\end{array}
\right)
\end{equation}
where 
\begin{align*}
B_{rr} &= A_{rr} -  A_{rs}T_p - T_p^\top A_{sr} + T_p^\top A_{ss} T_p, \\
B_{sr} &= B_{rs}^\top = A_{sr} -  A_{ss}T_p.
%\quad B_{rs} = B_{sr}^\top.
\end{align*}
Let $B_{rr} = G_r^\top G_r$ be its Cholesky factorization, and we have the following approximate partial factorization
\begin{equation} \label{e:alu}
U_r^\top L_p^\top A L_p U_r \approx
\left(
\begin{array}{ccc}
I & &  \\ 
 & B_{ss} & A_{sq}  \\ 
 & A_{qs} & A_{qq}
\end{array}
\right),
\quad
U_r = 
\left(
\begin{array}{ccc}
G_r^{-1} & - B_{rr}^{-1} B_{rs} \\
& I \\
&& I
\end{array}
\right),
\end{equation}
where
\begin{equation} \label{e:b}
B_{ss} = A_{ss} - B_{sr} B_{rr}^{-1} B_{rs}.
\end{equation}
Notice that the $A_{qq}$, $A_{sq}$, and $A_{qs}$  blocks have not been modified. Define $R_1 = L_{p} U_{r}$ corresponding to the first region $\tilde{\mathcal{P}}_1$ in \cref{e:oi}, and we call $R_1$ an ``compress-then-eliminate" operator.

We continue the above ``compress-then-eliminate" step for the remaining block rows/columns corresponding to each of the remaining regions in \cref{e:oi}. The resulting Schur complement is a block matrix, where off-diagonal blocks remain sub-matrices in $A$. To be precise, suppose the indices of every block row/column are split into a redundant subset $r_i$ and a skeleton subset $s_i$ for $i=1,2,\ldots, \tilde{M}$, then there exists a permutation matrix $Q$ such that 
%\begin{align}
%& Q^\top \left(R_{\tilde{M}}^\top \ldots R_2^\top R_1^\top A R_1  R_2 \ldots R_{\tilde{M}}\right) \,Q \notag \\
%\approx &
%\left(
%\begin{array}{cc}
%I \\
%& S
%\end{array}
%\right)
%\equiv
%\left(
%\begin{array}{ccccc}
%I & &  \\ 
% & B_{s_1 s_1} & A_{s_1 s_2} & \ldots & A_{s_1 s_{\tilde{M}}} \\ 
% & A_{s_2 s_1} & B_{s_2 s_2} & \ldots & A_{s_2 s_{\tilde{M}}} \\
% & \vdots & \vdots &  & \vdots \\
% & A_{s_{\tilde{M}} s_1} & A_{s_{\tilde{M}} s_2} & \ldots & B_{s_{\tilde{M}} s_{\tilde{M}}} 
%\end{array}
%\right),
%% \notag \\
%%S  = 
%%\left(
%%\begin{array}{cccc}
%% B_{s_1 s_1} & A_{s_1 s_2} & \ldots & A_{s_1 s_{\tilde{M}}} \\ 
%% A_{s_2 s_1} & B_{s_2 s_2} & \ldots & A_{s_2 s_{\tilde{M}}} \\
%% \vdots & \vdots &  & \vdots \\
%% A_{s_{\tilde{M}} s_1} & A_{s_{\tilde{M}} s_2} & \ldots & B_{s_{\tilde{M}} s_{\tilde{M}}} 
%%\end{array}
%%\right).
%\end{align}
\begin{equation} \label{e:schur}
%Q^\top (R_M^\top \ldots R_2^\top R_1^\top A R_1  R_2 \ldots R_M) Q \approx
Q^\top (R_{\tilde{M}}^\top \ldots R_2^\top R_1^\top A R_1  R_2 \ldots R_{\tilde{M}}) Q \approx
\left(
\begin{array}{ccccc}
I & &  \\ 
 & B_{s_1 s_1} & A_{s_1 s_2} & \ldots & A_{s_1 s_{\tilde{M}}} \\ 
 & A_{s_2 s_1} & B_{s_2 s_2} & \ldots & A_{s_2 s_{\tilde{M}}} \\
 & \vdots & \vdots &  & \vdots \\
 & A_{s_{\tilde{M}} s_1} & A_{s_{\tilde{M}} s_2} & \ldots & B_{s_{\tilde{M}} s_{\tilde{M}}} 
\end{array}
\right),
\end{equation}
where the first diagonal block has size $\sum r_i$, and every $B_{s_i s_i}$ is computed analogously to \cref{e:b}. \cref{f:scbd} shows a pictorial illustration of the row/column indices in the original subproblem $A_i$ and the Schur complement in \cref{e:schur}.

To continue factorizing the Schur complement (approximately), we need a coarse partitioning of the domain $\Omega_i$. It can be obtained by merging spatially close regions in \cref{e:oi}. With a coarse overlapping domain, we repeat the above process recursively until only one coarse partition remains. As the last step, we factorize the remaining Schur complement  with Cholesky factorization directly. 
In practice, we compute a hierarchical partitioning of the problem domain $\Omega$, where the decomposition \cref{e:oi} is  at the first level and the last Schur complement factorized with Cholesky factorization is at the last level. For convenience, we make the following definition:

\begin{definition} \label{d:dof}
The degrees of freedom (DOFs) at every level in a hierarchical partitioning of the problem domain $\Omega$ are the discretization points that have not been processed in the RS algorithm. In particular, we have

\begin{itemize}
\item
DOFs at the first level: discretization points corresponding to $A$ in \cref{e:block};

\item
DOFs at the second level: discretization points corresponding to the Schur complement in \cref{e:schur};

\item
DOFs at the last level: discretization points corresponding to the last Schur complement factorized with Cholesky factorization.

\end{itemize}
\end{definition}

\begin{figure}
     \centering
     \hfill
     \begin{subfigure}[b]{0.45\textwidth}
         \centering
         \includegraphics[width=0.67\textwidth]{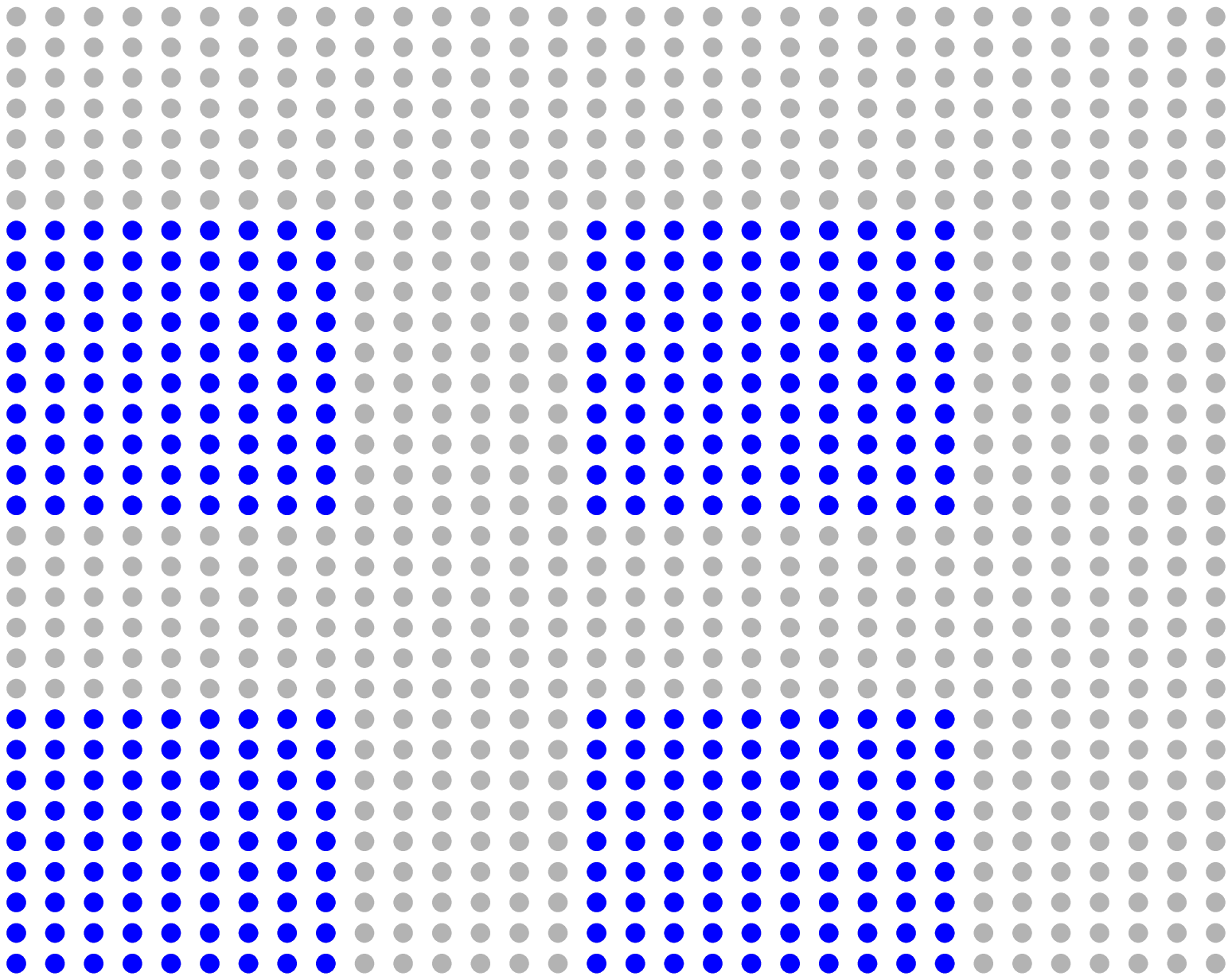}
         \caption{ DOFs at the first level.}
     \end{subfigure}
     \hfill
     \begin{subfigure}[b]{0.45\textwidth}
         \centering
         \includegraphics[width=0.67\textwidth]{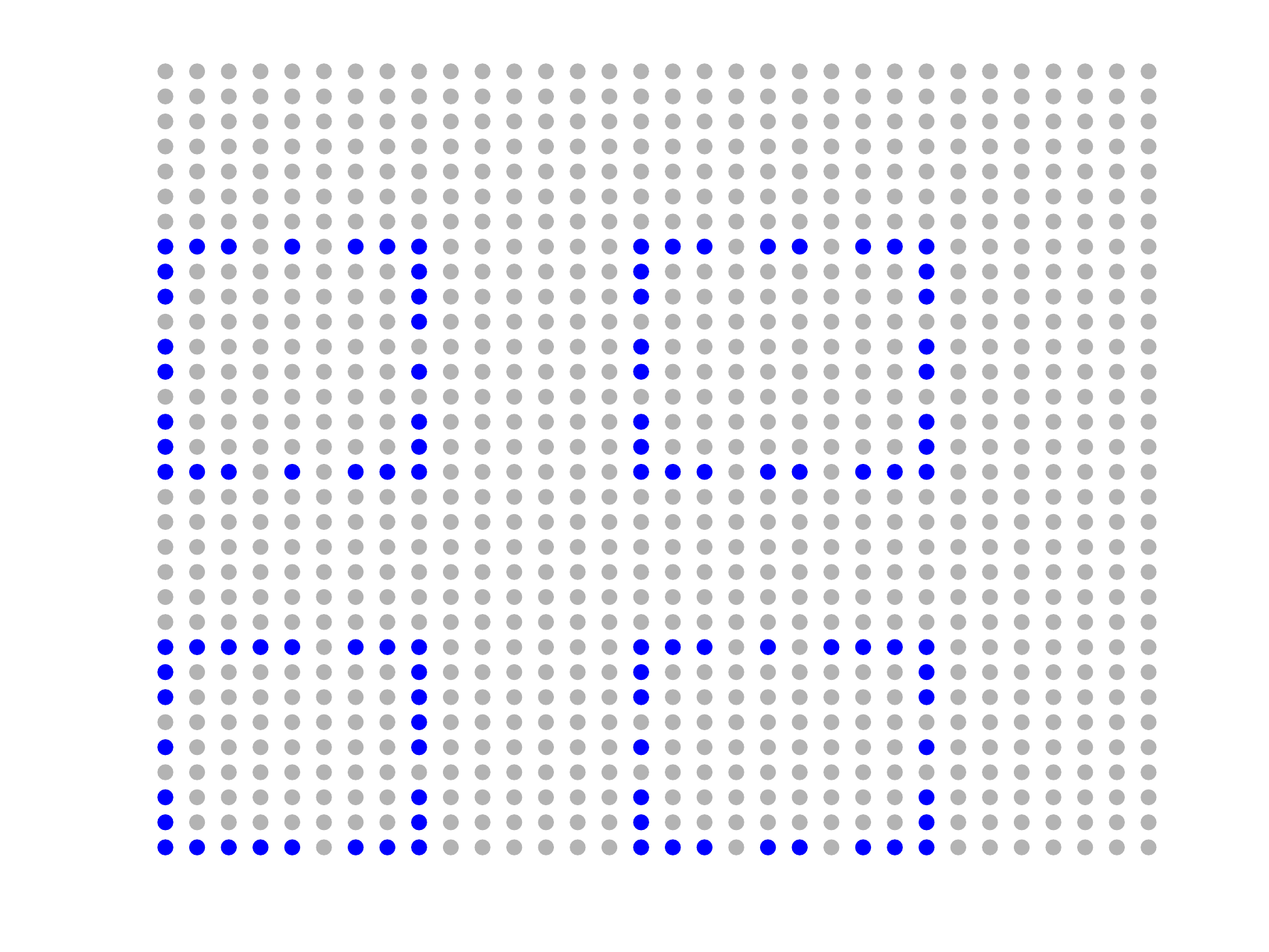}
         \caption{ DOFs at the second/last level.}
         \label{sf:sk}
     \end{subfigure}
     \hspace*{\fill}
     \caption{\em Illustration of  applying the RS to a subproblem in the CBD preconditioner ($N=32^2$, $M=4^2$, and $\varepsilon=10^{-6}$ in ID). %There are 107 DOFs in (b).
     %The problem setup is the same as in \cref{f:rs} . 
     %(a) Discretization grid over a subdomain consisting of four spatially distant regions; and (b) 107 skeletons over four regions corresponding to the Schur complement in \cref{e:schur}.}
     }
     \label{f:scbd}
\end{figure}

Assume in \cref{e:id} the numerical rank $|s_i| \equiv k$ is a constant for all $i$. The number of DOFs at the second level is
\begin{equation} \label{e:sk}
%$
\sum_{i=1}^{\tilde{M}} s_i = k \tilde{M},
%$
\end{equation}
which  immediately leads to the following:

\begin{theorem} \label{th:sk2}
%In the CBD preconditioner, the number of skeletons corresponding to the (first) Schur complement in \cref{e:schur} 
In the CBD preconditioner, the number of DOFs at the second level is $\bigO(M)$, independent of the original problem size $N$.
\end{theorem}

Recall that $M=D \tilde{M} = 2^d \tilde{M}$ is the number of partitions of the entire domain $\Omega$. Here, we do not require $M=\bigO(N)$ as in \Cref{s:schwarz}. In fact, we show how to choose $M$ to achieve optimal asymptotic scaling in \Cref{s:fast}.

\subsection{Comparison to the original RS factorization} \label{ss:compare}

The original RS factorization was applied to matrix $A$ in \cref{e:axb} directly, as other FDS's are typically used. To that end, the algorithm described in \Cref{ss:subdomain} remains the same except that we replace \cref{e:oi} with a partitioning (non-overlapping decomposition) of the entire problem domain as follows
\begin{equation} \label{e:part}
\Omega = \cup_{i=1}^M  \mathcal{P}_i, \quad \mathcal{P}_i \cap  \mathcal{P}_j = \emptyset, \text{ if } i\not=j.
\end{equation}
In particular, we view matrix $A$ as a block matrix as in \cref{e:block}, where $p$ and $q$ correspond to $\mathcal{P}_1$ and $\cup_{i=2}^M  \mathcal{P}_i$, respectively. Then, we compute the ID of  $A_{qp}$ as in \cref{e:id}. Since $\mathcal{P}_1$ is spatially adjacent to $\cup_{i=2}^M  \mathcal{P}_i$, we know the following according to Green's theorem:
\begin{equation} \label{e:rank}
%k = \bigO \left( (N/M)^{\frac{d-1}{d}} \right). \\ % = \bigO \left( (n/m)^{d-1} \right).
k = \bigO \left( (N/M)^{{(d-1)}/{d}} \right).
%k = \bigO \left( \left(\frac{N}{M}\right)^{{(d-1)}/{d}} \right). \\
\end{equation}
With a uniform partitioning, $\mathcal{P}_1$ contains $N/M$ discretization points, and the numerical rank $k$ scales as the perimeter or the surface area of $\mathcal{P}_1$ in 2D or 3D, respectively. Compare \cref{e:rank} to the numerical rank in \cref{e:k1} corresponding to applying the RS to a subproblem $A_i$ in the CBD preconditioner. %The increase of numerical rank means the original RS requires more work to compute the ID, whose cost scales linearly with $k$. 
%(The decrease becomes smaller at coarse levels.) 

%and $\varepsilon$ is the relative precision. In general, the concatenation of $r$ and $s$ equals to $p$ up to a permutation. \cref{e:id} is known as the interpolative decomposition (ID). It can be computed using the column-pivoting QR~\cite{cheng2005compression}, which was observed to be stable in practice although counter examples can be constructed in theory. If $A_{qp} \in \mathbb{R}^{(N-N/M) \times N/M}$ is formed explicitly, computing \cref{e:id} requires $\bigO(k N^2/M)$ operations, where $k=|s|$ is the numerical rank. Fortunately, an acceleration scheme known as the  proxy trick~\cite{ying2004kernel,martinsson2019fast}  avoids forming the entire  $A_{qp}$ explicitly and  reduces the computation cost to $\bigO(k N^2/M^2)$ operations. 

After $A_{qp}$ is compressed, we follow \cref{e:al,e:alu}, and we repeat the ``compress-then-eliminate" step for each of the remaining block rows/columns and obtain the Schur complement as in \cref{e:schur}. Assuming the numerical rank $k$ is a constant for all blocks, the resulting Schur complement has size $k M$. To factorize the Schur complement, we construct a coarse partitioning of $\Omega$ by merging adjacent partitions in \cref{e:part} and recurse until a single partition is left. Finally, we apply Cholesky factorization to the last Schur complement directly. \cref{f:rs} shows a pictorial illustration of the entire algorithm.

\begin{figure}
     \centering
     \hfill
     \begin{subfigure}[b]{0.32\textwidth}
         \centering
         \includegraphics[width=\textwidth]{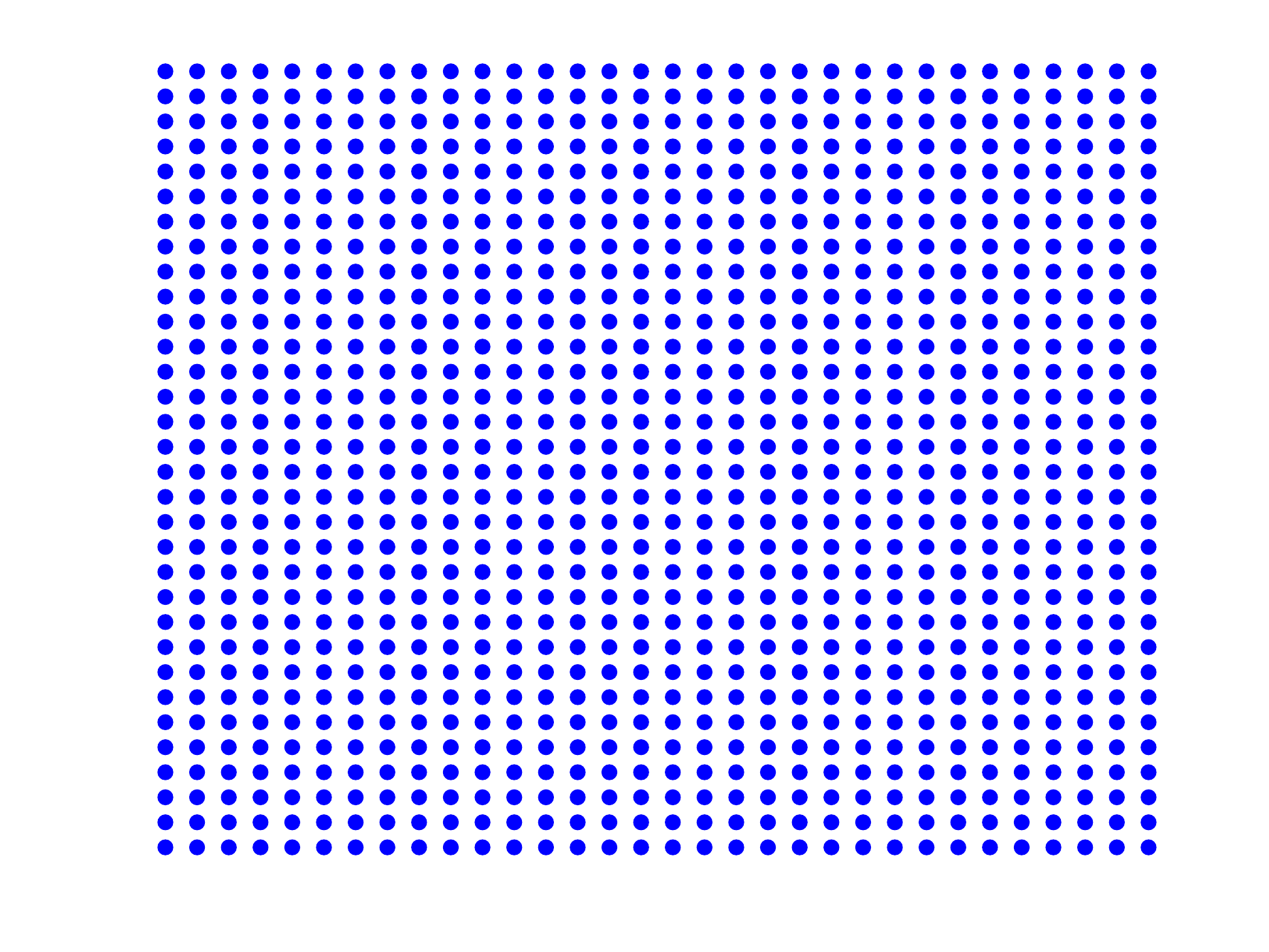}
         \caption{ DOFs at the first level.}
     \end{subfigure}
     \hfill
     \begin{subfigure}[b]{0.32\textwidth}
         \centering
         \includegraphics[width=\textwidth]{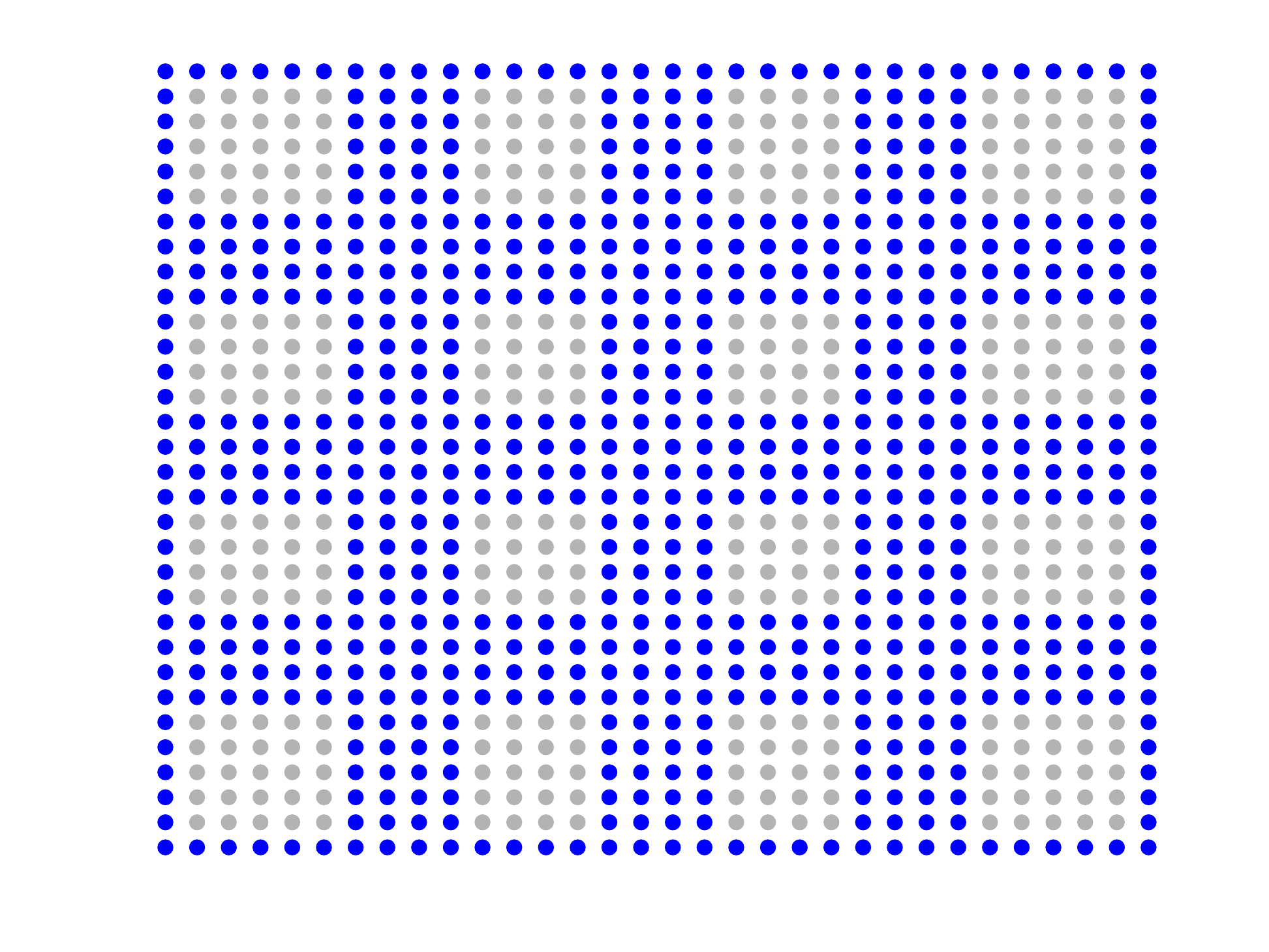}
         \caption{ DOFs at the second level.}
     \end{subfigure}
     \hfill
     \begin{subfigure}[b]{0.32\textwidth}
         \centering
         \includegraphics[width=\textwidth]{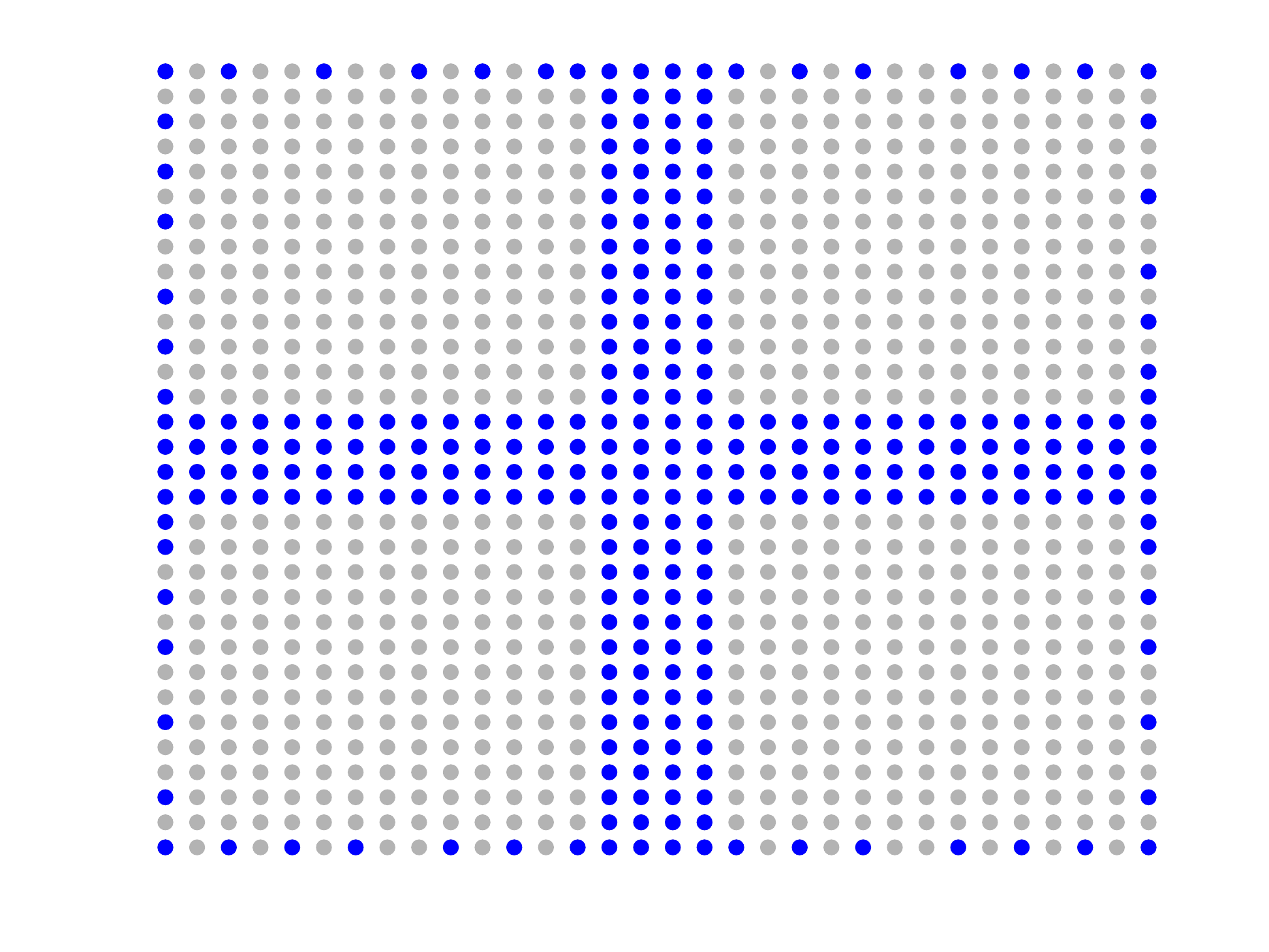}
         \caption{ DOFs at the last level.}
         \label{sf:sk_rs}
     \end{subfigure}
     \hfill
     \caption{\em Illustration of applying RS to $A$ ($N=32^2$, $M=4^2$, and $\varepsilon=10^{-6}$ in ID). %There are 292 DOFs in (c).
     %(a) Discretization grid over the entire problem domain; (b) skeletons at the coarse level associated with the Schur complement in \cref{e:schur}; and (c) 292 skeletons in the last partition.
     %(Compare to \cref{sf:sk}.)
     }
     \label{f:rs}
\end{figure}

Let $S$ be the number of DOFs at the last level, and we have the following~\cite{ho2016hierarchical}:
\begin{equation} \label{e:s}
S = \bigO \left( N^{(d-1)/d} \right).
\end{equation}
%see an example in \cref{sf:sk_rs}. 
Since work required to factorize the last Schur complement typically dominates that of the RS algorithm, we have the following~\cite{ho2016hierarchical}:%state the computational complexities and the storage requirement of the RS algorithm:
%The factorization cost of the RS solver is dominated by the work required for factorizing the Schur complement at the coarsest level directly using Cholesky factorization.
\begin{theorem} %[complexity of RS~\cite{ho2016hierarchical}] 
\label{th:rs}
The construction cost of applying the RS algorithm to $A$ in \cref{e:axb} is 
\[
t_f = \bigO \left( N^{3(d-1)/d} \right)
\]
for $d=2,3$, and storing/applying the approximate factorization requires
\[
m_f = t_a = 
\left\{
\begin{array}{lc}
\bigO \left(N \log N \right), & d=2, \\
\bigO \left( N^{4/3} \right), & d = 3.
\end{array}
\right.
\]
%Suppose we apply the RS solver to \cref{e:axb}. The following complexities holds:
%\begin{align*}
%t_f = \bigO \left( N^{3(d-1)/d} \right), \\
%m_f = t_a = \bigO \left( N^{2(d-1)/d} \right),
%\end{align*}
%where $d=2$ or 3 is the problem dimension. In particular, $t_f = \bigO(N^{3/2})$ and $t_f = \bigO(N^2)$ in 2D and in 3D, respectively.
\end{theorem}

Next, we derive results analogous to \cref{e:s,th:rs} for the CBD preconditioner.  According to  \cref{th:sk2}, the number of DOFs at the second level is $\bigO(M)$ in the CBD preconditioner. The DOFs corresponds to the Schur complement in \cref{e:schur}, to which we apply the RS algorithm recursively. If we apply \cref{e:s} to the Schur complement of size $\bigO(M)$, it is easy to see the following:

\begin{theorem} \label{th:sk}
%The number of skeleton points at the second level is $\bigO(M)$, independent of the original problem size $N$, in the CBD preconditioner. 
In the CBD preconditioner, the number of  DOFs at the last level is  (upper bounded by)
\[
S=\bigO \left( M^{(d-1)/d} \right).
\]
%which is the size of the last Schur complement (factorized exactly with  Cholesky factorization).
\end{theorem}
To construct the CBD preconditioner, we apply the RS algorithm to every subproblem $A_i$ for $i=1,2,\ldots,2^d$. In general, suppose we have an FDS with $C N^\alpha$ construction cost, where $C$ is a constant. Then, we immediately have 
\begin{theorem} \label{th:cbd}
The construction cost of the CBD preconditioner is 
\[
2^d \cdot \tilde{C}  \left( \frac{N}{2^d} \right)^\alpha =  \frac{\tilde{C}}{C \cdot 2^{(\alpha-1)d}} \, C N^\alpha < C N^\alpha,
\]
where $d=2$ or $3$, $\tilde{C} < C$, and $\alpha>1$. In particular, $\alpha = 3(d-1)/d$ for the RS factorization.
\end{theorem}
%We remark that the reduction of construction cost comes from two factors:
%\begin{enumerate}
%\item
%%$\frac{1}{2^{(\alpha-1)d}}$. 
%%The original matrix $A$ of size $N$ is decomposed into $2^d$ subproblems of size $N/2^d$. Since 
%
%Applying an FDS (with super-linear cost) to $2^d$ problems of size $N/2^d$ leads to a reduction factor of $2^{(\alpha-1)d}$.
%
%\item
%%$\frac{\tilde{C}}{C}$.
%
%The decrease of off-diagonal rank leads to a reduction factor of $C/\tilde{C}$. Consider the number of DOFs at the last level, which typically determines the cost of the entire algorithm. 
%
%
%Compare \cref{th:sk} to \cref{e:s}
%
%In practice, ; see, e.g., numerical results in \cref{t:cbd_weak_3d}.
%
%%In particular, the number of DOFs at the last level is smaller, which typically determines the entire cost of the 
%%number of skeletons corresponding to the last Schur complement is smaller. Compare \cref{sf:sk} to \cref{sf:sk_rs}.
%
%\end{enumerate}
Here, the fact that $\tilde{C} < C$ comes from the decrease of off-diagonal rank from \cref{e:rank} to \cref{e:k1}. To summarize, \cref{t:compare} shows the key differences between the original RS algorithm and the CBD preconditioner.

\begin{table}
    \centering \small 
    \caption{\em Comparison between the original RS factorization~\cite{ho2016hierarchical} (applied to $A$ directly) and the CBD preconditioner (employing RS for every subproblem $A_i$). Results of RS are from \cref{e:rank}, \cref{e:s}, and \cref{th:rs}. Results of CBD are from \cref{e:k1}, \cref{th:sk,th:cbd}. The problem dimension $d=2$ or 3.}
    \label{t:compare}
    \centering
    \begin{tabular}{ccc} 
    \toprule
          &  RS    &  CBD    \\ \midrule 
       Off-diagonal rank at first level &   $\bigO \left( \left(N/M\right)^{\frac{d-1}{d}} \right)$    &   $\bigO(1)$   \\
        \# DOFs at last level  & $\bigO \left(N^{\frac{d-1}{d}} \right)$ & $\bigO \left(M^{\frac{d-1}{d}} \right)$  \\
       Construction cost &  $C N^{\frac{3(d-1)}{d}}$ & $ \frac{\tilde{C}}{2^{2d - 3}}  N^{\frac{3(d-1)}{d}}$  \\
%       Off-diagonal rank at first level & \cref{e:rank} & \cref{e:k1} \\
%       \# DOFs at last level & \cref{e:s} & \cref{th:sk} \\
%       Construction cost & \cref{th:rs} & \cref{th:cbd} \\
    \bottomrule
    \end{tabular}
\end{table}

%\begin{table}
%    \centering \small 
%    \caption{\em Comparison between the original RS factorization~\cite{ho2016hierarchical} (applied to $A$ directly) and the CBD preconditioner (employing RS for every subproblem $A_i$).}
%    \centering
%    \begin{tabular}{ccc} 
%    \toprule
%          &  RS    &  CBD    \\ \midrule 
%       %Off-diagonal rank $k$ at first level &   $\bigO \left( (N/M)^{\frac{d-1}{d}} \right)$    &   $\bigO(1)$   \\
%       %Schur complement size  & $\bigO(N^{\frac{d-1}{d}} M^{\frac{1}{d}})$ & $\bigO(M)$  \\
%       %Construction cost &  $C N^{3(d-1)/d}$ & $ \tilde{C} N^{3(d-1)/d} / 2^{2d - 3} \, (\tilde{C} < C)$  \\
%       Off-diagonal rank at first level & \cref{e:rank} & \cref{e:k1} \\
%       \# DOFs at last level & \cref{e:s} & \cref{th:sk} \\
%       Construction cost & \cref{th:rs} & \cref{th:cbd} \\
%    \bottomrule
%    \end{tabular}
%\end{table}

\begin{remark} \label{rmk:schwarz}
We can fix $D = 2^d$ and employ an FDS as the subdomain solver in the Schwarz preconditioner. Following the analysis above, it is obvious that the associated construct cost is $C N^\alpha / 2^{(\alpha-1)d}$. In other words, the cost is higher by at least a factor of $C/\tilde{C}$ than that of the CBD preconditioner.
\end{remark}

\subsection{Asymptotically faster algorithm} \label{s:fast}

%When we employ the RS method as a subdomain solver to construct the CBD preconditioner, we reduce the constant of the original RS solver. 

Unlike that in the RS method,  the number of DOFs at the last level in the CBD preconditioner depends  on only the number of  partitions $M$ according to  \cref{th:sk}. {In this section, we show that the construction of the CBD preconditioner can be  asymptotically accelerated if we scale $M$ sublinearly to the problem size $N$.}

Let us revisit the construction of the CBD preconditioner. In particular, we decompose the factorization cost into two terms: one for the first level and the other for all remaining levels. As previously mentioned, we apply the ``compress-then-eliminate'' step for $\bigO(M)$ times at the first level, where computing an ID and the subsequent elimination both require $\bigO\left((N/M)^3\right)$. If we plug in the asymptotic complexities of the RS in \cref{th:rs} for the $\bigO(M)$ DOFs at the second level, we obtain the following:

\begin{theorem}
Suppose we apply the RS solver to every subdomain in the CBD preconditioner. Define $P \equiv \bigO(N/M)$.
%Let $P$ be the number of points in every partition, i.e., . 
The following complexities holds:
\begin{equation} \label{e:complexity}
t_f = \bigO \left( P^3 M + M^{3(d-1)/d} \right),
%\left\{
%\begin{array}{lc}
%\bigO \left( P^3 M + M^{3/2} \right), & d=2 \\
%\bigO \left( P^3 M + M^{2} \right), & d = 3
%\end{array}
%\right.
%\hfill
\quad
m_f = t_a = 
\left\{
\begin{array}{lc}
\bigO \left( P^2 M + M \log M \right), & d=2 \\
\bigO \left( P^2 M + M^{4/3} \right), & d = 3
\end{array}
\right.
\end{equation}
%\begin{align} 
%\label{e:complexity1}
%t_f = \bigO \left( P^3 M + M^{3/2} \right), \quad m_f = t_a = \bigO \left( P^2 M + M \log(M) \right),
%\end{align}
%in 2D, and 
%\begin{align} \label{e:complexity2}
%t_f = \bigO \left( P^3 M + M^{2} \right), \quad m_f = t_a = \bigO \left( P^2 M + M^{4/3} \right),
%\end{align}
%in 3D.
\end{theorem}
We note that the number of partitions $M$ is a parameter in \cref{e:complexity}, so we can optimize the complexities to obtain asymptotically more efficient methods than the RS solver as the following corollary shows:

\begin{corollary} \label{th:optimal}
In 2D, let $M=\bigO(N^{6/7})$, and we have
\begin{align*} 
t_f = \bigO \left( N^{9/7} \right), \quad
m_f = t_a = \bigO \left( N^{8/7} \right).
\end{align*}
In 3D, let $M=\bigO(N^{3/4})$, and we have
\begin{align*} 
t_f = \bigO \left( N^{3/2} \right), \quad
m_f = t_a = \bigO \left( N^{5/4} \right),
\end{align*}
\end{corollary}
Similarly, we can optimize the memory $m_f$ (or equivalently $t_a$) in \cref{e:complexity} with respect to $M$. In general, suppose we have an FDS of work complexity $\bigO(N^\alpha)$. \cref{th:optimal} can be generalized to the following:

\begin{theorem}
Suppose we have an FDS with construction cost $\bigO(N^\alpha)$ when applied to $A$ in \cref{e:axb}. The optimal complexity to construct the CBD preconditioner employing the FDS as a subdomain solver is
\[
\min_M \bigO \left( P^3M + M^\alpha \right) =\bigO \left( N^{3\alpha/(2+\alpha)} \right) < \bigO \left( N^\alpha \right),
\]
where the minimum is obtained when $M = N^{{3}/{(\alpha+2)}}$.
\end{theorem}

\section{Numerical results} \label{s:result}

In this section, we benchmark the CBD preconditioner and compare it to the RS factorization~\cite{ho2016hierarchical} on problems in 2D and in 3D. 
Recall the problem and the domain decompositions in \Cref{s:setup}. 
The RS solver was applied to solve the global problem \cref{e:axb} directly, while  the CBD preconditioner employed the RS solver for every subdomain (4 and 8 subdomains in 2D and in 3D, respectively). We fixed $\varepsilon = \num{1e-3}$ in \cref{e:id_err} as the relative accuracy of low-rank approximations computed with the ID\footnote{The radius of the proxy surface used to accelerate the computation of an ID  was set to 1.5 as suggested in~\cite{ho2016hierarchical}.}.
We used the PCG to solve \cref{e:axb}, where the matrix-vector (matvec) product with $A$ was carried out via the fast Fourier transform (FFT). In the general case where the discretization grid is non-uniform, the FMM could be used for fast matvec. Below are the notations we used to report results of our experiments (timing and storage are in seconds and in GB, respectively):

\begin{itemize}
\item
$N$: problem size/matrix size/number of discretization points;
\item
$M = \bigO(N)$: number of partitions of the problem domain in \cref{e:dd};
\item
$S$: number of DOFs at the last level in the RS solver and that associated with \emph{one} subproblem in the CBD preconditioner ($2^d$ subproblems in total);
%\item
%$\delta$: overlapping size (see \cref{f:subdomain}).
\item
$t_f$: time of factorization to construct a preconditioner;
\item
$m_f$: storage of a preconditioner;
\item
$t_s$: time of applying a preconditioner;
\item
$t_{pcg}$: total PCG time for solving \cref{e:axb} {(excluding $t_f$)};
\item
$n_{it}$: number of PCG iterations to reach a relative residual of $10^{-12}$.
%\item
%$res$: relative residual, i.e., $\|b-Ax_\text{solve}\|/\|b\|$;
%\item
%$err$: relative error, i.e., $\|x_\text{true}-x_\text{solve}\|/\|x_\text{true}\|$;
\end{itemize}
%(Number of subdomains $D=M$ in the Schwarz preconditioner, while $D=4$ and 8 in the CBD preconditioner in 2D and 3D, respectively.)

All experiments were performed with MATLAB\textsuperscript{\textregistered} R2020a on an Intel Xeon Platinum 8280M\footnote{\url{https://frontera-portal.tacc.utexas.edu/user-guide/system/\#large-memory-nodes}} (``Cascade Lake") that has 2.1 TB of memory and 112 cores on four sockets (28 cores/socket). Results in \Cref{ss:2d} corresponding to solving problems in 2D were obtained using one thread, since intermediate matrices are relatively small that associated computation benefits little from multithreading. Results in \Cref{ss:3d} corresponding to solving problems in 3D were obtained using the default multithreading in MATLAB, where the maximum number of computational threads equals to the number of physical cores. 

%(MATLAB command \texttt{maxNumCompThreads(1)})

In \Cref{ss:iter}, we show the convergence of the na\"ive CBD preconditioner, where subproblems are solved exactly (with round-off errors) without low-rank compression. In \Cref{ss:2d,ss:3d}, we show the practical performance of the CBD preconditioner described in \Cref{s:cbd}.

\subsection{Na\"ive CBD preconditioner} \label{ss:iter}
%\subsection{Comparison to block Jacobi preconditioner} \label{ss:iter}

{We benchmarked the number of PCG iterations required by the na\"ive CBD preconditioner, where subproblems were factorized exactly (with round-off errors) using Cholesky factorization rather than the RS method.} 
For comparison, we fixed the number of subdomains $D=2^d$ in the Jacobi preconditioner and the Schwarz preconditioner to be the same as in the CBD preconditioner. Recall that the domain decomposition associated with the block Jacobi preconditioner is non-overlapping but those with the other two preconditioners are overlapping. As \cref{t:it_m48} shows, the number of PCG iterations was almost constant with the two preconditioners based on overlapping domain decompositions. By contrast, the Jacobi preconditioner led to increasing PCG iterations as the problem size increased.

In addition, the number of PCG iterations required by the Schwarz method and the na\"ive CBD preconditioner were almost the same. {For the (practical) CBD preconditioner, we apply the RS method leveraging low-rank compression to every subproblem, which leads to efficient running time and slightly higher PCG iterations. As \cref{rmk:schwarz} mentioned, we can also apply the RS method to speedup the Schwarz preconditioner, but it will not be as efficient as the (practical) CBD preconditioner.}

%All subproblems in the three preconditioners were factorized using Cholesky factorization, so the construction cost was $\bigO(N^3)$ for all preconditioners. 

\begin{table}
    \centering \small 
    \caption{\em Number of PCG iterations $n_{it}$ with the Jacobi preconditioner, the Schwarz preconditioner and the CBD preconditioner (without low-rank compression) in 2D (left) and in 3D (right). The number of partitions $M=D=2^d$ is fixed for the block Jacobi and the Schwarz preconditioners. For the CBD preconditioner, we used $M=N/16$ and $M=N/8$ in 2D and in 3D, respectively.}
    \label{t:it_m48}
    \begin{tabular}{ccccc} 
    \toprule
      $N$   &  $D$ & Jacobi & Schwarz & CBD \\ \midrule 
      $16^2$ & $4$ & 34 & 18  & 18 \\
      $32^2$ & $4$ & 50 & 19  & 19 \\
      $64^2$ & $4$ & 72 &  19 & 20 \\
      $128^2$ & $4$ & 102 &  20 & 20 \\
    \bottomrule
    \end{tabular}
    %\hspace{1cm}
    \quad\quad
    \begin{tabular}{ccccc} 
    \toprule
      $N$   &  $D$ & Jacobi & Schwarz & CBD \\ \midrule 
      $8^3$ & $8$ & 33 & 26  & 27 \\
      $16^3$ & $8$ & 50 & 27  & 27 \\
      $32^3$ & $8$ & 72 &  29 & 29 \\
      %$64^3$ & $8$ & 64 &  27 & 30 \\
    \bottomrule
    \end{tabular}
%    %
%    \centering \small 
%    \caption{\em PCG iterations $n_{it}$ with the Jacobi preconditioner, the Schwarz preconditioner and the CBD preconditioner in 3D.}
%    \label{t:it_m8}
%    \begin{tabular}{ccccc} 
%    \toprule
%      $N$   &  $M$ & Jacobi & Schwarz & $\text{CBD}^*$ \\ \midrule 
%      $8^3$ & $8$ & 33 & 26  & 27 \\
%      $16^3$ & $8$ & 50 & 27  & 27 \\
%      $32^3$ & $8$ & 72 &  29 & 29 \\
%      %$64^3$ & $8$ & 64 &  27 & 30 \\
%    \bottomrule
%    \end{tabular}
\end{table}

\subsection{Two dimensions} \label{ss:2d}

\begin{table}
    %\centering \small 
    \caption{\em 2D results, where the problem size $N$ is fixed and the number of non-overlapping partitions $M$ varies. The matrix is defined in \cref{e:a_2d}.}
    %\label{t:}
    \begin{subtable}[h]{\textwidth}
    \centering
    \begin{tabular}{cccccccc} \toprule
        $N$ & $M$ & $S$ & $t_f$ & $m_f$ & $t_s$ & $n_{it}$ & $t_{pcg}$  \\ \midrule 
%$256^2$ & $16^2$ & 1047 & \num{1.52e+01} & \num{1.37e-01} & \num{1.13e-01}& 7 & \num{1.16e+00} \\ 
%$256^2$ & $32^2$ & 1046 & \num{8.68e+00} & \num{8.73e-02} & \num{1.36e-01}& 6 & \num{1.39e+00} \\ 
%$256^2$ & $64^2$ & 1047 & \num{8.74e+00} & \num{7.99e-02} & \num{2.94e-01}& 6 & \num{2.06e+00} \\ 
%
$256^2$ & $8^2$ & 1048 & \num{1.21e+02} & \num{3.50e-01} & \num{2.94e-01} & 6 & \num{3.22e+00} \\ 
$256^2$ & $16^2$ & 1043 & \num{1.43e+01} & \num{1.37e-01} & \num{1.20e-01} & 7 & \num{1.17e+00} \\ 
\rowcolor{Gray}
$256^2$ & $32^2$ & 1048 & \num{7.68e+00} & \num{8.73e-02} & \num{1.30e-01} & 6 & \num{1.07e+00} \\ 
$256^2$ & $64^2$ & 1048 & \num{8.62e+00} & \num{7.98e-02} & \num{3.04e-01} & 6 & \num{2.19e+00} \\ 
$256^2$ & $128^2$ & 1047 & \num{1.20e+01} & \num{7.98e-02} & \num{9.40e-01} & 6 & \num{5.94e+00} \\ 
    \bottomrule
    \end{tabular}
    \caption{\em RS preconditioner based on the domain partitioning \cref{e:dd}.}
    \label{t:rs_strong_2d}
    \end{subtable}
    \begin{subtable}[h]{\textwidth}
    \centering
    \begin{tabular}{cccccccc} \toprule
    $N$ & $M$ & $S$ & $t_f$ & $m_f$ & $t_s$ & $n_{it}$ & $t_{pcg}$  \\ \midrule 
%$256^2$ & $8^2$ & $256$ & \num{4.90e+00} & \num{3.39e-01} & \num{1.92e-01} & 26 & \num{6.24e+00} \\ 
%$256^2$ & $16^2$ & $464$ & \num{1.97e+00} & \num{1.13e-01} & \num{9.78e-02} & 26 & \num{3.54e+00} \\ 
%$256^2$ & $32^2$ & $844$ & \num{1.91e+00} & \num{5.53e-02} & \num{1.24e-01} & 24 & \num{3.91e+00} \\ 
%$256^2$ & $64^2$ & $1624$ & \num{4.33e+00} & \num{6.86e-02} & \num{3.44e-01} & 22 & \num{1.04e+01} \\ 
%
$256^2$ & $8^2$ & $64$ & \num{4.76e+00} & \num{3.39e-01} & \num{1.73e-01} & 26 & \num{5.71e+00} \\ 
$256^2$ & $16^2$ & $114$ & \num{2.08e+00} & \num{1.13e-01} & \num{8.18e-02} & 26 & \num{3.10e+00} \\ 
\rowcolor{Gray}
$256^2$ & $32^2$ & $213$ & \num{1.84e+00} & \num{5.53e-02} & \num{1.07e-01} & 24 & \num{3.76e+00} \\ 
$256^2$ & $64^2$ & $407$ & \num{5.20e+00} & \num{6.86e-02} & \num{3.27e-01} & 22 & \num{8.47e+00} \\ 
%$256^2$ & $128^2$ & $4175$ & \num{33.5e+00} & \num{3.16e-01} & \num{1.50e+00} & 16 & \num{24.9e+00} \\ 
    \bottomrule
    \end{tabular}
    \caption{\em CBD preconditioner based on the overlapping domain decomposition \cref{e:second} (four subdomains).}
    \label{t:cbd_strong_2d}
    \end{subtable}
\end{table}
%$256^2$ & $8^2$ & 1088 & \num{2.69e+01} & \num{3.53e-01} & \num{1.55e-01} & 5 & \num{2.32e+00} \\ 
%$256^2$ & $16^2$ & 1092 & \num{2.72e+00} & \num{1.40e-01} & \num{7.03e-02} & 5 & \num{6.04e-01} \\ 
%$256^2$ & $32^2$ & 1089 & \num{2.17e+00} & \num{9.02e-02} & \num{1.05e-01} & 5 & \num{8.01e-01} \\ 
%$256^2$ & $64^2$ & 1092 & \num{3.52e+00} & \num{8.64e-02} & \num{2.79e-01} & 5 & \num{1.64e+00} \\ 
%$256^2$ & $128^2$ & 1090 & \num{6.19e+00} & \num{8.64e-02} & \num{9.26e-01} & 5 & \num{4.93e+00} \\ 
%
%$256^2$ & $8^2$ & $256$ & \num{5.83e+00} & \num{3.39e-01} & \num{1.81e-01} & 26 & \num{7.42e+00} \\ 
%$256^2$ & $16^2$ & $456$ & \num{2.27e+00} & \num{1.13e-01} & \num{8.23e-02} & 26 & \num{3.15e+00} \\ 
%$256^2$ & $32^2$ & $851$ & \num{1.84e+00} & \num{5.53e-02} & \num{1.08e-01} & 24 & \num{3.60e+00} \\ 
%$256^2$ & $64^2$ & $1627$ & \num{4.61e+00} & \num{6.86e-02} & \num{3.24e-01} & 22 & \num{8.29e+00} \\ 
%$256^2$ & $128^2$ & $4175$ & \num{3.38e+01} & \num{3.16e-01} & \num{1.09e+00} & 16 & \num{1.86e+01} \\ 
%
% the extra skeleton size with too many partitions comes from the overlapping

\begin{table}
    %\centering \small 
    \caption{\em 2D results, where the problem size per partition, i.e., $N/M$, is fixed. The matrix is defined in \cref{e:a_2d}.}
    \label{t:weak_2d}    
    \begin{subtable}[h]{\textwidth}
    \centering
    \begin{tabular}{cccccccc} \toprule
        $N$ & $M$ & $S$ & $t_f$ & $m_f$ & $t_s$ & $n_{it}$ & $t_{pcg}$  \\ \midrule 
%$256^2$ & $32^2$ & 1089 & \num{2.40e+00} & \num{9.02e-02} & \num{1.10e-01} & 5 & \num{9.54e-01} \\ 
%$512^2$ & $64^2$ & 2105 & \num{1.25e+01} & \num{4.31e-01} & \num{4.52e-01} & 5 & \num{2.90e+00} \\ 
%$1024^2$ & $128^2$ & 4152 & \num{7.28e+01} & \num{2.03e+00} & \num{1.90e+00} & 6 & \num{1.56e+01} \\ 
%$2048^2$ & $256^2$ & 8248 & \num{4.94e+02} & \num{9.38e+00} & \num{8.16e+00} & 7 & \num{7.01e+01} \\ 
%$4096^2$ & $512^2$ & 16434 & \num{3.82e+03} & \num{4.27e+01} & \num{3.73e+01} & 8 & \num{4.01e+02} \\ 
%%
%%$256^2$ & $32^2$ & 1089 & \num{2.49e+00} & \num{9.02e-02} & \num{1.06e-01} & 5 & \num{8.91e-01} \\ 
%%$512^2$ & $64^2$ & 2105 & \num{1.35e+01} & \num{4.31e-01} & \num{4.47e-01} & 5 & \num{3.01e+00} \\ 
%%$1024^2$ & $128^2$ & 4152 & \num{7.54e+01} & \num{2.03e+00} & \num{1.86e+00} & 6 & \num{1.52e+01} \\ 
%%$2048^2$ & $256^2$ & 8248 & \num{4.94e+02} & \num{9.38e+00} & \num{1.38e+01} & 7 & \num{1.10e+02} \\ 
%%$4096^2$ & $512^2$ & 16434 & \num{3.97e+03} & \num{4.27e+01} & \num{4.30e+01} & 8 & \num{4.52e+02} \\ 
%%$8192^2$ & $1024^2$ & 19371 & \num{1.58e+04} & \num{1.66e+02} & \num{1.87e+02} & 13 & \num{3.13e+03} \\ 
%
%$256^2$ & $32^2$ & 1048 & \num{7.72e+00} & \num{8.73e-02} & \num{1.04e-01} & 6 & \num{9.98e-01} \\ 
%\rowcolor{Gray}
$512^2$ & $64^2$ & 2072 & \num{5.64e+01} & \num{4.25e-01} & \num{4.42e-01} & 7 & \num{4.22e+00} \\ 
$1024^2$ & $128^2$ & 4112 & \num{4.57e+02} & \num{2.01e+00} & \num{1.84e+00} & 8 & \num{2.01e+01} \\ 
%\rowcolor{Gray}
$2048^2$ & $256^2$ & 5713 & \num{1.67e+03} & \num{7.32e+00} & \num{7.40e+00} & 15 & \num{1.34e+02} \\ 
$4096^2$ & $512^2$ & 8897 & \num{8.05e+03} & \num{2.88e+01} & \num{3.17e+01} & 19 & \num{8.22e+02} \\ 
\rowcolor{Gray}
$8192^2$ & $1024^2$ & 15303 & \num{3.88e+04} & \num{1.05e+02} & \num{1.51e+02} & 27 & \num{5.44e+03} \\ 
$16384^2$ & - & - & - & - & - & - & -  \\
    \bottomrule
    \end{tabular}
    \caption{\em RS preconditioner based on the domain partitioning \cref{e:dd}.}
    \label{t:rs_weak_2d}    
    \end{subtable}
    \begin{subtable}[h]{\textwidth}
    \centering %\small 
    \begin{tabular}{cccccccc} \toprule
    $N$ & $M$ & $S$ & $t_f$ & $m_f$ & $t_s$ & $n_{it}$ & $t_{pcg}$  \\ \midrule 
%$256^2$ & $32^2$ & $1091$ & \num{2.40e+00} & \num{8.62e-02} & \num{1.19e-01} & 23 & \num{4.93e+00} \\ 
%$512^2$ & $64^2$ & $1776$ & \num{6.93e+00} & \num{3.55e-01} & \num{4.55e-01} & 24 & \num{1.40e+01} \\ 
%$1024^2$ & $128^2$ & $2741$ & \num{2.67e+01} & \num{1.43e+00} & \num{1.82e+00} & 24 & \num{5.61e+01} \\ 
%$2048^2$ & $256^2$ & $5384$ & \num{1.10e+02} & \num{5.51e+00} & \num{7.42e+00} & 24 & \num{2.14e+02} \\ 
%$4096^2$ & $512^2$ & $8459$ & \num{4.54e+02} & \num{2.17e+01} & \num{3.29e+01} & 27 & \num{1.23e+03} \\ 
%    $256^2$ & $844$ & \num{2.39e+00} & \num{5.53e-02} & \num{1.25e-01} & 24 & \num{4.12e+00} \\ 
%    $512^2$ & $1344$ & \num{7.91e+00} & \num{2.27e-01} & \num{4.53e-01} & 24 & \num{1.42e+01} \\ 
%    $1024^2$ & $2260$ & \num{3.66e+01} & \num{9.09e-01} & \num{1.79e+00} & 25 & \num{6.46e+01} \\ 
%    $2048^2$ & $4203$ & \num{1.57e+02} & \num{3.63e+00} & \num{7.12e+00} & 25 & \num{2.17e+02} \\ 
%    $4096^2$ & $8254$ & \num{7.89e+02} & \num{1.48e+01} & \num{3.36e+01} & 30 & \num{1.38e+03} \\ 
%
%$256^2$ & $32^2$ & $851$ & \num{2.15e+00} & \num{5.53e-02} & \num{1.17e-01} & 24 & \num{5.76e+00} \\ 
%\rowcolor{Gray}
$512^2$ & $64^2$ & $338$ & \num{8.94e+00} & \num{2.27e-01} & \num{4.37e-01} & 24 & \num{1.41e+01} \\ 
$1024^2$ & $128^2$ & $568$ & \num{3.55e+01} & \num{9.09e-01} & \num{1.77e+00} & 25 & \num{5.97e+01} \\ 
%\rowcolor{Gray}
$2048^2$ & $256^2$ & $1052$ & \num{1.56e+02} & \num{3.64e+00} & \num{7.34e+00} & 25 & \num{2.31e+02} \\ 
$4096^2$ & $512^2$ & $2066$ & \num{7.99e+02} & \num{1.48e+01} & \num{3.58e+01} & $25^*$ & \num{1.27e+03} \\ 
\rowcolor{Gray}
$8192^2$ & $1024^2$ & $4109$ & \num{3.91e+03} & \num{5.92e+01} & \num{1.80e+02} & $31^*$ & \num{7.52e+03} \\
$16384^2$ & $2048^2$ & $6759$ & \num{2.03e+04} & \num{2.37e+02} & \num{8.94e+02} & $28^*$ & \num{3.49e+04} \\ 
    \bottomrule
    \end{tabular}
    \caption{\em CBD preconditioner based on the overlapping domain decomposition \cref{e:second} (four subdomains). $^*$PCG stagnated at relative residuals \num{1.13e-12}, \num{1.71e-12}, and \num{3.99e-12}, respectively. }
    \label{t:cbd_weak_2d}
    \end{subtable}
\end{table}

%We recall that $M$ stands for the number of partitions of the problem domain, not the number of subdomains, for the RS solver and the CBD preconditioner. In addition, $S$ is the problem size at the coarsest level in the RS solver, which scales as $\bigO(N^{(d-1)/d})$~\cite{ho2012fast,ho2016hierarchical}. For the CBD preconditioner, $S$ is the sum of the coarsest problem sizes across all subdomains. 
%We summarize the two notations below:
%\begin{itemize}
%\item
%$M$: number of partitions of the problem domain;
%\item
%$S$: problem size at the coarsest level.
%\end{itemize}

%The RS method is based on domain partitioning \cref{e:dd}, where the number of partitions $M$ is a parameter. The CBD preconditioner is based on overlapping domain decomposition \cref{e:second}, where the number of partitions is also a parameter to be chosen. 

\cref{t:rs_strong_2d,t:cbd_strong_2d} show results of experiments with a fixed problem size but different numbers of partitions. Observe that $S$, the number of DOFs at the last level, is almost a constant for the RS method (see \cref{e:s}), whereas $S=\bigO(M^{1/2})$ for the CBD preconditioner (\cref{th:sk}). Notice that $S$, a deciding factor of the computational efficiency of both methods, is much smaller with the CBD preconditioner.

Results in \cref{t:rs_strong_2d,t:cbd_strong_2d} also indicate appropriate $M$ for the two methods. In \cref{t:rs_strong_2d}, results are similar between $M=32^2$ and $M=64^2$; so we chose $N/M=64$ in the RS solver for large problem sizes as done in~\cite{ho2016hierarchical,minden2017recursive}. 
In \cref{t:cbd_strong_2d}, $M=32^2$ led to half of the factorization time and the memory footprint than $M=16^2$; so  we also chose $N/M=64$ in the CBD preconditioner as well. %Recall \cref{th:sk} and observe that $S=\bigO(M^{1/2})$  in \cref{t:cbd_strong_2d}. For comparison, $S$ does not depend on $M$ in \cref{t:rs_strong_2d}.

\cref{t:rs_weak_2d} shows results for solving large problem sizes with the RS method, where the factorization time was always an order of magnitude larger than the PCG time. 
Due to an increasing condition number of the discretized integral operator defined in \cref{e:a_2d}, we observe that $S$  increased slower than the predicted $\bigO(N^{1/2})$ scaling; the factorization time $t_f$ scaled slower than the predicted $\bigO(N^{3/2})$ scaling, and the number of PCG iterations $n_{it}$ increased. 
(This phenomenon is known in~\cite{ho2016hierarchical}; see Table 5.3, 3rd row when $\varepsilon=\num{1e-3}$.)
%Although  $m_f$ and $t_s$ both followed the theoretical $\bigO(N\log N)$ trend approximately, 
The expected behaviors such as $S=\bigO(N^{1/2})$ can be obtained by employing  larger ranks in \cref{e:id} through either decreasing $\varepsilon$ in \cref{e:id_err} or reducing the radius of the proxy surface. Both require extra computation and  would lead to  longer factorization time.
%See~\cite{ho2016hierarchical,minden2017recursive} for higher accuracy and applications of the RS solver in solving other problems.

To compare with the RS solver, \cref{t:cbd_weak_2d} shows corresponding results of the CBD preconditioner: (1) The number of DOFs at the last level $S$ (for one subproblem) was much smaller. As a consequence, the factorization time $t_f$ of the CBD preconditioner was an order of magnitude smaller. (2) the storage $m_f$ of the CBD preconditioner was approximately half of that required by the RS solver. On a computer with 2.1 TB memory,  the CBD preconditioner could be used to solve problems as large as $N=16\,384^2 \approx 2.7\times 10^8$, for which the RS solver ran out of memory. (3) The CBD preconditioner required more iterations than the RS solver. Overall, the total running time of the CBD preconditioner was approximately a quarter of that taken by the RS solver.

In the CBD preconditioner, $S$, the number of DOFs at the last level, increased as $\bigO(N^{1/2})$ until $N$ reached $16384^2$, indicating that the ill-conditioning effect associated with the RS solver was mitigated. However, ill-conditioning still affected the convergence of the CBD preconditioner, preventing the PCG from reaching the relative residual \num{1e-12} when $N \ge 4096^2$. 
%In~\ref{t:cbd_weak_2d}, the factorization time of CBD initially scales sub-linearly due to the separation effect. This is also revealed in the strong scaling test in~\ref{t:cbd_strong_2d}.
%
{For environments where many right-hand sides need to be solved, the break-even point of the two preconditioners is $(t_f^{\text{RS}} - t_f^{\text{CBD}}) / (t_{pcg}^{\text{CBD}} - t_{pcg}^{\text{RS}})$, which evaluates to be about 16 for $N=8192^2$ in \cref{t:weak_2d}.}

%\subsection{3D problems}
\subsection{Three dimensions} \label{ss:3d}

We first show results of experiments with a fixed problem size but different numbers of partitions in \cref{t:rs_strong_3d,t:cbd_strong_3d}. Again, we observe results predicted by \cref{e:s,th:sk}: $S$ depended on $M$ but not on $N$ in the CBD preconditioner, whereas $S$ was independent of $M$ in the RS method. To solve large problems, we chose $N/M=8^3=512$ for both methods. Notice that this choice for the RS method is different from that in~\cite{ho2016hierarchical} because we used multithreading, which handled large matrices efficiently.

\cref{t:rs_weak_3d} shows results of the RS solver for large problem sizes, where the factorization time was at least two orders of magnitude longer than the PCG time. Observe that the numerical results basically match what \cref{e:s,th:rs} predict: $S = \bigO(N^{2/3})$, $t_f = \bigO(N^{2})$, and $m_f=t_s=\bigO(N^{4/3})$. In addition, the number of PCG iterations was almost constant.
%The number of DOFs at the last level, namely,  scaled as the predicted $$ complexity, leading to $\bigO(N^{2})$ factorization time, $\bigO(N^{4/3})$ storage, and $\bigO(N^{4/3})$ application time of the solver. 
It is also clear that the factorization time is prohibitive for a large problem size.

To compare with the RS solver, \cref{t:cbd_weak_3d} shows results of the CBD preconditioner, where $S$, the number of DOFs at the last level (for one subproblem), was at least an order of magnitude smaller. As a result, the CBD preconditioner required much less factorization time and storage. For example, compare the RS method to the CBD preconditioner for $N=128^3$: (1) the factorization time of the CBD preconditioner was approximately $50 \times$ smaller; (2) the storage of the CBD preconditioner was approximately $7 \times$ smaller; and (3) both methods required approximately the \emph{same} PCG time (CBD took more iterations but every iteration was faster). Empirically, the factorization time and the storage of the CBD preconditioner both scaled closely to $\bigO(N)$. This indicates that the estimates in \cref{th:sk,th:cbd} may be too pessimistic.

{For environments where many right-hand sides need to be solved, the break-even point of the two preconditioners is $(t_f^{\text{RS}} - t_f^{\text{CBD}}) / (t_{pcg}^{\text{CBD}} - t_{pcg}^{\text{RS}})$, which evaluates to be about 13 for $N=64^3$ and 2500 for $N=128^3$ in \cref{t:weak_3d}.}

%As the theory predicts, the scaling of the CBD preconditioner follows $t_f = \bigO(S^3)$,  $m_f = t_s = \bigO(S^2)$, and $n_{it}=\bigO(1)$. It is also obvious that the asymptotic complexity $S=\big(N^{2/3})$ was not reached yet when $N \le 256^3$. 
%Better scaling can be obtained employing $M=\bigO(N^{3/4})$ instead; see \Cref{s:fast}. 
%Since the condition number $\kappa(A) = \bigO(N^{2/3})$ increases relatively slower in 3D, the ill-conditioning effects that appeared in 2D were not observed here.

\begin{table}
    \caption{\em 3D results, where the problem size $N$ is fixed and the number of non-overlapping partitions $M$ varies. The matrix is defined in \cref{e:a_3d}.}
    %\label{t:}
    \begin{subtable}[h]{\textwidth}
    \centering %\small 
    \begin{tabular}{cccccccc} \toprule
        $N$ & $M$  & $S$ & $t_f$ & $m_f$ & $t_s$ & $n_{it}$ & $t_{pcg}$  \\ \midrule 
%$32^3$ & $4^3$ & 5981 & \num{1.45e+02} & \num{3.97e-01} & \num{1.28e-01} & 5 & \num{2.04e+00} \\ 
%$32^3$ & $8^3$ & 5980 & \num{1.34e+02} & \num{3.77e-01} & \num{1.44e-01} & 5 & \num{8.66e-01} \\ 
%$32^3$ & $16^3$ & 5983 & \num{1.33e+02} & \num{3.77e-01} & \num{3.06e-01} & 5 & \num{1.80e+00} \\ 
%
%$64^3$ & $8^3$ & 24086 & \num{1.31e+04} & \num{7.60e+00} & \num{4.68e+00} & 6 & \num{3.56e+01} \\ 
%$64^3$ & $16^3$ & 24085 & \num{1.35e+04} & \num{7.43e+00} & \num{2.56e+00} & 6 & \num{1.77e+01} \\ 
%$64^3$ & $32^3$ & 24086 & \num{1.32e+04} & \num{7.43e+00} & \num{3.68e+00} & 6 & \num{2.58e+01} \\ 
% randID
\rowcolor{Gray}
$32^3$ & $4^3$ & 5987 & \num{1.33e+01} & \num{3.97e-01} & \num{1.54e-01} & 5 & \num{8.55e-01} \\ 
$32^3$ & $8^3$ & 5982 & \num{1.60e+01} & \num{3.77e-01} & \num{1.92e-01} & 5 & \num{1.04e+00} \\ 
$32^3$ & $16^3$ & 5990 & \num{1.83e+01} & \num{3.78e-01} & \num{3.65e-01} & 5 & \num{1.91e+00} \\ 
    \bottomrule
    \end{tabular}
    \caption{\em RS preconditioner based on the domain partitioning \cref{e:dd}.}
    \label{t:rs_strong_3d}
    \end{subtable}
    \begin{subtable}[h]{\textwidth}	
    \centering %\small 
    \begin{tabular}{cccccccc} \toprule
    $N$ & $M$  & $S$ & $t_f$ & $m_f$ & $t_s$ & $n_{it}$ & $t_{pcg}$  \\ \midrule 
%$64^3$ & $4^3$ & $4063$ & \num{5.84e+02} & \num{7.55e+00} & \num{3.39e+00} & 34 & \num{1.19e+02} \\ 
%$64^3$ & $8^3$ & $9696$ & \num{1.94e+02} & \num{2.16e+00} & \num{1.04e+00} & 35 & \num{4.50e+01} \\ 
%$64^3$ & $16^3$ & $42897$ & \num{1.59e+03} & \num{4.60e+00} & \num{1.79e+00} & 32 & \num{6.72e+01} \\ 
% randID
%
\rowcolor{Gray}
$32^3$ & $4^3$ & $527$ & \num{3.45e+00} & \num{2.24e-01} & \num{1.22e-01} & 33 & \num{4.83e+00} \\ 
$32^3$ & $8^3$ & $1358$ & \num{1.01e+01} & \num{2.91e-01} & \num{1.53e-01} & 32 & \num{5.19e+00} \\
$32^3$ & $16^3$ & $4120$ & \num{6.17e+01} & \num{1.65e+00} & \num{1.02e+00} & 28 & \num{2.88e+01} \\ 
%
%\midrule
%
%$64^3$ & $4^3$ & $3990$ & \num{3.02e+01} & \num{7.55e+00} & \num{3.38e+00} & 34 & \num{1.17e+02} \\
%$64^3$ & $8^3$ & $9639$ & \num{2.81e+01} & \num{2.15e+00} & \num{1.35e+00} & 35 & \num{4.90e+01} \\
%$64^3$ & $16^3$ & $30766$ & \num{1.10e+02} & \num{3.19e+00} & \num{1.94e+00} & 33 & \num{6.55e+01} \\ 
    \bottomrule
    \end{tabular}
    \caption{\em CBD preconditioner based on the overlapping domain decomposition \cref{e:second} (eight subdomains).}
    \label{t:cbd_strong_3d}
    \end{subtable}
\end{table}

\begin{table}
    \caption{\em 3D results, where the problem size per partition, i.e., $N/M$, is fixed. The matrix is defined in \cref{e:a_3d}.}
    \label{t:weak_3d}
    \begin{subtable}[h]{\textwidth}
    \centering %\small 
    \begin{tabular}{cccccccc} \toprule
        $N$ & $M$  & $S$ & $t_f$ & $m_f$ & $t_s$ & $n_{it}$ & $t_{pcg}$  \\ \midrule 
%$32^3$ & $8^3$ & 5981 & \num{1.33e+02} & \num{3.77e-01} & \num{1.45e-01} & 5 & \num{1.96e+00} \\ 
%$64^3$ & $16^3$ & 24088 & \num{1.29e+04} & \num{7.43e+00} & \num{2.38e+00} & 5 & \num{1.36e+01} \\ 
%
% randID
$32^3$ & $4^3$ & 5987 & \num{1.33e+01} & \num{3.97e-01} & \num{1.54e-01} & 5 & \num{8.55e-01} \\ 
%$64^3$ & $8^3$ & 24084 & \num{3.16e+02} & \num{7.60e+00} & \num{2.20e+00} & 6 & \num{1.41e+01} \\ 
$64^3$ & $8^3$ & 24081 & \num{4.32e+02} & \num{7.60e+00} & \num{2.57e+00} & 6 & \num{1.60e+01} \\
%$128^3$ & $16^3$ & 93343 & \num{1.29e+04} & \num{1.35e+02} & \num{5.39e+01} & 9 & \num{5.51e+02} \\ 
\rowcolor{Gray}
$128^3$ & $16^3$ & 97073 & \num{1.28e+04} & \num{1.40e+02} & \num{5.63e+01} & 7 & \num{3.99e+02} \\ 
$256^3$ & - & - & - & - & - & - & -  \\
    \bottomrule
    \end{tabular}
    \caption{\em RS preconditioner based on the domain partitioning \cref{e:dd}.}
    \label{t:rs_weak_3d}    
    \end{subtable}
    \begin{subtable}[h]{\textwidth}
    \centering %\small 
    \begin{tabular}{cccccccc} \toprule
    $N$ & $M$  & $S$ & $t_f$ & $m_f$ & $t_s$ & $n_{it}$ & $t_{pcg}$  \\ \midrule 
%$32^3$ & $4^3$ & $4227$ & \num{1.53e+01} & \num{2.24e-01} & \num{1.48e-01} & 33 & \num{4.79e+00} \\ 
%$64^3$ & $8^3$ & $9682$ & \num{1.98e+02} & \num{2.16e+00} & \num{1.12e+00} & 35 & \num{4.99e+01} \\ 
%$128^3$ & $16^3$ & $31402$ & \num{2.30e+03} & \num{1.95e+01} & \num{1.14e+01} & 35 & \num{4.60e+02} \\ 
%$256^3$ & $32^3$ & $114047$ & \num{5.60e+04} & \num{1.77e+02} & \num{2.36e+02} & 35 & \num{9.22e+03} \\ 
% randID
$32^3$ & $4^3$ & $527$ & \num{3.31e+00} & \num{2.24e-01} & \num{1.23e-01} & 33 & \num{4.44e+00} \\ 
$64^3$ & $8^3$ & $1205$ & \num{2.68e+01} & \num{2.15e+00} & \num{1.25e+00} & 35 & \num{4.54e+01} \\
\rowcolor{Gray}
$128^3$ & $16^3$ & $3560$ & \num{2.60e+02} & \num{1.92e+01} & \num{1.13e+01} & 35 & \num{4.04e+02} \\
$256^3$ & $32^3$ & $10915$ & \num{2.62e+03} & \num{1.65e+02} & \num{1.90e+02} & 35 & \num{6.77e+03} \\
    \bottomrule
    \end{tabular}
    \caption{\em CBD preconditioner based on the overlapping domain decomposition \cref{e:second} (eight subdomains).}
    \label{t:cbd_weak_3d}
    \end{subtable}
\end{table}

\begin{remark}[Interface size]
{It seem intuitive that the smallest eigenvalue will increase when the interface size becomes larger. Numerical experiments on our regular problems support this; see \cref{t:d}. However, the benefit of using large interface size, which leads to smaller number of iterations, is usually tiny compared to the extra computational cost for constructing the preconditioner. For results of our new preconditioner reported in the paper, the interface size is $d=1$, i.e., every extended partition contains only one extra layer of grid points (see \cref{sf:extension}). In practice, this choice outperforms larger values of $d$, especially for 3D problems. The reason is that the  computation of the ID decomposition in \cref{e:id},  a significant portion of the total running time, increases rapidly with $d$, whereas the reduction of iteration numbers is very small.}

\begin{table}
    \caption{\em Interface size $d$ and the smallest eigenvalue $\lambda_{\min}$ of the preconditioned matrix in 2D and in 3D.}
    \label{t:d}
    \centering
    \begin{tabular}{cc|cc} 
    \toprule
      \multicolumn{2}{c|}{$N=64^2, D = 4^2$}  & \multicolumn{2}{c}{$N = 32^3, D = 4^3$}  \\
      $d$ & $\lambda_{\min}$ & $d$ & $\lambda_{\min}$  \\ \midrule 
	1 & 0.8685 & 1 & 0.9094 \\
	2 & 0.9045 & 2 & 0.9624 \\
	3 & 0.9183 & 3 & 0.9855 \\ 
    \bottomrule
    \end{tabular}
\end{table}

\end{remark}

\section{Conclusions} \label{s:end}

%We apply two preconditioners based on overlapping domain decomposition for solving integral equations. The first  is the additive Schwarz preconditioner, where the number of subdomains $M=\bigO(N)$. Similar to the block Jacobi preconditioner, the construction, the storage and the application of the Schwarz precondition all scale as $\bigO(N)$. The fundamental difference compared with the block Jacobi preconditioner is that the minimum eigenvalue of the preconditioned matrix is lower bounded, independent of the problem size $N$. %This leads to much less number of iterations with the Schwarz preconditioner. Our numerical results show that this preconditioner can be quite competitive in solving problems in 3D.

We introduce the CBD for constructing an overlapping domain decomposition, where  the number of subdomains is independent of the problem size, and every subdomain consists of spatially distant regions. The associated CBD (single-level additive Schwarz) preconditioner led to $\bigO(1)$ PCG iterations for solving SPD linear systems, which arises from the discretization of first-kind Fredholm integral equations for Laplace's equation in 2D and 3D. We apply the RS factorization to construct the CBD preconditioner efficiently, and our numerical results show that the new preconditioner is much more efficient than the original RS method, especially for solving problems in 3D.
%We introduce the CBD preconditioner, which involves a fixed number of subdomains, independent of the problem size $N$. To construct the preconditioner, we partition the problem domain into $M$ parts with overlapping and color the partitions such that neighbors have different colors. Finally, the partitions with the same color form one subdomain in the CBD preconditioner. The number of subdomains is at most four for problems in 2D and usually small for problems in 3D. The distinct feature of the resulting preconditioner (without low-rank compression) is that the condition number of the preconditioned matrix is $\bigO(1)$. For fast construction of the preconditioner, we apply the RS scheme to subproblems associated with every subdomain. Because of the separation among partitions in a subdomain, our approach is demonstrated to be more efficient than  applying the RS solver to the global problem directly, especially for solving problems in 3D. The PCG iteration observed was almost constant. Furthermore, we derive an asymptotically faster scheme than the RS algorithm by increasing the number of partitions $M$ slower than $\bigO(N)$.
Two future research directions include
\begin{itemize}
%\item
%Theoretical understanding of the minimum eigenvalues of the Schwarz preconditioner.
%\item
%Exploring fast algorithms such as the randomized scheme~\cite{halko2011finding} for computing the ID, which is currently the computation bottleneck in the RS algorithm.
\item
Incorporating other FDSs for subproblems in the CBD preconditioner. Some candidates include the HIF~\cite{ho2016hierarchical}, the RS-S~\cite{minden2017recursive,sushnikova2022fmm}, and the IFMM~\cite{ambikasaran2014inverse,coulier2017inverse,takahashi2020parallelization}, all of which  have been shown to achieve  quasi-linear complexities (under mild assumptions) for solving problems in both 2D and  3D.

%\item
%Parallelizing the construction and the application of the CBD preconditioner. As \cref{e:alu} shows, the ``compress-then-eliminate'' step employed throughout the construction is a local operation 

% parallel algorithms for the CBD preconditioner. For the CBD preconditioner, an efficient parallel algorithm is more challenging because it has to parallelize the RS algorithm.

\item
Extending our approach to solving SPD linear systems arising from machine learning and data science, where the coefficient matrix (a.k.a., kernel matrix) is generated by positive definite functions~\cite{wendland2004scattered}. When the underlying data set lies in a low-dimensional space~\cite{wang2021pbbfmm3d} or even in a high-dimensional space for some applications~\cite{chen2021fast}, the coefficient matrix also has the hierarchical low-rank structure we exploit.

%. Although we focus on solving discretized integral equations, our approach can be extended to general kernel matrices (with minor modification). 
%,  Although 
%, which requires two levels of parallelism. One is parallelism among all partitions (and the associated coarse partitions) in a subdomain. The other is parallelism within a partition when the number of partitions becomes small near the top of the hierarchy.
%improved parallel efficiency. E.g., all clusters at the leaf level of a subdomain can be eliminated in parallel as RS, while obtaining the efficiency of RS-S.
%\item
%Since the eigenvalues of the preconditioned matrices tend to cluster as shown in \cref{f:schwarz_eigenvalue}, a low-rank correction scheme as introduced in \cite{borges2019domain} may be used to improve  convergence.
\end{itemize}

%\section*{Acknowledgements}
%C.C. would also like to thank G. Martinsson, Y. Bao, B. Wu, and X. Xin for helpful discussion.

%\input{tmp.tex}

\appendix

\section{Proof of \cref{th:jacobi}} \label{s:app}

\begin{proof}
Write $T^{-1}_{\text{Jacobi}} A = P_1+P_2$, where $P_i$  is defined in \cref{e:pi}. Rewrite \cref{e:eig} as	
\begin{align*}
(P_1+P_2) (x_1+x_2) &= \lambda (x_1+x_2) \\
\stackrel{\cref{l:subdomain}}{\longrightarrow} \quad P_1 x_2 + P_2 x_1 &= (\lambda-1) (x_1+x_2). 
\end{align*}
Applying projections $R_1 R_1^\top$ and $R_2 R_2^\top$ on both sides, we obatin
\begin{equation} \label{e:px}
\left\{
\begin{array}{c}
P_1 x_2 = (\lambda-1) x_1, \\
P_2 x_1 = (\lambda-1) x_2.
\end{array}
\right.
\end{equation}

Therefore, we know that 
\begin{align*}
(P_1+P_2) (x_1-x_2) 
& \stackrel{\cref{l:subdomain}}{=} (x_1-x_2) + P_2 x_1 - P_1 x_2   \\
& \stackrel{\cref{e:px}}{=} (x_1-x_2) - (1-\lambda) x_2  + (1-\lambda) x_1  \\
&= (2-\lambda) (x_1-x_2).
\end{align*}
\end{proof}
%As a result, we know $\lambda_{\max} + \lambda_{\min} = 2$. See also numerical results in  \cref{t:jacobi_1d}.

%Consider the continuous limit of \cref{e:px} as $N \to \infty$. Since the two subdomains $\Omega_1$ and $\Omega_2$ share $1/2$ as a common boundary, the maximum eigenvector of $T^{-1}_{\text{Jacobi}} A$ corresponds to putting a monopole/point charge at $1/2$, and the minimum eigenvector corresponds to putting a dipole  
%%(two infinitely close point charges with the same amount of positive and negative charges, respectively) 
%there.  For the maximum eigenvalue, both $x_1$ and $x_2$ converges to the indicator function associated with $1/2$  in the continuous limit. In other words, we have, e.g., $P_1^{\infty} x_2^\infty =  x_1^\infty = (\lambda_{\max}^{\infty} -1) x_1^\infty$, where the superscript denotes the continuous limit and $x_2^\infty =  x_1^\infty$ is the indicator function. This implies that
%\[
%\lambda_{\max} \to \lambda_{\max}^{\infty} = 2.
%\]
%For the minimum eigenvalue, we have $x_2^\infty =  - x_1^\infty$ being the indicator function associated with $1/2$. So we have $P_1^{\infty} x_2^\infty =   - x_1^\infty = (\lambda_{\min}^{\infty} -1) x_1^\infty$, which implies that
%\[
%\lambda_{\min} \to \lambda_{\min}^{\infty} = 0.
%\]
%See \cref{f:jacobi_evec} (a) and (b) for an example of the eigenvectors corresponding to the maximum and the minimum eigenvalue, respectively. See also \cref{t:jacobi_1d} for numerical results of the maximum and the minimum eigenvalues. In fact, the eigenvalues of $T^{-1}_{\text{Jacobi}} A$ come in pairs: 

\section{{Numerical results for Schwarz preconditioner}} \label{result:schwarz}

\begin{table}
%    \centering \small 
%    \caption{\em Schwarz preconditioner  in 2D: $N$ fixed.}
%    \label{t:schwarz_strong_2d}
%    \begin{tabular}{ccccccc} \toprule
%        $N$   &  $M$ & $t_f$ & $m_f$ & $t_s$ & $n_{it}$ & $t_{pcg}$  \\ \midrule 
%        $512^2$ & $4^2$ & \num{4.98e+02} & \num{1.80e+01} & \num{7.27e+00} & 29 & \num{2.26e+02} \\ 
%        $512^2$  & $8^2$ & \num{5.72e+01} & \num{4.78e+00} & \num{2.02e+00} & 34 & \num{8.55e+01} \\ 
%        $512^2$ & $16^2$ & \num{8.68e+00} & \num{1.35e+00} & \num{6.59e-01} & 45 & \num{3.81e+01} \\ 
%        $512^2$ & $32^2$ & \num{2.83e+00} & \num{4.24e-01} & \num{2.10e-01} & 71 & \num{2.17e+01} \\ 
%        $512^2$ & $64^2$ & \num{1.65e+01} & \num{1.62e-01} & \num{2.22e-01} & 116 & \num{4.56e+01} \\ 
%    \bottomrule
%    \end{tabular}
    \centering \small 
    \caption{\em Schwarz preconditioner in 2D: $N/M$ fixed. The matrix is defined in \cref{e:a_2d}. (Compare to \cref{t:weak_2d}.) $^*$PCG stagnated at relative residuals \num{1.31e-12} and \num{1.47e-12}, respectively.}
    \label{t:schwarz_weak_2d}
    \begin{tabular}{ccccccc} \toprule
        $N$   &  $M$ & $t_f$ & $m_f$ & $t_s$ & $n_{it}$ & $t_{pcg}$  \\ \midrule 
        $512^2$ & $16^2$ & \num{9.53e+00} & \num{1.35e+00} & \num{6.18e-01} & 46 & \num{4.14e+01} \\
        $1024^2$ & $32^2$ & \num{3.53e+01} & \num{5.43e+00} & \num{3.86e+00} & 71 & \num{3.33e+02} \\ 
        $2048^2$ & $64^2$ & \num{1.41e+02} & \num{2.18e+01} & \num{1.81e+01} &$119^*$& \num{2.40e+03} \\ 
        $4096^2$ & $128^2$ & \num{5.63e+02} & \num{8.74e+01} & \num{7.72e+01} &$212^*$& \num{1.89e+04} \\ 
    \bottomrule
    \end{tabular}
\end{table}

\begin{table}
%    \centering \small 
%    \caption{\em Schwarz preconditioner in 3D: $N$ fixed.}
%    \label{t:schwarz_strong_3d}
%    \begin{tabular}{ccccccc} \toprule
%        $N$   &  $M$ & $t_f$ & $m_f$ & $t_s$ & $n_{it}$ & $t_{pcg}$  \\ \midrule 
%        $64^3$ & $2^3$ & \num{2.03e+03} & \num{4.13e+01} & \num{2.61e+01} & 29 & \num{1.27e+03} \\ 
%        $64^3$ & $4^3$ & \num{8.43e+01} & \num{7.37e+00} & \num{3.06e+00} & 38 & \num{1.29e+02} \\ 
%        $64^3$ & $8^3$ & \num{9.47e+00} & \num{1.77e+00} & \num{8.95e-01} & 51 & \num{5.83e+01} \\ 
%        $64^3$ & $16^3$ & \num{3.88e+00} & \num{6.80e-01} & \num{4.57e-01} & 88 & \num{6.25e+01} \\ 
%    \bottomrule
%    \end{tabular}
    \centering \small 
    \caption{\em Schwarz preconditioner in 3D: $N/M$ fixed. The matrix is defined in \cref{e:a_3d}. (Compare to \cref{t:weak_3d}.)}
    \label{t:schwarz_weak_3d}
    \begin{tabular}{ccccccc} \toprule
        $N$   &  $M$ & $t_f$ & $m_f$ & $t_s$ & $n_{it}$ & $t_{pcg}$  \\ \midrule 
        $32^3$ & $4^3$ & \num{1.03e+00} & \num{1.90e-01} & \num{8.99e-02} & 39 & \num{4.22e+00} \\ 
        $64^3$ & $8^3$ & \num{8.98e+00} & \num{1.77e+00} & \num{8.50e-01} & 52 & \num{5.72e+01} \\ 
        $128^3$ & $16^3$ & \num{7.76e+01} & \num{1.52e+01} & \num{8.79e+00} & 82 & \num{8.84e+02} \\ 
        $256^3$ & $32^3$ & \num{7.48e+02} & \num{1.26e+02} & \num{1.36e+02} & 156 & \num{2.62e+04} \\ 
    \bottomrule
    \end{tabular}
\end{table}

%\cref{t:schwarz_strong_2d,t:schwarz_strong_3d} show that when $N$ is fixed, the storage $m_f$ and the application time of the preconditioner $t_s$ scaled approximately as $\bigO(M^{-1})$; see \cref{e:mem}. The factorization time $t_f$ apparently decreased slower than $\bigO(M^{-2})$ as \cref{e:fact} predicts, and that is due to the dominance of the $\bigO(M^{-1})$ cost for evaluating diagonal blocks in $A$ associated with all subdomains. 

In this section, we show numerical results for the Schwarz preconditioner, where the number of partitions $M=\bigO(N)$. In \cref{t:schwarz_weak_2d,t:schwarz_weak_3d}, we chose $N/M=32^2$ and $N/M=8^3$, respectively, which empirically led to best performance.
In this case, the factorization time $t_f$, the storage $m_f$, and the application time of the Schwarz preconditioner $t_s$ all scaled approximately as $\bigO(N)$. However, the number of PCG iterations $n_{it}$ increased rapidly as the problem size increased. As a result, the overall running time of the Schwarz preconditioner exceeded that taken by the CBD preconditioner for large problems: $N \ge 2048^2$ in 2D or $N \ge 128^3$ in 3D (see running time of the CBD preconditioner in \cref{t:weak_2d,t:weak_3d}).

\bibliographystyle{siamplain}
\bibliography{biblio}

\end{document}